\documentclass[reqno]{amsart}
\usepackage{amsfonts,amsmath,amssymb,color}

\numberwithin{equation}{section}

\usepackage[babel=true,kerning=true]{microtype}

\usepackage{amsthm,leqno}
\usepackage{hyperref}
\usepackage{bigints}

\usepackage{amsfonts,amsmath,amssymb,enumerate}
\usepackage{graphicx}
\usepackage{tikz}
\usetikzlibrary{decorations.markings}
\usetikzlibrary{plotmarks}
\usetikzlibrary{patterns}

\newtheorem{theo}{Theorem}[section]
\newtheorem{prop}[theo]{Proposition}

\newtheorem{corol}[theo]{Corollary}
\newtheorem{claim}[theo]{Claim}

\theoremstyle{remark}

\newcommand{\be}{\begin{equation*}}
\newcommand{\ee}{\end{equation*}}
\newcommand{\ben}{\begin{equation}}
\newcommand{\een}{\end{equation}}
\newcommand{\begincal}{\begin{eqnarray*}}
\newcommand{\fincal}{\end{eqnarray*}}
\newcommand{\bal}{\begin{aligned}}
\newcommand{\eal}{\end{aligned}}
\newcommand{\ds}{\displaystyle}

\newcommand{\pui}{\frac{n-2}{2}}

\newcommand{\ua}{u_\alpha}
\newcommand{\va}{v_\alpha}
\newcommand{\tua}{\tilde{u}_\alpha}
\newcommand{\cua}{\check{u}_\alpha}
\newcommand{\tva}{\tilde{v}_\alpha}
\newcommand{\hua}{\hat{u}_\alpha}
\newcommand{\hva}{\hat{v}_\alpha}
\newcommand{\Ba}{B_\alpha}
\newcommand{\Ra}{R_\alpha}
\newcommand{\tBa}{\tilde{B}_\alpha}

\newcommand{\tRa}{\tilde{R}_\alpha}

\newcommand{\cBia}{\check{B}_{i,\alpha}}

\newcommand{\Wa}{W_\alpha}
\newcommand{\tWa}{\tilde{W}_\alpha}
\newcommand{\hWa}{\hat{W}_\alpha}
\newcommand{\cWa}{\check{W}_\alpha}
\newcommand{\Za}{Z_\alpha}
\newcommand{\Va}{V_\alpha}
\newcommand{\Pak}{P_{\alpha,k}}
\newcommand{\NRa}{\Vert \Ra \Vert_{\infty}}
\newcommand{\NtRa}{\Vert \tRa \Vert_{\infty}}

\newcommand{\Lg}{\mathcal{L}_g}
\newcommand{\Lx}{\mathcal{L}_\xi}
\newcommand{\Lh}{\mathcal{L}_h}

\newcommand{\Dg}{\overrightarrow{\triangle}_g }
\newcommand{\Dx}{\overrightarrow{\triangle}_\xi }
\newcommand{\Dh}{\overrightarrow{\triangle}_h }

\newcommand{\Xa}{X_{\alpha}}

\newcommand{\Ya}{Y_{\alpha}}

\newcommand{\xa}{x_\alpha}

\newcommand{\cxia}{\check{x}_{i,\alpha}}

\newcommand{\ya}{y_\alpha}
\newcommand{\za}{z_\alpha}

\newcommand{\ma}{\mu_\alpha}
\newcommand{\cmia}{\check{\mu}_{i,\alpha}}
\newcommand{\na}{\nu_\alpha}
\newcommand{\la}{\lambda_\alpha}
\newcommand{\rra}{r_\alpha}
\newcommand {\da}{\delta_\alpha}
\newcommand{\ta}{\theta_\alpha}
\newcommand{\vea}{\varepsilon_\alpha}
\newcommand{\bak}{\beta_{\alpha,k}}
\newcommand{\ba}{\beta_\alpha}
\newcommand{\ra}{\rho_\alpha}

\newcommand{\vpa}{\varphi_\alpha}
\newcommand{\psa}{\psi_\alpha}

\newcommand{\dda}{d_\alpha}

\newcommand{\RR}{\mathbb{R}}

\newcommand{\ve}{\varepsilon}
\newcommand{\vp}{\varphi}

\numberwithin{equation}{section}

\begin{document}
\title[The Einstein-Lichnerowicz constraints system]
{Stability and instability of the Einstein-Lichnerowicz constraint system}

\author{Bruno Premoselli } 
\address{Bruno Premoselli, Universit\'e de Cergy-Pontoise 
D\'epartement de Math\'ematiques 
2 avenue Adolphe Chauvin 
95302 Cergy-Pontoise cedex 
France}
\email{bruno.premoselli@u-cergy.fr}

\begin{abstract} 

We investigate the relevance of the conformal method by investigating stability issues for the Einstein-Lichnerowicz conformal constraint system in a nonlinear scalar-field setting. We prove the stability of the system with respect to arbitrary perturbations of generic focusing physics data on closed locally conformally flat manifolds, in any dimension. We also show that our stability result is sharp by constructing explicit instability examples when its assumptions are not satisfied. Our results apply to a more general class of constraint-like systems.

\end{abstract}

\maketitle 


\section{Introduction}

The constraint equations arise in General Relativity, in the analysis of the initial-value problem for the Einstein equations. Given a $n$-manifold $M$, $n\ge3$,  and a potential $V$ -- a smooth function in $\mathbb{R}$ -- in a nonlinear scalar-field setting the constraint equations write as follows:
\begin{equation} \label{cont}   \tag{$\mathcal{C}_{-}$}
\left \{
\begin{aligned}
& R(\tilde g) + (\mathrm{tr}_{\tilde g} \tilde K)^2 - \| \tilde K \|_{\tilde g}^2 = {\tilde \pi}^2 + |\nabla \tilde \psi|_{\tilde g}^2 + 2 V(\tilde \psi), \\
& \nabla^j \tilde K_{ij} - \nabla_i (\mathrm{tr}_{\tilde g} \tilde K)  = \tilde \pi \nabla_i \tilde \psi .
\end{aligned}
\right.
\end{equation}
The unknowns of \eqref{cont} are the Riemannian metric $\tilde g$ in $M$, $\tilde K$ a $(2,0)$-symmetric tensor field and $\tilde \psi$ and $\tilde \pi$ two functions in $M$, representing the scalar-field and its future-directed temporal derivative. Also, $\tilde{\nabla}$ is the Levi-Civita connection of $\tilde g$ and $R(\tilde g)$ is its scalar curvature. A solution $(\tilde g, \tilde K, \tilde \psi, \tilde \pi)$ of \eqref{cont} is called an \emph{initial data set}. This terminology has its roots in the celebrated well-posedness results of Choquet-Bruhat \cite{ChoBru} and Choquet-Bruhat-Geroch \cite{ChoGe}, that assert that every initial data set $(\tilde g, \tilde K, \tilde \psi, \tilde \pi)$ possesses a unique maximal globally hyperbolic spacetime development $(\mathcal{M},h,\Psi)$ satisfying the Einstein scalar-field equations. The system \eqref{cont} is therefore of paramount importance in General Relativity since it is a necessary and sufficient condition for the resolution of the Einstein equations. Furthermore, addressing the resolution of \eqref{cont} exactly amounts to determining the initial data set of the evolution.

\medskip

One of the most successful techniques developed to overcome the underdetermination of \eqref{cont} is the conformal method. Initiated by Lichnerowicz \cite{Lich} and improved by Choquet-Bruhat and York \cite{ChoYo}, it aims at turning \eqref{cont} into a determined system by parametrizing the unknowns in terms of prescribed background physics data. We will assume from now on that $M$ is a closed manifold -- that is, compact without boundary -- endowed with a reference Riemannian metric $g$. We let a potential $V$ and $(\psi, \pi, \tau, \sigma)$ be fixed physics data in $M$, where $\psi, \pi, \tau$ are functions and $\sigma $ is a traceless and divergence-free $(2,0)$ symmetric tensor in $M$. We look for solutions of \eqref{cont} under the following parametrization:
\begin{equation} \label{dataparam}
(\tilde g,\tilde K,\tilde \psi, \tilde \pi ) = \bigg( \varphi^{\frac{4}{n-2}}g, \frac{\tau}{n}\varphi^{\frac{4}{n-2}}g + \varphi^{-2}(\sigma + \mathcal{L}_g W), \psi, \varphi^{- \frac{2n}{n-2}} \pi \bigg), 
\end{equation}
where $\vp > 0$ is a positive function in $M$, $W$ is a field of $1$-forms and 
\begin{equation} \label{conflie}
  \mathcal{L}_gW_{ij} = W_{i,j} + W_{j,i} - \frac{2}{n} \left(\mathrm{div}_g W\right) g_{ij}\hskip.1cm
\end{equation}
is the conformal Killing derivative of $W$. As easily checked, $(\tilde g,\tilde K,\tilde \psi, \tilde \pi )$ given by \eqref{dataparam} solves \eqref{cont} if and only if $(\vp,W)$ solves the following \emph{Einstein-Lichnerowicz conformal constraint system of physics data $D = (\psi, \pi, \tau, \sigma)$ and $V$}. Namely:
\begin{equation} \label{cconf} \tag{$C_{D}$} 
\left \{ 
\begin{aligned} 
&    \frac{4(n-1)}{n-2}\triangle_g \varphi + \mathcal{R}_\psi  \varphi  = \mathcal{B}_{\tau, \psi, V} \varphi^{2^*-1} + \frac{\left(  |\sigma + \mathcal{L}_g W |_g^2 + \pi^2 \right)}{ \varphi^{2^*+1}}~,  \\ 
 & \Dg  W  = - \frac{n-1}{n}\varphi^{2^*} \nabla\tau - \pi\nabla \psi~. 
\end{aligned}
\right.
\end{equation}
In \eqref{cconf}, the coefficients express in terms of the physics data  as: 
\begin{equation} \label{coefficients}
\begin{aligned}
& \mathcal{R}_{\psi}  = R(g) - |\nabla \psi|_g^2 ~~,~~\mathcal{B}_{\tau,\psi,V}  =  2 V(\psi) - \frac{n-1}{n} \tau^2 . \\
\end{aligned} \end{equation}
Also, $2^* = \frac{2n}{n-2}$ is the critical exponent for the embedding of the Sobolev space $H^1(M)$ into Lebesgue spaces, $\triangle_g = - \mathrm{div}_g (\nabla \cdot)$ is the Laplace-Beltrami operator and $\Dg$ is the Lam\'e operator acting on $1$-forms:
\[\Dg W = - \mathrm{div}_g (\mathcal{L}_g W).\]
Since $M$ is closed, the operator $\Dg$ is elliptic and self-adjoint on $TM$ and possesses the usual regularizing fetaures of elliptic operators. Fields of $1$-forms $W$ satisfying $\Lg W = 0$ in $M$ will be called conformal Killing $1$-forms, see \eqref{defK} below. System \eqref{cconf} is in particular invariant up to the addition of any conformal Killing $1$-form to $W$. In the following we shall therefore always considers solutions $(\vp,W)$ of \eqref{cconf} \emph{up to conformal Killing $1$-forms}, that is we shall assume that $W$ is $L^2$-orthogonal to the space of such conformal Killing $1$-forms.

\medskip

The conformal method, starting from given physics data, generates initial data sets that consequently provide us with space-time developments. The whole procedure can be summed up in the following $3$-steps construction, that we shall call the Choquet-Bruhat-Geroch-Lichnerowicz (CBGL) formalism:
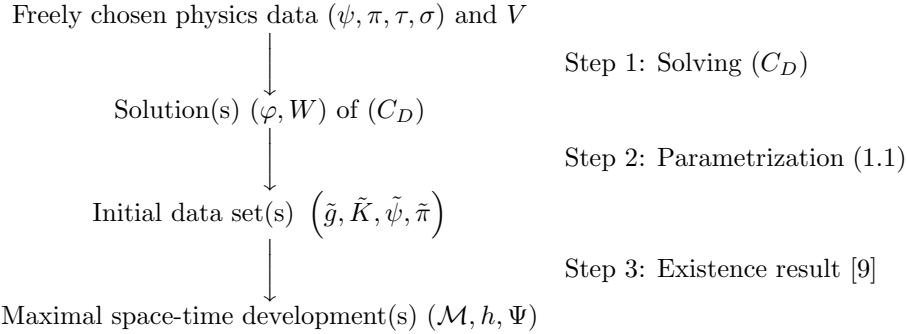
\[ \begin{array}{cl}
\textrm{Freely chosen physics data } (\psi, \pi, \tau, \sigma) \textrm{ and } V & \\  
\bigg{\downarrow}  & \textrm{Step } 1\textrm{: Solving } \eqref{cconf} \\
\textrm{Solution(s) } (\vp,W) \textrm{ of } \eqref{cconf} & \\
 \bigg{\downarrow}  &  \textrm{Step } 2 \textrm{: Parametrization } \eqref{dataparam} \\
\textrm{Initial data set(s) } \left( \tilde g, \tilde K, \tilde \psi, \tilde \pi \right) & \\
 \bigg{\downarrow}  &  \textrm{Step } 3\textrm{: Existence result \cite{ChoGe}}  \\
\textrm{Maximal space-time development(s) } (\mathcal{M}, h,  \Psi) & \\ 
\end{array} \]
Space-times obtained via the latter construction capture the features of the specific initial data sets $( \tilde g, \tilde K, \tilde \psi, \tilde \pi )$ obtained through the conformal method. Investigating the validity of the conformal method therefore reduces to the investigation of the physical relevance of the CBGL formalism. We address it here through the following fundamental question:
\medskip

\noindent \textbf{Question:} \emph{is the CBGL formalism robust with respect to the initial choice of the physics data $(\psi, \pi, \tau, \sigma)$ and $V$ of the conformal method? }

\medskip

More specifically, 
we do not expect minor perturbations of the physics data $(\psi, \pi, \tau, \sigma)$ and $V$ to create dramatic changes in the geometry of the resulting space-times. In the $3$-steps CBGL construction, once an initial data set $(\tilde g, \tilde K, \tilde \psi, \tilde \pi)$ is given, continuity for Step $3$ is ensured by the notion of Cauchy stability (see Ringstr\"om, \cite{Ringstrom}). Clearly, Step $2$ is continuous. Hence, the crucial point lies in the proof of the continuity of Step $1$. Proving the robustness of the CBGL formalism and of the conformal method therefore boils down to proving the stability of system \eqref{cconf}, that is the continuous dependence of the set of solutions of \eqref{cconf} on the choice of the physics data $(\psi, \pi, \tau, \sigma)$ and $V$.  
\medskip

We introduce the following terminology: we shall say that physics data $V$ and $(\psi, \pi, \tau, \sigma)$ are
\ben \label{focusing}
 \textrm{\emph{focusing} if } \mathcal{B}_{\tau,\psi,V} > 0 \textrm{ in } M \quad \textrm{ and } \quad \textrm{ \emph{defocusing} if } \mathcal{B}_{\tau,\psi,V} \le 0 \textrm{ in } M, 
 \een
where $\mathcal{B}_{\tau,\psi,V}$ is as in \eqref{coefficients}. 

\medskip

In the present work we establish the stability of system \eqref{cconf} in strong topologies with respect to arbitrary perturbations of focusing physics data on locally conformally flat manifolds, in any dimension. It is the content of our main result:
\begin{theo} \label{grosth}
Let $(M,g)$ be a closed locally conformally flat Riemannian manifold of dimension $n \ge 3$. Consider a smooth potential $V$ and smooth physics data $D = (\psi, \pi, \tau, \sigma)$. Assume that the data are focusing as in \eqref{focusing} and that $\pi \not \equiv 0$. If $ n \ge 6$, assume in addition that $\tau$ and $\psi$ have no common critical points in $M$. Let $(V_\alpha)_\alpha$ and $(D_\alpha)_\alpha$, $D_\alpha = (\psi_\alpha, \pi_\alpha, \tau_\alpha, \sigma_\alpha)_\alpha$ be sequences of potentials and of physics data converging respectively to $V$ and $D$ in the following topology:
\ben \label{convstabi}
 \Vert V_\alpha-V \Vert_{C^2} +   \Vert \tau_\alpha - \tau \Vert_{C^3} + \Vert \psi_\alpha - \psi \Vert_{C^2}  + \Vert \pi_\alpha - \pi \Vert_{C^0} + \Vert \sigma_\alpha - \sigma \Vert_{C^0}  \underset{\alpha \to +\infty}{\longrightarrow} 0 .\een
Consider $(\vpa, \Wa)_\alpha$, $\vpa > 0$, a sequence of solutions of the Einstein-Lichnerowicz constraints system of physics data $D_\alpha$ and $V_\alpha$:
\begin{equation} \label{cconfa} \tag{$C_{D_\alpha}$} 
\left \{ 
\begin{aligned} 
&   \frac{4(n-1)}{n-2} \triangle_g \vpa + \mathcal{R}_{\psi_\alpha}  \vpa  = \mathcal{B}_{\tau_\alpha, \psi_\alpha, V_\alpha} \vpa^{2^*-1} + \frac{\pi_\alpha^2 + \left| \sigma_\alpha + \Lg W_\alpha \right|_g^2}{ \vpa^{2^*+1}}~,  \\ 
 & \Dg  \Wa  = - \frac{n-1}{n}\vpa^{2^*} \nabla\tau_\alpha - \pi_\alpha \nabla \psi_\alpha~,  
\end{aligned}
\right.
\end{equation}
where the coefficients express as in \eqref{coefficients}.  Then, up to a subsequence and up to conformal Killing $1$-forms, the sequence $(\vpa,\Wa)_\alpha$ converges in $C^{1,\eta}(M)$, for any $0 < \eta < 1$, to some solution $(\vp_0,W_0)$, $\vp_0 > 0$, of the limiting Einstein-Lichnerowicz constraints system of equations \eqref{cconf}.
\end{theo}
Theorem \ref{grosth} 
 is also a compactness result. It establishes in particular that sequences of solutions of \eqref{cconf} for perturbations of given physics data do not blow-up. Although not explicitly stated in Theorem \ref{grosth}, the result still holds true if we also allow perturbations of the geometry of $M$ in the locally conformally flat category in strong topologies. The consequence of Theorem \ref{grosth} in terms of the CBGL formalism is as follows:
\begin{corol} \label{corol}
The CBGL formalism is stable with respect to the choice of generic focusing initial data $(\psi, \pi, \tau, \sigma)$ and $V$ in any locally conformally flat geometry in $M$. 
\end{corol}
The focusing case investigated here and defined in \eqref{focusing} covers the general physical setting where nontrivial nongravitational data are considered. Among the prominent cases allowed one finds for instance the positive cosmological constant case or the nontrivial Klein-Gordon field case. Note that in small dimensions Theorem \ref{grosth} only requires the non-staticness assumption $\pi \not \equiv 0$ while in dimensions $n \ge 6$ stability is ensured by a higher order stationarity condition on the scalar-field $\psi$ and the mean curvature $\tau$. These assumptions are in particular generic on the set of all focusing physics data $ (\psi, \pi, \tau, \sigma)$ and $V$. 

\medskip

Obviously, solutions of \eqref{cconf} exist, at least in some cases. Partial existence results for the vacuum defocusing case were obtained in Isenberg \cite{IseCMC},  Holst-Nagy-Tsogtgerel \cite{HoNaTso}, Maxwell \cite{Maxwell} and Dahl-Gicquaud-Humbert \cite{DaGiHu}. Only recently the more involved focusing case, which is the case of interest in Theorem \ref{grosth}, has been addressed: we refer to Hebey-Pacard-Pollack \cite{HePaPo}, Premoselli \cite{Premoselli1, Premoselli2} and Holst-Meier \cite{HolstMeier}. Stability results for system \eqref{cconf} had been obtained in specific cases, see Druet-Hebey \cite{DruHeb}, Premoselli \cite{Premoselli2} and Druet-Premoselli \cite{DruetPremoselli}.
Note also that, unlike all the known existence results, Theorem \ref{grosth} requires no smallness assumptions on the physics data and allows, in full generality, the manifold $M$ to possess non-trivial conformal Killing $1$-forms in $M$. As a remark, for geometric nonlinear critical elliptic equations, low-regularity  perturbations of the metric and/or of the coefficients may lead to the existence of blowing-up sequences of solutions. See for instance Berti-Malchiodi \cite{BertiMalchiodi}, Druet-Laurain \cite{DruetLaurain} and Druet-Hebey-Laurain \cite{DruetHebeyLaurain} for examples of such phenomena for nonlinear stationary Schr\"{o}dinger equations. In dimensions $3$, the convergence of $(V_\alpha)_\alpha$ and $(D_\alpha)_\alpha$ can be lowered from $C^2$ to $C^1$.

\medskip

Note also that, since system \eqref{cconf} is elliptic, as the regularity of the convergence of $(D_\alpha)_\alpha$ to $D$ and of $(V_\alpha)_\alpha$ to $V$ increases, the regularity of convergence of the solutions increases accordingly. For physics data converging in $C^\infty(M)$ the convergence of solutions in Theorem \ref{grosth} holds in $C^\infty(M)$ too.

\medskip

Let us point out one important fact. We are investigating here a specific
notion of elliptic stability (see Hebey \cite{HebeyZLAM} for a reference
in book form), which turns out to be particularly well-suited for the conformal
method and the CBGL construction. As already noticed, in our setting, the
notion precisely measures the robustness of the CBGL construction with respect
to the choice of the upstream physics data of the problem. There are other very
natural (and more historical) notions of stability which arise when
investigating the stability of specific solutions of the Einstein equations.
There is a huge 
litterature in this deep direction. Without pretending to be exhaustive, we
mention the groundbreaking global nonlinear stability of the Minkowski
space-time by Christodoulou-Klainerman \cite{ChristodoulouKlainerman} (see also
Klainerman-Nicol\`o \cite{KlainermanNicolo} or Lindblad-Rodnianski
\cite{LindbladRodnianski}) and refer to Dafermos-Rodnianski
\cite{DafermosRodnianski} and the references therein for a detailed account of
the existing work on the problem of the stability of black holes.

\bigskip 

The proof of Theorem \ref{grosth} goes through the proof of a stability result for a general class of constraint-like systems, see \eqref{systype} and Theorem \ref{Th1} in section \ref{PDEframework} below. 
The proof proceeds by contradiction and requires an involved asymptotic analysis of blowing-up sequences of solutions of \eqref{systype} around concentration points. In Section \ref{preuvethstabi} we isolate the regions of $M$ where loss of compactness may occur and we show that concentration points do not actually appear in $M$.  The proof of this fact -- that concludes the proof of Theorem \ref{grosth} -- crucially relies on an accurate pointwise asymptotic description of the behavior of blowing-up sequences of solutions around concentration points. Such an asymptotic description is by far the core of the analysis in this paper and, for the sake of clarity, is postponed to Section \ref{analyseasymptotique}. Unlike in the fully decoupled case (where $\nabla \tau \equiv 0$), treated in Druet-Hebey \cite{DruHeb} and Premoselli \cite{Premoselli2}, obtaining a sharp pointwise description of blowing-up sequences of \eqref{cconf} in full generality as we do here requires an involved simultaneous analysis of the defects of compactness that occur in each of the two equations of \eqref{cconf} and of their interactions. We perform this analysis in Section \ref{analyseasymptotique}. Finally, section \ref{instabilite} is aimed at showing that the assumptions of Theorem \ref{Th1} below are sharp and is concerned with the construction of blowing-up sequences of solutions of constraint-like systems of equations, and sections \ref{HI} to \ref{Greenfunc} gather some technical results used throughout the paper.

\section{A PDE framework} \label{PDEframework}

We investigate the Einstein-Lichnerowicz conformal constraint system \eqref{cconf} in an elliptic PDE framework and prove a general stability result for an extended class of constraint-like systems. We consider $(M,g)$ a closed locally conformally flat manifold of dimension $n \ge 3$ and consider a sequence $(u_\alpha, W_\alpha)_\alpha$ of solutions of:
\begin{equation} \label{systype}
\left \{ \begin{aligned}
& \triangle_g \ua + h_\alpha \ua = f_\alpha \ua^{2^*-1} + \frac{a_\alpha(\Wa)}{\ua^{2^*+1}}, \\
& \Dg \Wa = \ua^{2^*} \Xa + \Ya,
\end{aligned} \right. 
\end{equation}
with $\ua > 0$ in $M$, where 
\[a_\alpha(\Wa) = b_\alpha + \left| U_\alpha + \Lg \Wa \right|_g^2,\]
$h_\alpha, f_\alpha, b_\alpha$ are smooth functions in $M$, $\Xa, \Ya$ are smooth $1$-forms in $M$ and $U_\alpha$ is a smooth symmetric  
$(2,0)$ tensor field in $M$. Here, as before, $2^* = \frac{2n}{n-2}$, $\triangle_g = - \textrm{div}_g (\nabla \cdot)$,
\[ \Lg W_{ij} = \nabla_i W_j + \nabla_j W_i - \frac{2}{n} (\textrm{div}_g W) g_{ij} \]
and $\Dg = - \textrm{div}_g \left( \Lg \cdot \right)$. We assume that there holds
\ben \label{convdonnees}
\left( h_\alpha, f_\alpha, b_\alpha, U_\alpha, \Xa, \Ya \right)_\alpha \to (h_0, f_0, b_0, U_0, X_0, Y_0) 
\een
as $\alpha \to \infty$, where all the convergences take place in $C^0(M)$ except the convergence of $(f_\alpha)_\alpha$ and $(\Xa)_\alpha$ to $f_0$ and $X_0$ which take place in $C^2(M)$. We assume in addition that $\triangle_g + h_0$ is coercive and that $f_0 > 0$.

\medskip

We let $K_g$ denote the set of conformal Killing $1$-forms in $M$:
\ben \label{defK}
K_g = \{ W \in H^1(M) \textrm{ st } \Lg W = 0 \},
\een
where $H^1(M)$ denotes the usual Sobolev space on $1$-forms. Since the only quantity depending on $\Wa$ that appears in system \eqref{systype} is $\Lg \Wa$ the system is invariant under the addition to $\Wa$ of elements of $K_g$. In the following we may therefore consider sequences of solutions $(\ua, \Wa)_\alpha$ \emph{up to conformal Killing $1$-forms}, that is we may assume that $\Wa$ is orthogonal to $K_g$ for the $L^2$-scalar product.

\medskip

Our main result is an \emph{a priori} boundedness result:

\begin{theo} \label{Th1}
Let $( h_\alpha, f_\alpha, b_\alpha, U_\alpha, \Xa, \Ya )_\alpha$ be a sequence of coefficients, converging to some limiting coefficients $(h_0, f_0, b_0, U_0, X_0, Y_0)$, and satisfying the convergence conditions of \eqref{convdonnees}. Assume that $\triangle_g + h_0$ is coercive and that $f_0 > 0$ in $M$. If $3 \le n \le 5$, assume that either $b_0 \not \equiv 0$ or $|X_0|_g > 0$ in $M$. If $n \ge 6$, assume that $X_0$ and $\nabla f_0$ have no common zero or, if they do, assume that there holds at these zeroes:
\ben \label{condzeros}
 h_0 <  \frac{n-2}{4(n-1)} R(g) - C(n) f_0 ^{-1}\triangle_g f_0 , 
 \een
where $C(n)$ denotes some positive constant only depending on $n$ explicitly given in \eqref{defC(n)} below, and $R(g)$ is the scalar curvature of $g$. Let $(\ua, \Wa)_\alpha$ be a sequence of solutions of \eqref{systype}. Let $0 < \eta < 1$. There exists then a positive constant $C(\eta)$ that does not depend on $\alpha$ such that, up to a subsequence and up to conformal Killing $1$-forms, 
\[ \Vert \ua \Vert_{L^\infty(M)} + \Vert \Wa \Vert_{C^{1,\eta}(M)} \le C(\eta), \]
for all $\alpha$. 
\end{theo}
The limiting system associated to \eqref{systype}--\eqref{convdonnees} is:
\ben \label{systlimite}
\left \{ 
\bal
& \triangle_g u_0 + h_0 u_0 = f_0 u_0^{2^*-1} + \left( b_0 + \left| U_0 + \Lg W_0\right|_g^2\right)u_0^{-2^*-1}, \\
& \Dg W_0 = u_0^{2^*} X_0 + Y_0.
\eal
\right. 
\een
As a consequence of Theorem \ref{Th1} we have the following compactness result:
 \begin{corol} \label{corol1}
Under the assumptions of Theorem \ref{Th1} there holds:
\begin{enumerate}
\item either $\Vert \ua \Vert_{L^\infty(M)} \to 0 $ as $\alpha \to \infty$, in which case $b_0 \equiv 0$, $\textrm{div}_g U_0 = Y_0$ in $M$ and, up to a subsequence and up to conformal Killing $1$-forms, $\Wa \to W_0$ in $C^{1, \eta}(M)$ for any $0 < \eta < 1$, where $\Dg W_0 = Y_0$ in $M$, 
\item or $\Vert \ua \Vert_{L^\infty(M)} \not \to 0 $ as $\alpha \to \infty$. In this case, up to a subsequence and up to conformal Killing $1$-forms, $(\ua, \Wa) \to (u_0, W_0)$ in $C^{1,\eta}(M)$ for any $0 < \eta < 1$ with $u_0 > 0$ in $M$, where $(u_0,W_0)$ is a solution of the limiting system \eqref{systlimite}.
\end{enumerate}
\end{corol}
In both cases $(\ua, \Wa)$ converges to a solution of the limiting system, even if it is somewhat degenerate in the first case.

\begin{proof}
Assume that the first alternative holds. By standard elliptic theory, up to conformal Killing $1$-forms, $\Wa \to W_0$ in $C^{1,\eta}(M)$ where $\Dg W_0 = Y_0$. Let $G_\alpha$ be the Green function of the operator $\triangle_g + h_\alpha$ in $M$. It is uniformly positive thanks to \eqref{convdonnees} and since $\triangle_g + h_0$ is coercive (see Robert \cite{RobDirichlet}). A Green's formula on the scalar equation gives, for some positive constant $C$:
\[ \Vert \ua \Vert_{L^\infty(M)}^{2^*+2} \ge \frac{1}{C} \int_M \left( b_\alpha + \left| U_\alpha + \Lg \Wa \right|_g^2 \right)dv_g\]
which shows that $b_0 \equiv 0$ and $U_\alpha + \Lg \Wa \to 0$ in $L^2(M)$. Passing to the limit we obtain $\textrm{div}_g U_0 = Y_0$.

\medskip

In the second case, if $\Vert \ua \Vert_{L^\infty(M)} \not \to 0$, the Harnack inequality as stated in Section \ref{HI} shows that $\inf_M \ua \ge \frac{1}{C}$ for some positive constant $C$. The $C^{1,\eta}(M)$-bounds on $\ua$ and $\Wa$ therefore follow by standard elliptic theory in each equation.
\end{proof}
Note that if the assumptions of Theorem \ref{Th1} are satisfied the second alternative in Corollary \ref{corol1} holds in dimensions $3 \le n \le 5$.

\medskip

Compactness and stability results for elliptic PDEs have a long time history. A major result was the complete proof of the compactness of the Yamabe equation by Khuri-Marques-Schoen \cite{KhuMaSc} together with its dimensional limitation by Brendle \cite{Bre} and Brendle-Marques \cite{BreMa}. Note that an equation can be compact but unstable, that is, sensitive to changes of the parameters in the equation.  
General references on the stability of elliptic PDEs are by Druet \cite{DruetENSAIOS}, Druet-Hebey-Robert \cite{DruetHebeyRobert} and Hebey \cite{HebeyZLAM}. The specific case of the Einstein-Lichnerowicz equation is addressed in Druet-Hebey \cite{DruHeb}, Hebey-Veronelli \cite{HebeyVeronelli} and Premoselli \cite{Premoselli2}. 

\medskip

Theorem \ref{Th1} leaves open the question of stability when its assumptions are not satisfied. The reason for this is that they are sharp. As shown in Section \ref{instabilite} below, whenever the assumptions of Theorem \ref{Th1} are not satisfied we can construct blowing-up sequences of solutions of \eqref{systype}:

\begin{theo} \label{thinstabEL}
For any $n \ge 3$, when we contradict the assumptions $b_0 \not \equiv 0$ and \eqref{condzeros}, there exist examples of  smooth closed Riemannian $n$-manifolds $(M,g)$ and of sequences of coefficients $\left( h_\alpha, f_\alpha, b_\alpha, U_\alpha, \Xa, \Ya \right)_\alpha$ converging to some limiting coefficients $(h_0, f_0, b_0, U_0, X_0, Y_0)$ as in \eqref{convdonnees} with $U_0 \not \equiv 0$ and $X_0 \not \equiv 0$, and there exist sequences $(\ua, \Wa)_\alpha$ of solutions of \eqref{systype} such that $\sup_M \ua \to + \infty \textrm{ as } \alpha \to + \infty.$
\end{theo}
\noindent The precise statements gathered in Theorem \ref{thinstabEL} are given in Section \ref{instabilite}.

\medskip

Theorem \ref{Th1} highlights a dimensional hiatus in the behavior of \eqref{systype}. In dimensions $3 \le n \le 5$, the assumption $b_0 \not \equiv 0$ ensures that any sequence $(u_\alpha)_\alpha$ of solutions blows up with a (hypothetic) nonzero limit profile (compare with Proposition \ref{instabn3} where blowing-up sequences with zero limit profiles are constructed). Such an a priori property is crucially used in the blow-up analysis to control the radius of extension of the defects of compactness of the sequence $(\ua, \Wa)_\alpha$. Conversely, when $n \ge 6$ or $X_0$ never vanishes, estimations of the extension radius are directly obtained in the analysis. Also, Theorem \ref{Th1} requires $(M,g)$ to be locally conformally flat. This assumption is crucial in our approach at this stage to get uniform estimates on the Green $1$-forms of the operator $\Dg$ with Neumann boundary conditions on small balls. If $g$ is not locally conformally flat, the Kernel of $\Lg$ on small balls degenerates and this leads to a possible loss of compactness which is not quantifiable with the existing blow-up techniques.

\section{Proof of Theorem \ref{Th1}} \label{preuvethstabi}

We let $(M,g)$ be a closed locally conformally flat manifold of dimension $ n\ge 3$ and we let $(\ua, \Wa)_\alpha$ be a sequence of solutions of \eqref{systype} on $(M,g)$ with $\ua > 0$ in $M$. Assume that \eqref{convdonnees} holds. To prove Theorem \ref{Th1} it is enough to show that there exists some positive constant $C$ such that
\[ \Vert \ua \Vert_{L^\infty(M)} \le C .\]
Indeed, since the system \eqref{systype} is invariant up to adding to $\Wa$ elements of \eqref{defK}, if we assume $\Wa$ to be orthogonal to $K_g$ as in \eqref{defK} there holds, by standard elliptic theory for the second equation (see Section \ref{stddelltheory}), that $\Vert \Wa \Vert_{C^{1,\eta}(M)} \le C'$ for all $0 < \eta < 1$, for some positive constant $C'$. We proceed by contradiction and assume thus that there holds:
\ben \label{contradiction}
\sup_M \ua \to +\infty
\een
as $\alpha \to \infty$. In all of this section and in the remaining of the paper the letter $C$ will always denote some positive constant, whose value may change from one line to another, but that will never depend on $\alpha$.

\begin{prop} \label{proppointconc}
Let $(\ua, \Wa)_\alpha$ be a sequence of solutions of \eqref{systype} such that \eqref{convdonnees} and \eqref{contradiction} hold. 
There exists $N_\alpha \in \mathbb{N}^*$ and $N_\alpha$ points $(x_{1,\alpha} , \cdots, x_{N_\alpha, \alpha})$ of $M$ satisfying, up to a subsequence:
\begin{enumerate}
\item $\nabla \ua (x_{i, \alpha}) = 0 $ for $1 \le i \le N_\alpha$, 
\item $d_g \left( x_{i, \alpha}, x_{j, \alpha} \right)^{\pui} \ua(x_{i,\alpha}) \le 1$ for all $1 \le i,j \le N_\alpha$, $i\not = j$, and
\item \label{trois} there exists a positive constant $C$ independent of $\alpha$ such that
\ben \label{contptsconc}
 \left(  \min_{1 \le i \le N_\alpha} d_g \left( x_{i,\alpha}, x \right) \right)^n \left( \ua(x)^{2^*} + |\Lg \Wa|_g(x) \right) \le C 
 \een
for any $x \in M$.
\end{enumerate}
\end{prop}

\begin{proof}
Applying Lemma $1.1$ in Druet-Hebey \cite{DruHeb} we obtain that for any $\alpha$ there exists $N_\alpha \ge 1$ and $N_\alpha$ critical points $x_{1,\alpha}, \cdots, x_{N_\alpha, \alpha}$ of $\ua$ satisfying:
\[ d_g(x_{i, \alpha},x_{j,\alpha})^\pui \ua(x_{i,\alpha}) \ge 1 \quad \textrm{ for } 1 \le i, j \le N_\alpha, i\not = j,\]
and
\ben  \label{controleptscrit}
\left(  \min_{1 \le i \le N_\alpha} d_g \left( x_{i,\alpha}, x \right) \right)^\pui  \ua(x) \le 1
\een
for any critical point $x$ of $\ua$. For any $x \in M$ let
\[ \Psi_\alpha(x) =   \left(  \min_{1 \le i \le N_\alpha} d_g \left( x_{i,\alpha}, x \right) \right)^n \left( \ua(x)^{2^*} + |\Lg \Wa|_g(x) \right). \]
Assume by contradiction that \eqref{contptsconc} is false, and let $\xa \in M$ be a sequence of points of $M$ such that
\ben \label{explosion}
\Psi_\alpha(\xa) = \sup_M \Psi_\alpha \to + \infty
\een
as $\alpha \to + \infty$. Define $\ma$ by:
\ben \label{ma}
\ma^{-n} = \ua(\xa)^{2^*} + |\Lg \Wa|_g(\xa).
\een
By \eqref{explosion} there holds:
\ben \label{distxasa}
\frac{d_g \left( \xa, \mathcal{S}_\alpha \right)}{\ma} \to + \infty
\een
where $\mathcal{S}_\alpha = \{ x_{1, \alpha}, \cdots, x_{N_\alpha, \alpha} \}$ is constructed above and since $M$ is compact, there holds
\ben \label{maconv0}
\ma \to 0
\een
as $\alpha \to + \infty$. Since $(M,g)$ is locally conformally flat we can let $\delta < i_g(M)$ and we can find a local chart $(B_{\xa}(\delta), \Phi_\alpha)$ centered at $\xa$ such that in this chart there holds
\[ g_{ij} = \vp_\alpha^{\frac{4}{n-2}} \xi_{ij} \]
for $1 \le i,j \le n$ and for some positive function $\vpa$ in $\Phi_\alpha(B_{\xa}(\delta))$ satisfying in addition:
\ben \label{propvpa}
 \vpa(0) = 1 \textrm{ and } \nabla \vpa (0) = 0.
 \een
 In addition to \eqref{propvpa} we can choose the conformal factor $\vpa$ to be bounded in the $C^k(M)$-topology for any $k \ge 0$ in $B_0(\delta)$. The operators $\triangle_g, \Lg$ and $\Dg$ satisfy some conformal invariance formulae. Let $v$ be a smooth function in $M$ and $X$ be a smooth $1$-form in $M$. For any $x \in \Phi_\alpha (B_{\xa}(\delta))$ we have that:
 \ben \label{duconforme}
\triangle_\xi \left( \vpa v \circ \Phi_\alpha^{-1} \right) (x) = \vpa^{2^*-1} (x)\left[ \triangle_g v (x) + \frac{n-2}{4(n-1)} R(g) v \right] (\Phi_\alpha^{-1}(x)),
\een
that
\ben \label{lwconforme}
\vpa^{\frac{4}{n-2}} \Lx \left( \vpa^{- \frac{4}{n-2}} \left( \Phi_\alpha \right)_* X \right) = \left( \Phi_\alpha \right)_* \left( \Lg X \right),
\een
and that for $1 \le i  \le n$:
\ben \label{dwconforme}
\bal
 \Dx \left( \vpa^{- \frac{4}{n-2}} \left( \Phi_\alpha \right)_* X \right)_i - 2^* \xi^{kl} \partial_k \left( \ln \vpa \right)&  \Lx  \left( \vpa^{- \frac{4}{n-2}} \left( \Phi_\alpha \right)_* X \right)_{li} \\
&  = \left( \Phi_\alpha \right)_* \left( \Dx X \right)_i .\\
\eal
\een
 We let, for any $x \in B_0(\frac{\delta'}{\ma})$ :
\ben \label{rescalingcarteconf}
\begin{aligned}
& \hua(x) = \ma^{\pui} \vpa(\ma x) \ua \circ \Phi_\alpha^{-1}(\ma x), \\
& \hWa(x) = \ma^{n-1} \vpa(\ma x)^{- \frac{4}{n-2}} \left( \Phi_\alpha \right)_* \Wa(\ma x), \\
\end{aligned}
\een
so that, using \eqref{duconforme} and \eqref{dwconforme} the sequence $(\hua, \hWa)$ satisfies the following system of equations in $B_0 \left( \frac{\delta'}{\ma}\right)$:
\ben \label{syscarteconf}
\left \{ \begin{aligned}
& \triangle_\xi \hua + \ma^2 \hat{h}_\alpha \hua = \hat{f}_\alpha \hua^{2^*-1} + \frac{\hat{a}_\alpha}{ \hua^{2^*+1}}, \\
& \left(  \Dx \hWa \right)_i = 2^* \ma \xi^{kl} \partial_k (\ln \vpa)(\ma \cdot)  \left( \Lx \hWa \right)_{li} + \ma \hua^{2^*} \left( \hat{X}_\alpha \right)_i + \ma^{n+1} \left( \hat{Y}_\alpha \right)_i,
\end{aligned} \right.
\een
where we have let:
\ben \label{paramcarteconf}
\begin{aligned}
& \hat{a}_\alpha(x) = \ma^{2n} \hat{b}_\alpha(x) + \left|\ma^n \hat{U}_\alpha(x) + \vpa(\ma x)^{2^*} \Lx \hWa \right|_\xi^2, \\
& \hat{h}_\alpha (x) = \vpa(\ma x)^{\frac{4}{n-2}} \left( h_\alpha \circ \Phi_\alpha^{-1} - \frac{n-2}{4(n-1)} R(g) \right)(\ma x), \\
& \hat{f}_\alpha (x) = f_\alpha \circ \Phi_\alpha^{-1}(\ma x), \\
& \hat{b}_\alpha(x) = \vpa(\ma x)^{2 \cdot 2^*} b_\alpha \circ \Phi_\alpha^{-1}(\ma x), \\
& \hat{U}_\alpha(x) = \vpa(\ma x)^2 \left( \Phi_\alpha \right)_* U_\alpha(\ma x), \\
& \hat{X}_\alpha(x) = \vpa(\ma x)^{- 2^*} \left( \Phi_\alpha \right)_* X_\alpha(\ma x) \textrm{ and } \\
& \hat{Y}_\alpha(x) = \left( \Phi_\alpha \right)_* Y_\alpha(\ma x). \\
\end{aligned}
\een
By definition of $\ma$ in \eqref{ma} there holds 
\ben \label{valua0}
\hua(0)^{2^*}+ |\Lx \hWa|_\xi(0) = 1
\een
and using \eqref{explosion} and \eqref{ma} there holds, for any $R > 0$:
\ben \label{borneloc}
\sup_{B_0(R)} \left( \hua^{2^*} + |\Lx \hWa|_\xi \right) \le 1 + o(1).
\een
We now let some $R > 0$ and $x \in B_0(R)$. A Green's formula for the first equation in \eqref{syscarteconf} shows that:
\[ \begin{aligned} 
\hua(x) & \ge 
 \int_{B_x(3R)} \frac{1}{(n-2) \omega_{n-1}} \Big( |x-y|^{2-n} - (3R)^{2-n} \Big) \frac{\hat{a}_\alpha}{\hua^{2^*+1}}(y) dy \\
& - \ma^2 \int_{B_x(3R)} \frac{1}{(n-2) \omega_{n-1}} \Big( |x-y|^{2-n} - (3R)^{2-n} \Big)\hat{h}_\alpha(y) \hua(y) dy
\end{aligned}\]
so that it is easily seen that there holds, for some positive constant $C$ that does not depend on $\alpha$ nor on $R$:
\be
\int_{B_x(2R)}|x-y|^{2-n} \left| \Lx \hWa \right|^2_\xi(y) dy \le C.
\ee
As a consequence, there exists $s_\alpha \in (\frac{3}{2}R, 2R)$ such that, up to a subsequence:
\ben \label{intbord}
\int_{\partial B_0(s_\alpha)} \left| \Lx \hWa \right|_\xi^2(y) d \sigma(y) \le C R^{n-3}.
\een
Let now $x \in B_0(R)$. By the properties of $\vpa$ in \eqref{propvpa} there holds, for some positive $C$:
\[ \left| 2^* \ma \xi^{kl} \partial_k (\ln \vpa)(\ma \cdot)  \left( \Lx \hWa \right)_{li} \right| \le C \ma^2 |y| \left| \Lx \hWa \right|_\xi \]
so that a Green formula for the operator $\Dx$ in $B_0(2R)$ (see Section \ref{Greenfunc}) along with \eqref{borneloc} and \eqref{intbord} yield, for some positive $C$ and $C'$:
\[ \begin{aligned}
\left| \Lx \hWa \right|_\xi(x) & \le C \int_{B_0(s_\alpha)} |x-y|^{1-n} \left| \Dx \hWa \right|_\xi dy + C \int_{\partial B_0(s_\alpha)} |x-y|^{1-n} \left| \Lx \hWa \right|_\xi(y) d \sigma(y) \\
& \le C' R \ma + \frac{C'}{R}.
\end{aligned}
\]
In the end we thus obtain that
\ben \label{lwaconv0}
\left| \Lx \hWa \right|_\xi \to 0 \textrm{ in } C^0_{loc}(\mathbb{R}^n)
\een
as $\alpha \to + \infty$. Coming back to \eqref{valua0} with \eqref{lwaconv0} we obtain that
\[ \hua(0) = 1 + o(1)\]
and independently that
\ben \label{convaa}
 \hat{a}_\alpha \to 0 \textrm{ in } C^0_{loc} (\mathbb{R}^n)
 \een
as $\alpha \to + \infty$. The Harnack inequality as stated in Section \ref{HI} shows then that for any compact set $K \subset \mathbb{R}^n$ there exists a positive constant $C(K)$ such that
\ben \label{uaminor}
C(K)^{-1} \le  \hua \le C(K) \textrm{ in } K.
 \een
By \eqref{convaa}, \eqref{uaminor} and standard elliptic theory one gets that $\hua \to U$ in $C^1_{loc}(\mathbb{R}^n)$, where $U$ satisfies $U(0) = 1$ and
\[ \triangle_{\xi} U = f_0(x_0) U^{2^*-1}, \]
where $x_0 = \lim \xa$. The classification result in Caffarelli-Gidas-Spruck \cite{CaGiSp} then shows that 
\[ U(x) = \left( 1 + \frac{f_0(x_0)}{n(n-2)}\right)^{1 - \frac{n}{2}}.\]
In particular, since $0$ is a strict local maxima of $U$, for $\alpha$ large enough there exists a sequence $\ya \in M$ of critical points of $\ua$, with $d_g(\xa, \ya) = o(\ma)$ and $\ma^\pui \ua(\ya) \to 1$ as $\alpha \to + \infty$ . Using \eqref{distxasa} this contradicts \eqref{controleptscrit} applied at $\ya$ and proves that \eqref{contptsconc} must hold, which concludes the proof of the Proposition.
\end{proof}

Proposition \ref{proppointconc} provides us with a suitable set of concentration points around which to investigate the behavior of $\ua$. For the following claims of this section we consider two sequences $(\xa)_\alpha$ and $(\ra)_\alpha$, where $\xa \in M$ and $16 \ra < i_g(M)$, such that $\nabla \ua (\xa) = 0$,
\ben \label{pointconc1}
d_g(\xa,x)^n \left( \ua(x)^{2^*} + |\Lg \Wa|_g(x) \right) \le C \quad \textrm{ for } x \in B_{\xa}(8 \ra), 
\een
and such that
\ben \label{pointconc2}
\ua(\xa) \ge \frac{1}{C}
\een
for some positive constant $C$ independent of $\alpha$. An example of such sequences is given by $\xa = x_{i,\alpha}$ and $\ra = \frac{1}{16} d_g \left( x_{i,\alpha}, \{x_{j, \alpha}, j \not = i \} \right)$ where $1 \le i \le N_\alpha$ and $x_{i,\alpha}$ is as \ref{proppointconc}. Two cases can occur and are described by each of the following Propositions:

\begin{prop} \label{propononblow}
Assume that \eqref{convdonnees} holds and let $(\ua, \Wa)_\alpha$ be a sequence of solution of \eqref{systype} such that \eqref{contradiction} holds. Assume that for some $C_1>0$ there holds, up to a subsequence:
\ben \label{nonblow}
\ra^n  \sup_{B_{\xa}(8 \ra)} \left( \ua^{2^*} + |\Lg \Wa |_g \right) \le C_1. 
\een
Then, for any $x \in B_{\xa}(8 \ra)$, there holds:
\[ \ra^\pui \ua(x) \ge \frac{1}{C_2}\]
for some positive $C_2$.
\end{prop}

\begin{proof}
Let $(B_{\xa}(\delta), \Phi_\alpha)$, $\delta < i_g(M)$, be a conformal chart around $\xa$ satisfying the same properties as the one constructed in \eqref{propvpa}. Let, for any $ x \in B_0(8)$, 
\[ \begin{aligned}
& \tua (x) =  \ra^\pui \vpa(\ra x) \ua \circ \Phi_\alpha^{-1} (\ra x), \\
& \tWa (x) = \ra^{n-1} \vpa(\ra x)^{- \frac{4}{n-2}} \left( \Phi_\alpha \right)_* \Wa(\ra x) .\\ 
\end{aligned}
\]
Using \eqref{duconforme} and \eqref{lwconforme} $\tua$ satisfies: 
\[ \triangle_\xi \tua + \ra^2 \tilde{h}_\alpha \tua = \tilde{f}_\alpha \tua^{2^*-1} + \left( \ra^{2n} \tilde{b}_\alpha + \left| \ra^n \tilde{U}_\alpha + \vpa^{2^*}(\ra x) \Lx \tWa(x) \right|_\xi^2 \right) \tua^{-2^*-1},\]
where $ \tilde{h}_\alpha$, $\tilde{f}_\alpha$, $\tilde{b}_\alpha$, $\tilde{U}_\alpha$, are defined as in \eqref{paramcarteconf} replacing $\ma$ by $\ra$. Using \eqref{nonblow} there holds:
\[  \ra^{2n} \tilde{b}_\alpha + \left| \ra^n \tilde{U}_\alpha + \vpa^{2^*}(\ra x) \Lx \tWa(x) \right|_\xi^2 \le C \textrm{ in } B_0(8)\]
for some $C > 0$, so that using the Harnack inequality as stated in Section \ref{HI} and \eqref{pointconc2} there holds:
\[ \inf_{B_0(8)} \tua \ge \frac{1}{C}\]
which concludes the proof of the Proposition. 
\end{proof}

\begin{prop} \label{propoblow}
Assume that \eqref{convdonnees} holds and let $(\ua, \Wa)_\alpha$ be a sequence of solution of \eqref{systype} such that \eqref{contradiction} holds. We assume that the assumptions of Theorem \ref{Th1} hold and that
\ben \label{propblow}
\ra^n \sup_{B_{\xa}(8 \ra)} \left( \ua^{2^*} + |\Lg \Wa|_g \right) \to +\infty \quad \textrm{ as } \alpha \to +\infty.
\een
We let 
\[ \ma = \ua(\xa)^{- \frac{2}{n-2}}.\]
Then $\ma \to 0$ and there holds:  $\frac{\ra}{\ma} \to + \infty$ and $\ra \to 0$ as $\alpha \to + \infty$. If we let in addition, for any $x \in M$,
\[ B_\alpha(x) = \ma^{\pui} \left( \ma^2 + \frac{f_\alpha(\xa)}{n(n-2)} d_g(\xa,x)^2  \right)^{1 - \frac{n}{2}},\]
there holds:
\[ \sup_{B_{\xa}(\ra)} \left| \frac{\ua}{B_\alpha} - 1 \right| \to 0\]
as $\alpha \to + \infty$.
\end{prop}
\noindent Proposition \ref{propoblow} is the main part of the analysis in this paper and we postpone its proof for now. Section \ref{analyseasymptotique} is devoted to the proof of Proposition \ref{propoblow}.

\medskip

\noindent We now conclude the proof of Theorem \ref{Th1}, assuming temporarily Proposition \ref{propoblow}. We let 
\ben \label{defda}
16 \dda = \min_{1 \le i,j \le N_\alpha} d_g \left( x_{i,\alpha}, x_{j,\alpha} \right),
\een
where the $x_{i,\alpha}$ are given by Proposition \ref{proppointconc}. We assume that $\dda \to 0$ as $\alpha \to + \infty$ and that the $x_{i, \alpha}$ have been reordered so that 
\ben \label{distx1x2}
16 \dda = d_g \left( x_{1,\alpha}, x_{2,\alpha} \right).
\een
Note that this definition only has a meaning if $N_\alpha \ge 2$ up to a subsequence. If $N_\alpha = 1$ we let $d_\alpha = \frac{1}{16} i_g(M)$. Let $(B_{x_{1,\alpha}}(\delta), \Phi_\alpha)$, with $2 \delta < i_g(M)$, be a conformal chart centered at $x_{1,\alpha}$ and satisfying $\left( \Phi_\alpha \right)_* g = \vpa^{\frac{4}{n-2}} \xi$, where $\vpa$ is as in \eqref{propvpa}.  We let
\ben \label{defcua}
\begin{aligned}
& \cua(x) = \dda^{\pui} \vpa(\dda x) \ua \circ \Phi_\alpha^{-1}(\dda x), \\
& \cWa(x) = \dda^{n-1} \vpa(\dda x)^{- \frac{4}{n-2}} \left( \Phi_\alpha \right)_* \Wa(\dda x), \\
\end{aligned}
\een
so that, using \eqref{duconforme} and \eqref{dwconforme} the sequence $(\cua)$ satisfies in $B_0 \left( \frac{\delta}{\dda} \right)$:
\ben \label{syscua}
 \triangle_\xi \cua + \dda^2 \check{h}_\alpha \cua = \check{f}_\alpha \cua^{2^*-1} + \frac{\check{a}_\alpha}{ \cua^{2^*+1}},
\een
where we have let:
\ben \label{paramcheck}
\begin{aligned}
& \check{a}_\alpha(x) = \dda^{2n} \check{b}_\alpha(x) + \left| \dda^n \check{U}_\alpha(x) + \vpa(\dda x)^{2^*} \Lx \cWa \right|_\xi^2, \\
& \check{h}_\alpha (x) = \vpa(\dda x)^{\frac{4}{n-2}} \left( h_\alpha \circ \Phi_\alpha^{-1} - \frac{n-2}{4(n-1)} R(g) \right)(\dda x), \\
& \check{f}_\alpha (x) = f_\alpha \circ \Phi_\alpha^{-1}(\dda x), \\
& \check{b}_\alpha(x) = \vpa(\dda x)^{2 \cdot 2^*} b_\alpha \circ \Phi_\alpha^{-1}(\dda x), \\
& \check{U}_\alpha(x) = \vpa(\dda x)^2 \left( \Phi_\alpha \right)_* U_\alpha(\dda x). \\
\end{aligned}
\een
For $1 \le i \le N_\alpha$ we let $\cxia = \frac{1}{\dda} \Phi_\alpha( x_{i,\alpha})$. Let $R > 0$ and define $N_R$ by the following property:
\[ |\cxia| \le R \iff 1 \le i \le N_R .\]
By \eqref{distx1x2} there holds $N_R \ge 2$ for $R$ large enough. Propositions \ref{propononblow} and \ref{propoblow} show that for any $1 \le i \le N_R$ the following alternative occurs:
\ben \label{alternatives}
\bal
& \textrm{ either there exists } C_i > 0 \textrm{ such that }  \frac{1}{C_i} \le \cua \le C_i 
 \textrm{ in } B_{\cxia}\left(\frac{1}{2} \right) \\
& \textrm{ or } \sup_{B_{\cxia}(\frac{1}{2})} \left| \frac{\cua}{\cBia} -1 \right| \to 0 \textrm{ as } \alpha \to + \infty,
\eal
\een
where we have let 
\ben \label{defcBia}
\cBia(x) = \cmia^{\pui} \left( \cmia^2 + \frac{f_\alpha(x_{i, \alpha})}{n(n-2)} \left| x - \cxia \right|^2 \right)^{1 - \frac{n}{2}}
\een
and $ \cmia = \cua(\cxia)^{- \frac{2}{n-2}}$. By Proposition \ref{propoblow} there holds $\cmia \to 0$ as $\alpha \to + \infty$. At any point $\cxia$ these alternatives are exclusive and the second one occurs only when $\cua (\cxia) \to + \infty$ or when $\cua \to 0$ in $C^0 \left( B_{\cxia} \left( \frac{1}{2} \right) \backslash B_{\cxia} \left( \frac{1}{4} \right) \right)$. We start proving that either the first alternative in \eqref{alternatives} holds for any $1 \le i \le N_R$ or the second alternative holds for all $1 \le i \le N_R$. Let $x \in B_0(R)$. We let $\check{G}_\alpha(x, \cdot)$ be the Green function of the operator $\triangle_\xi + \dda^2 \check{h}_\alpha$ in $B_x(3R)$. Since $\triangle_g + h_0$ is coercive there holds (see Robert \cite{RobDirichlet} ) that for any $y \in B_x(2R)$,
\[ \check{G}_\alpha(x,y) \ge \frac{1}{C}\]
for some positive $C$. A Green formula thus shows that:
\[ \cua(x) \ge \frac{1}{C} \int_{B_x(2R)} \check{f}_\alpha (y) \cua(y)^{2^*-1} dy,\]
so that if there exists $1\le i_0 \le N_R$ for which the first alternative in \eqref{alternatives} is satisfied then there holds:
\[ \cua(x) \ge C(R) > 0 \]
for some positive $C(R)$ depending on $n$ and on $R$. Hence the first alternative in \eqref{alternatives} is satisfied in each $\cxia$, $1 \le i \le N_R$.

\noindent We claim now that the second alternative in \eqref{alternatives} cannot hold for all $1 \le i \le N_R$. We let $x \in B_0(R)$ and write once again a Green formula as above: there holds 
\[
\bal
 \cua(x) & \ge \int_{B_0 \left(\frac{1}{2} \right)} \check{G}_\alpha(x,y) \check{f}_\alpha (y) \cua(y)^{2^*-1} dy \\
&  +  \int_{B_{\check{x}_{2,\alpha}}  \left(\frac{1}{2} \right)} \check{G}_\alpha(x,y) \check{f}_\alpha (y) \cua(y)^{2^*-1} dy.
\eal
\]
Since $\check{G}_{\alpha}(x, \cdot) $ converges to the Green function of $\triangle_\xi$ in $B_x(3R)$ in $C^1_{loc}(B_x(3R) \backslash \{x \} )$ standard computations yield:
\ben \label{desbulles}
\cua(x) \ge \left( 1 + o(1) \right) \left( \check{B}_{1, \alpha}(x) + \check{B}_{2, \alpha}(x) \right) - \frac{C}{R^{n-2}} \left( \mu_{1,\alpha}^\pui + \mu_{2,\alpha}^\pui \right)
\een
for some positive $C$ that does not depend on $R$ nor on $\alpha$. We assume now that $|x| \le \frac{1}{4}$ and $x \not = 0$. Applying the second alternative in \eqref{alternatives} at $0$ yields, with \eqref{desbulles} and \eqref{defcBia}:
\[
\left( \frac{ \mu_{1,\alpha}}{\mu_{2,\alpha}} \right)^\pui |x|^{n-2} \Big( |x-\check{x}_2|^{2-n} - C R^{2-n} + o(1) \Big) \le o(1) + \frac{C}{R^{n-2}} |x|^{2-n},
\]
where $\check{x}_2 = \lim \check{x}_{2,\alpha}$, so that $|\check{x}_2| = 16 $ by \eqref{distx1x2}. In the end letting $\alpha \to +\infty$, dividing by $|x|$ and letting $x \to 0$ yields, for $R$ large enough:
\ben \label{contrafinale}
\limsup_{\alpha \to + \infty} \left( \frac{ \mu_{1,\alpha}}{\mu_{2,\alpha}} \right)^\pui \le C \frac{16^{n-2}}{R^{n-2} - C 16^{n-2}}.
\een
The same arguments as those developed to obtain \eqref{contrafinale} work in the same way exchanging the roles of $\check{x}_{1,\alpha}$ and $\check{x}_{2,\alpha}$ so that in the end we also obtain:
\be
\limsup_{\alpha \to + \infty} \left( \frac{ \mu_{2,\alpha}}{\mu_{1,\alpha}} \right)^\pui \le C \frac{16^{n-2}}{R^{n-2} - C 16^{n-2}}.
\ee
This clearly contradicts \eqref{contrafinale} up to choosing $R$ large enough. So far we have proven  that for $R$ large enough and for any $1 \le i \le N_R$ the first alternative in \eqref{alternatives} holds at every $\cxia$. Note in particular that by Proposition \ref{propononblow} there also holds $\left| \Lx \cWa \right|_\xi \le C_i$ in $B_{\cxia} \left( \frac{1}{2} \right)$. Thus, by Proposition \ref{proppointconc} and by the first alternative in \eqref{alternatives} there holds that $ \cua$ and $\left| \Lx \cWa \right|_\xi $ belong to $L^\infty_{loc} (\RR^n)$. Mimicking the arguments in the proof of Proposition \ref{proppointconc} that led to \eqref{lwaconv0} we have here again:
\[ \left| \Lx \cWa \right|_\xi \to 0 \textrm{ in } C^0_{loc}(\RR^n) \]
so that, as in the proof of Proposition \ref{proppointconc}, we obtain that $\cua \to \check{U}$ in $C^1_{loc} (\RR^n)$, where $U$ satisfies
\[ \triangle_\xi U = f_0(x_1) U^{2^*-1}, \]
with $x_1 = \lim x_{1,\alpha}$ and where $U(x_1) > 0$ since the first alternative of \eqref{alternatives} holds at $0$. The classification result of Caffarelli-Gidas-Spruck \cite{CaGiSp} thus shows that $U$ has only a critical point, but this is impossible since $\cua$ had at least two critical points, namely $0$ and $\check{x}_{2,\alpha}$, and $\left| \check{x}_{2,\alpha} \right| = 16$. This therefore contradicts the fact that $\dda \to 0$ as $\alpha \to + \infty$.

\medskip

Thus there exists some $\delta_0 > 0$ such that $\dda \ge \delta_0$. In particular the $x_{i,\alpha}$ constructed in Proposition \ref{proppointconc} are in finite number and around each of them $\ua$ is bounded, since the situation of Proposition \ref{propoblow} cannot happen because $\dda \not \to 0$. In the end, $\ua$ is bounded in $M$ and Theorem \ref{Th1} is proven.

\section{Blow-up analysis -- Proof of Proposition \ref{propoblow}} \label{analyseasymptotique}

In this section we develop the asymptotic blow-up analysis needed to prove Proposition \ref{propoblow}. We let $(\ua, \Wa)_\alpha$ be a sequence of solutions of \eqref{systype} and assume that \eqref{convdonnees} and the assumptions of Theorem \ref{Th1} hold. 
We let $(\xa)_\alpha$ be a sequence of critical points of $\ua$ in $M$ and $(\ra)_\alpha$ be a sequence of positive real numbers with $16 \ra < i_g(M)$ such that
\ben \label{pc1}
d_g(\xa,x)^n \left( \ua(x)^{2^*} + |\Lg \Wa|_g(x) \right) \le C \quad \textrm{ for } x \in B_{\xa}(8 \ra), 
\een
and such that
\ben \label{p2}
\ra^n \sup_{B_{\xa}(8 \ra)} \left( \ua^{2^*} + |\Lg \Wa|_g \right) \to +\infty \quad \textrm{ as } \alpha \to +\infty.
\een
In this section, the letter $C$ will always denote a positive constant, whose value may change from one line to another but that will never depend on $\alpha$. Let $(B_{\xa}(16 \ra), \Phi_\alpha)$ be some conformal chart around $\xa$ as in \eqref{propvpa}, such that in the chart there holds $g_{ij} = \vp_\alpha^{\frac{4}{n-2}} \xi_{ij}$ for some conformal factor $\vpa$ . We can always assume that $B_0(8 \ra) \subset \Phi_\alpha(B_{\xa}(16 \ra))$ and that the sequence of conformal factors $(\vpa)_\alpha$ is uniformly bounded in $C^k(B_0(8 \ra))$ for any $k \ge 0$. We let, for any $x \in B_0(8 \ra)$ :
\ben \label{blowcarteconf}
\begin{aligned}
& \va(x) =  \vpa( x) \ua \circ \Phi_\alpha^{-1}( x), \\
& \Za(x) = \vpa( x)^{- \frac{4}{n-2}} \left( \Phi_\alpha \right)_* \Wa( x). \\
\end{aligned}
\een
Using \eqref{duconforme} and \eqref{dwconforme} the sequence $(\va, \Za)_\alpha$ satisfies the following system of equations in $B_0( 8 \ra)$:
\ben \label{sysconf}
\left \{ \begin{aligned}
& \triangle_\xi \va + \tilde{h}_\alpha \va = \tilde{f}_\alpha \va^{2^*-1} + \frac{\tilde{a}_\alpha}{ \va^{2^*+1}}, \\
& \left(  \Dx \Za \right)_i = 2^*  \xi^{kl} \partial_k (\ln \vpa)  \left( \Lx \Za \right)_{li} +  \va^{2^*} \left( \tilde{X}_\alpha \right)_i +  \left( \tilde{Y}_\alpha \right)_i,
\end{aligned} \right.
\een
where we have let:
\ben \label{paramconf}
\begin{aligned}
& \tilde{a}_\alpha(x) = \tilde{b}_\alpha(x) + \left| \tilde{U}_\alpha(x) + \vpa( x)^{2^*} \Lx \Za \right|_\xi^2, \\
& \tilde{h}_\alpha (x) = \vpa( x)^{\frac{4}{n-2}} \left( h_\alpha \circ \Phi_\alpha^{-1} - \frac{n-2}{4(n-1)} R(g) \right)( x), \\
& \tilde{f}_\alpha (x) = f_\alpha \circ \Phi_\alpha^{-1}( x), \\
& \tilde{b}_\alpha(x) = \vpa( x)^{2 \cdot 2^*} b_\alpha \circ \Phi_\alpha^{-1}(x), \\
& \tilde{U}_\alpha(x) = \vpa( x)^2 \left( \Phi_\alpha \right)_* U_\alpha( x), \\
& \tilde{X}_\alpha(x) = \vpa( x)^{- 2^*} \left( \Phi_\alpha \right)_* X_\alpha( x) \textrm{ and } \\
& \tilde{Y}_\alpha(x) = \left( \Phi_\alpha \right)_* Y_\alpha( x). \\
\end{aligned}
\een
As before, note that by \eqref{propvpa} there holds, for $x \in B_0(8 \ra)$:
\ben \label{contconf}
\left| 2^*  \xi^{kl} \partial_k (\ln \vpa)  \left( \Lx \Za \right)_{li}\right| \le C |y| \left| \Lx \Za \right|_\xi,
\een
and \eqref{pc1} becomes: for all $x \in B_0(4 \ra)$,
\ben \label{contfaible}
|x|^n \left( \va^{2^*}(x) + |\Lx \Za |_\xi(x) \right) \le C.
\een

\subsection{Sharp pointwise estimates}

We first prove a local version of Proposition \ref{propoblow}.

\begin{claim} \label{claimlocal}
Let 
\ben \label{defma}
\ma = \ua(\xa)^{- \frac{2}{n-2}} = \va(0)^{- \frac{2}{n-2}}. 
\een
There holds, up to a subsequence: $ \ma \to 0$, $\frac{\ra}{\ma} \to + \infty$,
\[ \ma^{\pui} \va(\ma x) \to U(x) = \left( 1 + \frac{f_0(x_0)}{n(n-2)} |x|^2 \right)^{1- \frac{n}{2}} \textrm{ in } C^{1,\eta}_{loc}(\mathbb{R}^n),\]
and
\[ \ma^n \left| \Lx \Za \right|_\xi (\ma x) \to 0 \textrm{ in } C^0_{loc}(\mathbb{R}^n)\]
as $\alpha \to + \infty$, where $x_0 = \underset{\alpha \to \infty} {\lim} \xa$.
\end{claim}

\begin{proof}
The proof of Claim \ref{claimlocal} is similar to that of Proposition \ref{proppointconc}. Let $\ya \in B_{\xa}(8 \ra)$ be a sequence of points of $M$ satisfying
\ben \label{maxlocalya}
\ua^{2^*}(\ya) + \left| \Lg \Wa \right|_g(\ya) = \sup_{B_{\xa}(8 \ra)} \left( \ua^{2^*} + \left| \Lg \Wa \right|_g \right),
\een
and let 
\ben \label{defnua}
\na^{-n} = \ua^{2^*}(\ya) + \left| \Lg \Wa \right|_g(\ya).
\een
By \eqref{p2} there holds
\ben \label{rhosurnu}
\frac{\ra}{\na} \to + \infty
\een
as $\alpha \to \infty$, so that in particular there holds
\ben \label{nato0}
\na \to 0
\een
as $\alpha \to + \infty$. By \eqref{pc1} we also get:
\ben \label{yadistfinie}
d_g(\ya,\xa) \le C \na
\een
for some positive $C$. We let, for any $x \in B_0 \left( \frac{8 \ra}{\na}\right)$:
\ben
\bal
& \hva (x) = \na^\pui \va(\na x), \\
& \hat{Z}_\alpha (x) = \na^{n-1} \Za(\na x) ,\\
\eal
\een
where $\va$ and $\Za$ are as in \eqref{blowcarteconf}. It is easily seen that $\hva$ and $\hat{Z}_\alpha$ satisfy:
\be 
\left \{ \begin{aligned}
& \triangle_\xi \hva + \na^2 \hat{h}_\alpha \hva = \hat{f}_\alpha \hva^{2^*-1} + \frac{\hat{a}_\alpha}{ \hva^{2^*+1}}, \\
& \left(  \Dx \hat{Z}_\alpha \right)_i = 2^* \na \xi^{kl} \partial_k (\ln \vpa)(\na x)  \left( \Lx \hat{Z}_\alpha \right)_{li} +  \na \hva^{2^*} \left( \hat{X}_\alpha \right)_i + \na^{n+1} \left( \tilde{Y}_\alpha \right)_i,
\end{aligned} \right.
\ee
where we have let:
\ben \label{paramconfna}
\begin{aligned}
& \hat{a}_\alpha(x) = \na^{2n} \hat{b}_\alpha(x) + \left| \na^n \hat{U}_\alpha(x) + \vpa(\na x)^{2^*} \Lx \hat{Z}_\alpha \right|_\xi^2, \\
& \hat{h}_\alpha (x) = \tilde{h}_\alpha(\na x), \\ 
& \hat{f}_\alpha (x) = \tilde{f}_\alpha(\na x), \\ 
& \hat{b}_\alpha(x) = \tilde{b}_\alpha(\na x), \\
& \hat{U}_\alpha(x) = \tilde{U}_\alpha(\na x), \\
& \hat{X}_\alpha(x) = \tilde{X}_\alpha(\na x), \\ 
& \hat{Y}_\alpha(x) = \tilde{Y}_\alpha(\na x). \\
\end{aligned}
\een
By definition of $\hva$ and $\hat{Z}_\alpha$ there holds: for any $x \in B_0 \left( \frac{8 \ra}{\ma}\right)$
\ben \label{bornesupna}
\hva^{2^*}(x) + \left| \Lx \hat{Z}_\alpha \right|_\xi (x) \le 1 = \hva^{2^*}(\hat{y}_\alpha) + \left| \Lx \hat{Z}_\alpha \right|_\xi (\hat{y}_\alpha),
\een
where we have let $\hat{y}_\alpha = \frac{1}{\na} \Phi_\alpha(\ya)$. Let $R > 0$ and $x \in B_0(R)$. Mimicking the arguments used in the proof of Proposition \ref{proppointconc} to obtain \eqref{intbord}, it is easily seen that there exists a positive constant $C$ that does not depend on $\alpha$ nor on $R$ and a sequence $s_\alpha \in (\frac{3}{2}R, 2R)$ such that
\ben \label{intbordna}
\int_{\partial B_0(s_\alpha)} \left| \Lx \hat{Z}_\alpha \right|_\xi^2(y) d \sigma(y) \le C R^{n-3}.
\een
As a consequence, using \eqref{intbordna}, the arguments in the proof of Proposition \ref{proppointconc} that led to \eqref{lwaconv0} still apply here and show that there holds:
\ben \label{lwaconv0na}
\left| \Lx \hat{Z}_\alpha \right|_\xi \to 0 \textrm{ in } C^0_{loc}(\RR^n)
\een
as $\alpha \to + \infty$. By \eqref{lwaconv0na} we therefore obtain with \eqref{bornesupna}  that $\hva(\hat{y}_\alpha) = 1 + o(1)$ and that
\[  \hat{a}_\alpha(x)  \to \textrm{ in } C^0_{loc}(\RR^n),\]
where $ \hat{a}_\alpha(x) $ is as in \eqref{paramconfna}. By \eqref{yadistfinie}, there exists $\hat{y}_0 \in \RR^n$ such that $\hat{y}_\alpha \to \hat{y}_0$ as $\alpha \to + \infty$. Here again, the Harnack inequality of Section \ref{HI} shows that
\ben \label{convhvana}
 \hva \to U \textrm{ in } C^1_{loc}(\RR^n),
 \een
where $U$ satisfies $\nabla U(0) = 0$ by definition of $\xa$, $U(\hat{y}_0) = 1$ and:
\[ \triangle_{\xi} U = f_0(x_0) U^{2^*-1}. \]
The classification result in Caffarelli-Gidas-Spruck therefore shows that $U(x) = \left( 1 + \frac{f_0(x_0)}{n(n-2)} |x|^2 \right)^{1 - \frac{n}{2}}$. In particular, $\hat{y}_0 = 0$, and since there independently holds by \eqref{defma} that
\[ \hva(\hat{y}_\alpha) =  \left( \frac{\na}{\ma}\right)^{\pui}, \]
we obtain in the end 
\[ \underset{\alpha \to \infty}{\lim} \frac{\na}{\ma} = 1, \]
which, with \eqref{convhvana}, \eqref{rhosurnu} and \eqref{nato0}, concludes the proof of Claim \ref{claimlocal}.
\end{proof}

\noindent We set in the following, for $x \in \mathbb{R}^n$:
\ben \label{bullea}
\Ba(x) = \ma^\pui \left( \ma^2 + \frac{f_\alpha(\xa)}{n(n-2)} |x|^2\right)^{1 - \frac{n}{2}},
\een
where $\ma$ is as in \eqref{defma}. It satisfies $\triangle_\xi \Ba = f_\alpha(\xa) \Ba^{2^*-1}$ in $\mathbb{R}^n$. 

\medskip

Claim \ref{claimlocal} establishes Proposition \ref{propoblow} when the distance to the concentration point $0$ is of order $\ma$. In order to prove Proposition \ref{propoblow} we have to extend the estimate on $\va$ given by Claim \ref{claimlocal} to the whole ball $B_0(\ra)$. We need to show that  the bubble $B_\alpha$ is dominant with respect to other possible defects of compactness of the sequence $(\va)_\alpha$  up to a radius $\ra$. The standard way to proceed would consist in showing that $\ra$ coincides with the maximal radius up to which $0$ is an isolated simple blow-up point of $\va$, following standard terminology.
The usual arguments, however, fail to work here. One reason is that the estimates we have on $\va$ and $\Lx \Za$ at this point -- given by Claim \ref{claimlocal} and the weak estimate \eqref{contfaible} -- 
do not improve because of the highly non-trivial coupling of system \eqref{systype}.   

\medskip

We therefore proceed in a different way. Let $\ve > 0$ be given. We define the radius of influence $\rra$ of the bubble centered at $0$ to be the maximal radius where $\va$ looks almost like a bubble $B_\alpha$, up to an $\ve$ error factor: precisely
\ben \label{defra}
r_\alpha = \sup \mathcal{R}
\een
where
\ben
\begin{aligned}
\mathcal{R} = \Bigg \{ 0 < r \le \ra \textrm{ such that } \va \le (1+\ve) \Ba,  \quad  \left| \nabla \left( \va - \Ba \right) \right|_\xi \le \ve|\nabla \Ba|_\xi  \\
 \textrm{ and } x^k \partial_k \va + \pui \va \le 0 \textrm{ on } B_0(r) \backslash B_0(2 R_\alpha \ma)  \Bigg \}, \\
\end{aligned}
\een 
where we have let
\ben \label{Ralpha}
 R_\alpha^2 = \frac{n(n-2)}{f_\alpha(\xa)}.
 \een
Using Claim \ref{claimlocal} there easily holds
\ben  \label{rsurmu}
\frac{r_\alpha}{\ma} \to + \infty
\een
as $\alpha \to + \infty$. Because of the definition of $\rra$, 
in $B_0(\rra)$ the unknown $\va$ looks almost like what we believe to be its sharp asymptotic. We will prove Proposition \ref{propblow} by proving that $\rra = \ra$. To do this we have to obtain sharp pointwise asymptotic estimates on $\va$ and $\Lx \Za$ in all of $B_0(\rra)$. To reach the desired precision we improve the estimates on $\va$ and $\Lx \Za$ step by step, by performing some kind of ping-pong game: we plug the available estimates on the unknowns in system \eqref{systype} and, through several Green representation formulae, iteratively recover better estimates for both the unknowns. 

We start improving the information on $\va$ in $B_0(8 \rra)$:
\begin{claim} \label{claimcontbulle}
Let $(\da)_\alpha$, $0 < \da \le \rra$ be some sequence of radii. There exists a sequence $(\kappa_\alpha)_\alpha$, $0 \le \kappa_\alpha  < 1$, of real numbers such that for any $\za \in B_0(8 \da)$ there holds:
\ben \label{contbulles}
\left( 1 - \kappa_\alpha \right) \Ba(\za) \le \va(\za) \le C B_\alpha(\za),
\een
where $C > 1$ is a constant that does not depend on $\alpha$ and $\Ba$ is as in \eqref{bullea}. Furthermore, if $\da \to 0$ as $\alpha \to \infty$, there holds $\kappa_\alpha \to 0$.
\end{claim}

\begin{proof}
We start proving the upper bound. We define, for $x \in B_0(8)$:
\[ \bar{v}_\alpha(x) = \rra^{\pui} \va(\rra x),\]
where $\va$ and $\rra$ are defined in \eqref{blowcarteconf} and \eqref{defra}. Then $\bar{v}_\alpha$ satisfies:
\[ \triangle_\xi \bar{v}_\alpha + \rra^2 \bar{h}_\alpha(x) \bar{v}_\alpha = \bar{f}_\alpha \bar{v}_\alpha^{2^*-1} +  \rra^{2n} \frac{\bar{a}_\alpha}{\bar{v}_\alpha^{2^*+1}},\]
where, using the notations in \eqref{paramconf}:
\[
\bal
& \bar{a}_\alpha(x) = \tilde{a}_\alpha(\rra x), \\ 
& \bar{h}_\alpha(x) = \tilde{h}_\alpha (\rra x) , \\ 
& \bar{f}_\alpha(x) =  \tilde{f}_\alpha ( \rra x). \\ 
\eal
\]
By the definition of $\rra$ there holds, for some positive $C$ that does not depend on $\alpha$:
\be 
 \bar{v}_\alpha (x) \le C \left( \frac{\ma}{\rra} \right)^{\pui} \textrm{ in } B_0(1) \backslash B_0 \left( \frac{1}{2} \right).  
 \ee
The weak estimate \eqref{contfaible} shows both that:
\[ \bar{v}_\alpha \le C \quad  \textrm{ and } \quad \rra^{2n} \bar{a}_\alpha \le C \qquad \textrm{ in } B_0(8) \backslash B_0 \left( \frac{1}{2} \right).\]
Hence the Harnack inequality, as stated in Section \ref{HI}, successively applied to the annuli $B_0(3/2) \backslash B_0(1/2), \dots, B_0(17/2) \backslash B_0(11/2)$ shows that:
\[  \bar{v}_\alpha (x) \le C \left( \frac{\ma}{\rra} \right)^{\pui} \textrm{ in } B_0(8) \backslash B_0 \left( \frac{1}{2} \right),   \]
for some positive $C$. By the definition of $\rra$ in \eqref{defra} and because of \eqref{rsurmu} and the expression of  $\Ba$ in \eqref{bullea} this yields  the upper bound in  \eqref{contbulles}.

\medskip
\noindent For the lower bound, we let $G_\alpha$ be the Green function of $\triangle_g + h_\alpha$ in $M$. Using the expression of $\va$ in \eqref{blowcarteconf} and the notations of \eqref{paramconf} we have that, for any sequence $\za$ of points in $ B_0(8 \rra)$:
\[ \va(\za) \ge \vpa(\za) \int_{B_0(8\rra)} \vpa(y) G_\alpha \left( \Phi_\alpha^{-1}(\za),  \Phi_\alpha^{-1}(y) \right) \tilde{f}_\alpha(y) \va^{2^*-1}(y) dy .\]
In particular, with the expression of $\Ba$ as in \eqref{bullea} we can write that:
\[
\bal
\frac{\va}{\Ba}(\za) \ge \vpa(\za) \bigintss_{B_0 \left( \frac{6 \rra }{\ma}\right)} \vpa(\ma y) \tilde{f}_\alpha(\ma y) \left( \ma^\pui \va(\ma y) \right)^{2^*-1} \\
\times G_\alpha \left( \Phi_\alpha^{-1}(\za),  \Phi_\alpha^{-1}(\ma y) \right) d_g \left( \Phi_\alpha^{-1}(\za),  \Phi_\alpha^{-1}(\ma y) \right)^{n-2} \\
\times   \left( \frac{ \ma^2 + \frac{f_\alpha(\xa)}{n(n-2)} |\za|^2}{d_g \left( \Phi_\alpha^{-1}(\za),  \Phi_\alpha^{-1}(\ma y) \right)^2} \right)^\pui dy ,
\eal
 \]
 which yields the lower bound in \eqref{contbulles} using Claim \ref{claimlocal}, Fatou's lemma and standard properties of Green functions (see Robert \cite{RobDirichlet}).
\end{proof}

Exploiting the coupling of the system allows one to recover an integral control on $\Lx \Za$, better than the one given by the weak estimate \eqref{contfaible}. This issue is addressed in the next Claim:

\begin{claim} \label{claimint}
Let $(\da)_\alpha$ be a sequence of positive numbers satisfying 
\ben \label{contda}
\frac{\ma}{\da}\to 0 \quad \textrm{ and } \quad  \da \le \min(\rra, \ma^{\frac{1}{2}}).
\een
There holds then, for some positive constant $C$, that for any $x \in B_0(7 \da)$,
\ben \label{claimint1}
\int_{B_0(6\da)} |x-y|^{2-n} \left| \Lx \Za \right|_\xi^2\va^{-2^*-1}(y) dy \le C \Ba(x).
\een
As a consequence:
\ben \label{claimint1plus}
\int_{B_0(6\da) \backslash B_0(\da)} \left| \Lx \Za \right|_\xi^2 dy \le C \ma^{2n-2} \da^{2-3n},
\een
and there exists a sequence of positive numbers $s_\alpha \in (5 \da, 6 \da)$ such that  
\ben \label{claimint2}
\int_{\partial B(s_\alpha)} \left| \Lx \Za \right|_\xi(y)^2 d \sigma(y) \le C \ma^{2n-2} \da^{1 - 3n} .
\een
\end{claim}

\begin{proof}
Using \eqref{sysconf} we write a Green formula for $\triangle_\xi + \tilde{h}_\alpha$ in $B_0(8 \da)$ to obtain that for some $C > 0$ there holds, for any $x \in B_0(7 \da)$:
\[ \va(x) \ge \frac{1}{C}\int_{B_0(6 \da)}  |x-y|^{2-n} \left( \tilde{b}_\alpha + \left| \tilde{U}_\alpha + \vp^{2^*} \Lx \Za \right|_\xi^2 \right)\va^{-2^*-1}(y) dy .\]
Using \eqref{contbulles} and \eqref{contda} yields \eqref{claimint1} since $\tilde{b}_\alpha \ge 0$. As a consequence, choosing $|x| = 7 \da$ in \eqref{claimint1} there holds
\[ \int_{B_0(6 \da) \backslash B_0(\lambda \da)} \left| 
 \Lx \Za \right|_\xi^2\va^{-2^*-1}(y) dy \le C  \ma^{\pui} \]
 for any $1 \le \lambda \le 6$ and  \eqref{claimint1plus} and \eqref{claimint2} follow then using once again \eqref{contbulles}.
\end{proof}

The goal of our analysis is to obtain sharp pointwise estimates on $\va$ and $|\Lx \Za|_\xi$. The two unknowns may blow-up simultaneously but at different rates and we are then led to the simultaneous investigation of their defects of compactness. 
\noindent Let us define now the following $1$-forms in $\mathbb{R}^n$ by: 
\ben \label{def1forme}
\begin{aligned}
& \Va(x)_i = \tilde{X}_\alpha(0)^j \int_{\mathbb{R}^n} \Ba^{2^*}(y) \mathcal{G}_i(x-y)_j dy, \\
& \Pak(x)_i = \partial_k \tilde{X}_\alpha(0)^j \int_{\mathbb{R}^n} y_k \Ba^{2^*}(y) \mathcal{G}_{i}(x-y)_j dy \quad \textrm{ for } 1 \le k \le n, \\
\end{aligned}
\een
where, for $y \not = 0$,
\ben \label{expG}
\bal
 \mathcal{G}_i(y)_j =&  - \frac{1}{4(n-1) \omega_{n-1}} |y|^{2-n} \left((3n-2) \delta_{ij} + (n-2) \frac{y_i y_j}{|y|^2} \right)
 \eal
 \een
is the $i$-th fundamental solution of $ \Dx $ in $\mathbb{R}^n$, see Section \ref{Greenfunc}. Then $\Va$ and $\Pak$ satisfy in $\mathbb{R}^n$:
\ben \label{eq1forme}
\Dx \Va = \va^{2^*} \tilde{X}_\alpha(0) \quad \textrm{ and } \quad \Dx \Pak = y_k \Ba^{2^*} \partial_k \tilde{X}_\alpha(0).
\een 
We also let in the following:
\ben \label{defeaba}
\vea = |\tilde{X}_\alpha(0)|_\xi, \quad \bak = | \partial_k \tilde{X}_\alpha (0) |_\xi, \quad \ba = \sum_{k=1}^n \bak
\een
and
\ben \label{defz}
\zeta_0 = \underset{\alpha \to \infty }{ \lim} \frac{ \tilde{X}_\alpha (0)}{ \vea}, \quad  \zeta_k = \underset{\alpha \to \infty }{ \lim} \frac{\partial_k \tilde{X}_\alpha (0)}{ \bak}  \textrm{ for } 1 \le k \le n,
\een
which are all vectors of Euclidean norm $1$. To define the $\zeta_k$ we should require $\vea$ and $\bak$ to be nonzero. This will however not be an issue since the vectors $\zeta_k$ always appear in the computations multiplied by $\vea$ or $\bak$. Define, for any $\za \in B_0(8 \rra)$:
\ben \label{ta}
\ta(\za) = \left( \ma^2 + |\za|^2 \right)^{\frac{1}{2}}.
\een
It is easily seen that there holds, for any $x \in \mathbb{R}^n$:
\ben \label{estVa}
\left| \Lx \Va \right|_\xi \le C \vea \ta(x)^{1-n},
\een
and
\ben \label{estPak}
\left| \Lx \Pak \right|_\xi \le C \bak \ma^2 \ta(x)^{-n}
\een
where $\ta(x)$ is as in \eqref{ta}. Moreover,  if $\vea \not = 0$ and $\ba \not = 0$ straightforward computations from \eqref{def1forme}  
give the following asymptotic expansion of $\Va$ and $\Pak$: for any sequence $\za \in B_0(8 \ra)$ satisfying $\frac{|\za|}{\ma} \to +\infty $ there holds, up to a subsequence:
\ben \label{asymptova}
\begin{aligned}
 \Lx \Va (\za)_{ij}  = C_1(n)f_\alpha(\xa)^{- \frac{n}{2}}  & \Bigg[ \delta_{ij} \zeta^p \check{z}_p - \zeta_i \check{z}_j - \zeta_j \check{z}_i - (n-2) \zeta^p \check{z}_p \check{z}_i \check{z}_j \Bigg]  \\
&  \times \vea |\za|^{1 - n}  (1 + o(1)), \\ 
\end{aligned}
\een 
and
\ben \label{asymptopak}
\begin{aligned}
&\Lx \Pak (\za)_{ij} =    C_2(n) f_\alpha(\xa)^{- \frac{n+2}{2}}  \left \{ - \left(\zeta_k \right)_i \Bigg[ n \check{z}_j \check{z}_k - \delta_{ij} \Bigg] \right.\\
& +  \Bigg [ n \delta_{ij} \check{z}_k - (n-2) \delta_{ik} \check{z}_j - (n-2) \delta_{jk} \check{z}_i + (n+2)(n-2) \check{z}_i \check{z}_j \check{z}_k \Bigg] \left \langle \zeta_k, \check{z} \right \rangle_\xi \\
& \quad \quad + \left. \left( \zeta_k\right)_j  \Bigg[ n \check{z}_i \check{z}_k - \delta_{ik} \Bigg] - \left( \zeta_k \right)_k \Bigg[ (n-2) \check{z}_i \check{z}_j + \delta_{ij} \Bigg]   \right \} \\
& \quad \quad \quad \quad   \times \bak \ma^2 |\za|^{-n} \big(1 + o(1)\big) \\
\end{aligned} 
\een
where $ \check{z} = \underset{\alpha \to + \infty}{\lim} \frac{\za}{|\za|}$,
\[ C_1(n) = \frac{n^{\frac{n+2}{2}} (n-2)^{\frac{n}{2}} \omega_n}{2^{n+1} (n-1) \omega_{n-1}}, \quad  C_2(n) = - \frac{n^{\frac{n+2}{4}} (n-2)^{\frac{n}{2}} \omega_n}{2^{n+1} (n-1) \omega_{n-1}} ,\]
and where $\omega_d$ stands for the area of the $d$-dimensional sphere in $\mathbb{R}^{d+1}$. 

\medskip
\noindent We show in what follows that the $1$-forms $\Va$ and $\Pak$ respectively appear as the first and second-order terms in the asymptotic development of $\Za$. We start estimating the difference $\left| \Lx \left( \Za - \Va \right)\right|_\xi$.

\begin{claim} \label{claimcont1}
Let $(\da)_\alpha$ be a sequence of positive numbers satisfying:
\ben \label{contda2}
\frac{\ma}{\da} \to 0 \quad \textrm{ and } \quad \da \le \min(\rra, \ma^{\frac{1}{2}}).
\een
There holds, for any sequence $\za \in B_0(3 \da)$:
\ben \label{estlwa1}
\begin{aligned}
\left| \Lx \left( \Za - \Va \right) \right|_\xi(\za)  \le C & \left(  \vea \da^{1-n} + \ma^{n-1} \da^{1 - 2n} + \eta \left( \frac{\ma}{\ta(\za)} \right) \ma \ta(\za)^{1-n} \right) \\
& + o \left( \ve \ta(\za)^{1-n} \right), \\
\end{aligned}
\een
where $\ta(\za)$ is as in \eqref{ta}, $\vea$ is as in \eqref{defeaba} and $\eta$ is a nonnegative, continuous function in $\mathbb{R}$ with $\eta(0) = 0$. As a consequence, there holds
\ben \label{estea1}
\vea  = O \left( \ma^{n-1} \da^{-n} \right) + o(\ma).
\een
\end{claim}

\begin{proof}
Let $(\da)_\alpha$ be a sequence of positive numbers satisfying \eqref{contda2}.
We estimate the difference $\left| \Lx \left( \Za - \Va \right) \right|_\xi(\za)$ using a Green representation formula.
We let $\mathcal{G}_{\alpha,i}$ be the $i$-th Green $1$-form for $\Dx$ with Neumann boundary condition on $B_0(s_\alpha)$, where $s_\alpha$ is as in Claim \ref{claimint} (see Section \ref{Greenfunc}). A Green representation formula for $\Dx$ on $B_0(s_\alpha)$ writes then: for any sequence $\za \in B_0(4 \da)$,
\ben \label{idgreenvec}
\begin{aligned}
 \Lx \left( \Za - \Va \right)_{ij} (\za) = \int_{B_0(s_\alpha)} \mathcal{H}_{ij,\alpha}(\za,y)_p   \Dx \left( \Za - \Va \right)^p(y) dy \\
 + \int_{\partial B_0(s_\alpha)} \mathcal{H}_{ij,\alpha}(\za,y)_p \nu_q \Lx \left( \Za - \Va \right)^{pq}(y) d \sigma(y) \\
\end{aligned}
\een
where we have let, for $x,y \in B_0(s_\alpha)$:
\ben \label{Hija}
\mathcal{H}_{ij,\alpha}(x,y)_p = \partial_i \mathcal{G}_{\alpha,j}(x,y)_p +  \partial_j \mathcal{G}_{\alpha,i}(x,y)_p  - \frac{2}{n}\xi_{ij} \sum_{k=1}^n \partial_k \mathcal{G}_{\alpha,k}(x,y)_p 
\een
and the derivatives are taken with respect to $x$. Since $|\za| \le 4 \da$ and $s_\alpha \ge 5 \da$ we can thus write, using \eqref{sysconf}, \eqref{contconf}, \eqref{eq1forme}, \eqref{estVa}, Claim \ref{claimint} and Claim \ref{vecteurGreen} below, that:
\ben \label{ineqint1}
\left|  \Lx \left( \Za - \Va \right) \right|_\xi(\za) \le C \left( I_1 + I_2 + I_3 + I_4 +  \vea \da^{1-n} + \ma^{n-1} \da^{1-2n}  \right)
\een
where we have let:
\ben \label{claimcont11}
\begin{aligned}
& I_1 = \int_{B_0(6 \da)} |\za - y |^{1-n} |y| \left| \Lx \Za \right|_\xi(y) dy, \\
& I_2 = \vea \int_{B_0(6 \da)} |\za-y|^{1-n} \left| \va^{2^*} - \Ba^{2^*} \right| dy, \\
& I_3 = \left | \int_{B_0(6 \da)} \mathcal{H}_{ij, \alpha}(\za,y)_p \left( \tilde{X}_\alpha(y) - \tilde{X}_\alpha(0) \right)^p \va^{2^*}(y) dy \right|, \\
& I_4 = \int_{B_0(6 \da)} |\za - y|^{1-n} dy . \\
\end{aligned}
\een
We clearly have that
\ben \label{i4}
I_4 \le C \da.
\een
By Claim \ref{claimlocal} there exists $R_\alpha \to + \infty$ such that
\[ \sup_{B_0(R_\alpha \ma)} |\va - \Ba| = o(1)\]
as $\alpha \to + \infty$. Using Claim \ref{claimcontbulle} we can therefore write
\[ \begin{aligned}
I_2 \le &  \vea \int_{B_0(R_\alpha \ma)}  |\za-y|^{1-n} \left| \va - \Ba \right|(y) \Ba^{2^*-1}(y) dy \\
& + O \left( \vea \int_{B_0(6 \da) \backslash B_0(R_\alpha \ma)} |\za-y|^{1-n} \Ba^{2^*}(y) dy \right) \\
\end{aligned} \]
so that in the end we get:
\ben \label{i2}
I_2 = o \left( \vea \ta(\za)^{1-n} \right)
\een
where $\ta(\za)$ is as in \eqref{ta}.
We let $\mathcal{H}_{ij}$ be the analogue of $\mathcal{H}_{ij, \alpha}$ in \eqref{Hija} for the fundamental solution of $\Dx$ in $\mathbb{R}^n$, that is: 
\ben \label{defHij}
\mathcal{H}_{ij}(x,y)_p = \partial_i \mathcal{G}_{j}(x-y)_p +  \partial_j \mathcal{G}_{i}(x-y)_p  - \frac{2}{n}\xi_{ij} \sum_{k=1}^n \partial_k \mathcal{G}_{k}(x-y)_p,
\een
where $\mathcal{G}_i$ is as in \eqref{expG}. Since, by \eqref{convdonnees}, $\tilde{X}_\alpha$ has a limit in $C^2_{loc}(\mathbb{R}^n)$, we can write with \eqref{claimcont11} that:
\[\begin{aligned} 
I_3 &\le  \left| \int_{B_0(s_\alpha)} \mathcal{H}_{ij, \alpha}(\za,y)_p y^k \partial_k \tilde{X}_\alpha(0)^p \va^{2^*}(y) dy \right| \\
& + C \int_{B_0(s_\alpha)} |\za-y|^{1-n} |y|^2 \Ba^{2^*}(y) dy \\
& \le  \left| \int_{B_0(s_\alpha)} \mathcal{H}_{ij, \alpha}(\za,y)_p y^k \partial_k \tilde{X}_\alpha(0)^p \va^{2^*}(y) dy \right| + \ma^2 \ta(\za)^{1-n} .
\end{aligned}  \]
We now separate between two cases. First, assume that $|\za| = O(\ma)$. Then one easily gets that
\ben \label{i31}
 \left| \int_{B_0(s_\alpha)} \mathcal{H}_{ij, \alpha}(\za,y)_p y^k \partial_k \tilde{X}_\alpha(0)^p \va^{2^*}(y) dy \right| \le \ma \ta(\za)^{1-n} .
 \een
 Next assume that 
 \ben \label{i32a}
\frac{|\za|}{\ma} \to + \infty 
 \een
 as $\alpha \to 0$. We write:
 \[\begin{aligned}
  \int_{B_0(s_\alpha)} & \mathcal{H}_{ij, \alpha}(\za,y)_p y^k \partial_k \tilde{X}_\alpha(0)^p \va^{2^*}(y) dy   \\
   &  = \int_{B_0(s_\alpha)} \Big( \mathcal{H}_{ij, \alpha}(\za,y) -  \mathcal{H}_{ij}(\za,y)\Big)_p y^k \partial_k \tilde{X}_\alpha(0)^p \va^{2^*}(y) dy  \\
   &  +  \int_{B_0(s_\alpha)} \mathcal{H}_{ij}(\za,y)_p y^k \partial_k \tilde{X}_\alpha(0)^p \va^{2^*}(y) dy. \\
 \end{aligned} \]
Straightforward computations using \eqref{i32a} show that
\ben \label{i32b}
  \int_{B_0(s_\alpha)} \mathcal{H}_{ij}(\za,y)_p y^k \partial_k \tilde{X}_\alpha(0)^p \va^{2^*}(y) dy = o \left( \ma \ta(\za)^{1-n} \right),
\een
where $\mathcal{H}_{ij}$ is as in \eqref{defHij}. Let now
\[ \Omega_\alpha = \left \{ y \in B_0(s_\alpha) \textrm{ st } |\za-y| \ge \frac{1}{2} |\za| \right \}. \]
On one hand there easily holds, using \eqref{i32a}:
\ben \label{i32c}
\begin{aligned}
& \left| \int_{{}^c \Omega_\alpha}  \Big( \mathcal{H}_{ij, \alpha}(\za,y) -   \mathcal{H}_{ij}(\za,y)\Big)_p y^k \partial_k \tilde{X}_\alpha(0)^p \va^{2^*}(y) dy \right| \\  
&\qquad \le \left( \frac{\ma}{\ta(\za)}\right)^{n-1} \ma \ta(\za)^{1-n} = o \left(\ma \ta(\za)^{1-n} \right).
\end{aligned}
\een
On the other hand, by the definition of $\mathcal{G}_{\alpha,i}$ as in \eqref{solfondNeumann} below there holds, for any $y \in B_0(s_\alpha) \backslash \{ \za \}$: 
\[ \mathcal{G}_{\alpha,i}(\za,y) = \mathcal{G}_{i}(\za-y) + U_{i,\za}^{s_\alpha}(y),\]
where $U_{i,\za}^{s_\alpha}$ is as in \eqref{Uix} below. The scaling arguments developed in Claim \ref{claimrescaling} below show that there holds:
\ben \label{convUia}
 s_\alpha^{n-2} U_{i,\za}^{s_\alpha}(s_\alpha \cdot ) \to {\tilde U}_{0,i} \textrm{ in } C^1(\overline{B_0(1)}), 
 \een
where ${\tilde U}_{0,i}$ satisfies:
\[
\left \{ \bal
& \Dx {\tilde U}_{0,i}(y) = - \sum_{j=1}^{\frac{(n+1)(n+2)}{2}} K_j^0 (\check{z}) K_j^0 (y)  & B_0(1), \\
& \nu^k \Lx \left( {\tilde U}_{0,i} \right)_{kl}(y) = - \nu^k \Lx \left( \mathcal{G}_i (\check{z} - \cdot) \right)_{kl}(y)  & \partial B_0(1), \\
 \eal \right. 
 \]
where we have let $\check{z} = \underset{\alpha \to + \infty}{\lim} \frac{\za}{s_\alpha} \in \overline{B_0(1)}$ and where $\left( K_j^0 \right)_{1 \le j \le (n+1)(n+2)/2}$ is an orthonormal basis for the $L^2$-scalar product of the set $K_1$ of conformal Killing $1$-forms in $B_0(1)$ defined in \eqref{Killingspace}. In the end, using \eqref{convUia}, \eqref{dersolNeumann} below and Lebesgue's dominated convergence theorem there holds:
\ben \label{i32d}
\begin{aligned} 
& \ma^{-1} \beta_{\alpha,k}^{-1} |\za|^{n-1} \Bigg| \int_{\Omega_\alpha}  \Big( \mathcal{H}_{ij, \alpha}(\za,y) -  \mathcal{H}_{ij}(\za,y)\Big)_p y^k \partial_k \tilde{X}_\alpha(0)^p \va^{2^*}(y) dy \Bigg|  \\
& \qquad \qquad = \left| \check{z} \right|^{n-1} \int_{\RR^n}  \left ( \partial_i {\tilde U}_{0,j} +  \partial_j {\tilde U}_{0,i} - \frac{2}{n} \xi_{ij} \sum_{k=1}^n \partial_k {\tilde U}_{0,k}\right)_p(0) \\
& \qquad \qquad \qquad \qquad \times  y^k \zeta_k ^p \left( 1 + \frac{f_0(x_0)}{n(n-2)} |y|^2 \right)^{-n} dy + o(1) \\
& \qquad \qquad = o(1), 
\end{aligned} 
\een
where $\beta_{\alpha,k}$ and $\zeta_k$ are as in \eqref{defeaba} and \eqref{defz}. In \eqref{i32d}, the notation $(\cdot)_p$ denotes the $p$-th coordinate of the $1$-form considered. Gathering \eqref{i32b}, \eqref{i32c} and \eqref{i32d} there thus holds, if $|\za| >>\ma$:
\ben \label{i32}
 \int_{B_0(s_\alpha)}  \mathcal{H}_{ij, \alpha}(\za,y)_p y^k \partial_k \tilde{X}_\alpha(0)^p \va^{2^*}(y) dy = o \left( \ma \ta(\za)^{1-n} \right).
 \een
 In the end, \eqref{i31} and \eqref{i32} together combine by writing that:
 \ben \label{i3}
 I_3 = \eta \left( \frac{\ma}{\ta(\za)}\right) \ma \ta(\za)^{1-n} 
 \een
 for some continuous bounded $\eta$ with $\eta(0) = 0$. We finally compute $I_1$. A H\"older inequality along with \eqref{contbulles} gives, for any $p > 2$:
 \[ \begin{aligned}
  I_1 & \le \left( \int_{B_0(6 \da)} |\za-y|^{2-n} \frac{\left| \Lx \Za \right|_\xi^2}{\va^{2^*+1}}(y) dy \right)^{\frac{1}{p}} \\
 & \times \left( \int_{B_0(6 \da)} |\za-y|^{\frac{n-2 - p(n-1)}{p-1}} |y|^{\frac{p}{p-1}} \left| \Lx \Za \right|_\xi(y)^{\frac{p-2}{p-1}} \Ba(y)^{\frac{2^*+1}{p-1}} dy \right)^{\frac{p-1}{p}}.
\end{aligned}
  \]
If $2 < p < 2n-2$, the second integral can be estimated with \eqref{contfaible} as:
\[ 
\bal
\int_{B_0(6 \da)} |\za-y|^{\frac{n-2 - p(n-1)}{p-1}} |y|^{\frac{p}{p-1}} \left| \Lx \Za \right|_\xi(y)^{\frac{p-2}{p-1}} \Ba(y)^{\frac{2^*+1}{p-1}} dy \\
\le \ma^{\frac{3n-2}{2(p-1)} + \frac{p+2-2n}{p-1}} \ta(\za)^{\frac{n-2-p(n-1)}{p-1}},
\eal
\]
so that using \eqref{claimint1} to estimate the first integral yields in the end:
\ben \label{i11}
 \begin{aligned}
I_1 & \le C \ma \ta(\za)^{1-n}
\end{aligned}
\een
for some positive constant $C$. Gathering \eqref{i4}, \eqref{i2}, \eqref{i3} and \eqref{i11} in \eqref{ineqint1} and using \eqref{estVa} we obtain, since $\ma \le \ta(\za) \le O(\da)$ and by \eqref{contda2}, that for $\za \in B_0(4 \da)$:
\ben \label{paraordre1}
\begin{aligned}
 \left|  \Lx   \Za  \right|_\xi(\za) & \le C \left(  \vea \da^{1-n} + \ma^{n-1} \da^{1-2n}  \right)  + C  \vea \ta(\za)^{1-n}  \\ 
& \qquad +  C \ma \ta(\za)^{1-n} .
\end{aligned}
\een
We now use this to refine the estimate of $I_1$, where $I_1$ is as in \eqref{claimcont11}. We assume that the sequence $\za$ belongs to $B_0(3 \da)$. We therefore write:
\[ I_1 \le \int_{B_0(4 \da)} |\za - y |^{1-n} |y| \left| \Lx \Za \right|_\xi(y) dy + \da^{2-n} \int_{B_0(6 \da) \backslash B_0(4 \da)} \left| \Lx \Za \right|_\xi(y) dy,\]
and we use estimate \eqref{paraordre1} to compute the first integral and \eqref{claimint1plus} to compute the second one. In the end these computations lead to:
\[ 
\bal
 I_1 & \le C \ma^{n-1} \da^{3-2n}  + C \big(\ma + \vea \big) \ta(\za)^{3-n} \ln \left( 2 + \frac{\da}{\ta(\za)} \right) \\ 
\eal
\]
which gives, using \eqref{contda2} and since $\ta(\za) \le 4 \da$:
\ben \label{i12}
\begin{aligned}
I_1 \le C \ma^{n-1} \da^{3-2n}  + o \left( \ma \ta(\za)^{1-n}\right) + o \left( \vea \ta(\za)^{1-n}\right).
 \end{aligned}
 \een
Note that the factor $\ln\left(2 + \frac{\da}{\ta(\za)}\right)$ in the computations above has to be taken into account only in dimension $3$. In the end, gathering \eqref{i4}, \eqref{i2}, \eqref{i3} and \eqref{i12} in \eqref{ineqint1} yields, for any $\za \in B_0(3 \da)$:
\ben \label{estordre1}
\begin{aligned}
 \left|  \Lx \left( \Za - \Va \right) \right|_\xi & (\za) \le C \Big( \ma^{n-1} \da^{1-2n} + \vea \da^{1-n} \Big) \\
& + o \left( \vea \ta(\za)^{1-n} \right)  + \eta \left(\frac{\ma}{\ta(\za)} \right) \ma \ta(\za)^{1-n} \\
\end{aligned}
\een
where $\ta(\za)$ is as in \eqref{ta} and $\vea$ is as in \eqref{defeaba}. This proves \eqref{estlwa1}.

\medskip
\noindent We now use this pointwise control to obtain an estimate on $\vea$. Let $0<\gamma <1$. We write that:
\[ 
\begin{aligned}
\int_{B_0(2 \gamma \da) \backslash B_0(\gamma \da) }  \left| \Lx \Va \right|_\xi^2 (y) dy \le   2 \int_{B_0(2 \gamma \da) \backslash B_0(\gamma \da) }  \left| \Lx \left( \Za - \Va \right) \right|_\xi^2 (y) \\
 + 2 \int_{B_0(2 \gamma \da) \backslash B_0(\gamma \da) }  \left| \Lx \Za \right|_\xi^2 (y) dy  .
\end{aligned}
\] 
Thanks to \eqref{contda2} we can estimate the left-hand side with \eqref{asymptova}, while for the right-hand side, we estimate the first integral using \eqref{estordre1} and the second one using \eqref{claimint1plus}. This yields:
\[ \begin{aligned}
& \vea^2 \le  C \gamma^{2n-2} \vea^2 
+  C \gamma^{-2n} \ma^{2n-2} \da^{-2n} + o(\vea^2) + o(\ma^2), \\
\end{aligned}\]
where $C$ is some constant that does not depend on $\alpha$ nor on $\gamma$. Up to choosing $\gamma$ small enough we therefore obtain \eqref{estlwa1}.
 \end{proof}

\medskip

\noindent Note, as the proof of Claim \ref{claimcont1} shows, that one could have obtained more easily a less precise estimate on $I_3$ defined in \eqref{claimcont11}: indeed, requiring only the $X_\alpha$ to converge to $X_0$ in $C^1(M)$ would yield, for any $\za \in B_0(3 \da)$:
\be
 I_3 \le \ma \ta(\za)^{1-n} 
 \ee
instead of \eqref{i3}. One would then subsequently obtain that
\ben \label{autreestea}
 \vea \le C \left( \ma^{n-1} \da^{-n} + \ma \right) 
 \een
in \eqref{estea1}. This estimate is actually enough to conclude in dimensions $3$ and $4$ if the convergence of the $X_\alpha$ is in $C^1(M)$, see Claims \ref{controlera} and \ref{raegalerhoa} below. In dimensions $n \ge 5$ the $C^2(M)$ convergence is needed to push the analysis one step further but, as soon as $n \not = 6$, \eqref{autreestea} remains precise enough to provide a suitable starting point to improve the estimates on $\Lx \Za$ up to second order. It turns out that if $n = 6$ estimate \eqref{autreestea} fails to be sufficient to perform the required second-order approximation of $\Lx \Za$ and that \eqref{estea1} is needed. See the computations leading to equation \eqref{estordre2b} below for more details. 

\medskip

\noindent In high dimensions Claim \ref{claimcont1} is not enough to conclude. We now push the analysis one order further and estimate the difference $\left| \Lx \left(\Za - \Va - \sum_{k=1}^n \Pak \right) \right|_\xi$. This time, unlike in Claim \ref{claimcont1}, controlling $\va$ by $\Ba$ with Claim \ref{claimcontbulle} is not enough, and the precision of the approximation of $\va$ by the bubble $B_\alpha$ comes into play. Such a precision is unknown at this stage. 
In the next Claim we obtain a simultaneous description of the second order terms in the expansion of both $\va$ and $\Lx \Za$. We simultaneously carry out the ping-pong analysis to handle these unknown second-order terms. As stated in the previous paragraph, and as we will see in the last step of the proof of Proposition \ref{propoblow} -- namely Claim \ref{raegalerhoa} -- Claim \ref{claimcont1} is enough to conclude in dimensions $3$ and $4$. For the next Claim we will thus assume that $n \ge 5$.

\begin{claim} \label{claimcont2}
Assume $n \ge 5$. Let $(\ua, \Wa)_\alpha$ be a sequence of solutions of \eqref{systype} such that \eqref{convdonnees} and \eqref{contradiction} hold. Let $(\xa)_\alpha$ and $(\ra)_\alpha$ be two sequences satisfying \eqref{pc1} and \eqref{p2} and define $\va$ and $\Za$ as in \eqref{blowcarteconf}. Let $(\da)_\alpha$ be a sequence of positive numbers satisfying $\da = O(\rra)$ and:
\ben \label{contda3}
\frac{\da}{\ma} \to + \infty \quad \textrm{ and } \quad \left \{
\begin{aligned}
& \da = O \left( \ma^{\frac{1}{2}} \right) \textrm{ if } n=5, \\
& \da = O \left( \ma^{\frac{n-4}{n-2}}\right) \textrm{ if } n \ge 6 .\\
\end{aligned}
\right.
\een
Then there holds
\ben \label{estea2}
\vea + \ma^2 \da^{-1} \ba \le \frac{\ma^{n-1}}{\da^{n}},
\een
where $\vea$ and $\ba$ are as in \eqref{defeaba}, and for any $\za \in B_0(2 \da)$ we have:
\ben \label{estlwa2}
\Big| \Lx \Za \Big|_\xi(\za) \le C  \left( \frac{\ma}{\da} \right)^{n-1} \ta(\za)^{-n},
\een
where $\ta(\za)$ is as in \eqref{ta}.
\end{claim}

\begin{proof}
The proof proceeds in several steps.

\medskip

\noindent \textbf{Step $1$:} \emph{there holds, for any $\za \in B_0(2 \da)$:
\ben \label{estordre2a}
\begin{aligned}
\left| \Lx \Big(\Za - \Va - \sum_{k=1}^n \Pak \Big) \right|_\xi  (\za) & \le C \Big( \ma^{n-1} \da^{1-2n} + \vea \da^{1-n} + \ma^2 \ba \da^{-n} \Big) \\
& + \Big( \vea + \ma \ba \Big) \ma^{\pui} \Vert \Ra \Vert_\infty \ta(\za)^{1-n} \\
& + \ma^2 \ta(\za)^{1-n}  + o \left( \vea \ta(\za)^{1-n}\right), \\
\end{aligned}
\een
where we have let, in $B_0(8 \ra)$:
\ben \label{defRa}
R_\alpha = \va - \Ba,
\een
and the $L^\infty$-norm of $\Ra$ is taken over $B_0(2 \da)$.}

\medskip
\noindent  To prove \eqref{estordre2a}, we write once again a Green representation formula for $\Dx$ in $B_0(3 \da)$. By \eqref{eq1forme} and \eqref{sysconf}, there holds on $B_0(3 \da)$:
\ben
\bal
& \Dx \left(\Za - \Va - \sum_{k=1}^n \Pak \right)_i(y)  =   2^*  \xi^{kl} \partial_k \ln \vpa  \left( \Lx \Za \right)_{li} \\
& \qquad + \va^{2^*}(y) \left( \tilde{X}_\alpha(y) - \tilde{X}_\alpha(0) - y^p \partial_p \tilde{X}_\alpha(0)  \right)_i + \left( \tilde{Y}_\alpha \right)_i(y) \\
& \qquad + \left( \va^{2^*}(y) - \Ba^{2^*}(y) \right)\left( \tilde{X}_\alpha(0) +  y^p \partial_p \tilde{X}_\alpha(0) \right)_i, \\
\eal
\een
where $\vpa$ is as in \eqref{propvpa} so that we can write, for any $\za \in B_0(2 \da)$:
\ben \label{claimcont21}
\left| \Lx \Big(\Za - \Va - \sum_{k=1}^n \Pak \Big) \right|_\xi  (\za) \le C \left(J_B + J_1 + J_2 + J_3 + J_4 \right),
\een
where we have let:
\ben \label{claimcont22}
\begin{aligned}
& J_B = \int_{\partial B_0(\frac{5}{2} \da)} |\za-y|^{1-n} \left| \Lx \left(\Za - \Va - \sum_{k=1}^n \Pak \right) \right|_\xi (y) d\sigma (y), \\
& J_1 = \int_{B_0(3 \da)} |\za - y |^{1-n} |y| \left| \Lx \Za \right|_\xi(y) dy, \\
& J_2 =\int_{B_0(3 \da)} |\za-y|^{1-n} |y|^2 \va^{2^*}(y) dy , \\
& J_3 = \int_{B_0(3 \da)} \left( \vea + \ba |y| \right) |\za-y|^{1-n} \left| \va^{2^*} - \Ba^{2^*} \right|(y) dy, \\
& J_4 = \int_{B_0(3 \da)} |\za - y|^{1-n} dy . \\
\end{aligned}
\een
We estimate $J_B$ using Claim \ref{claimint} (with a radius $\frac{1}{2} \da$): there holds that
\ben \label{JB}
J_B \le C \left( \ma^{n-1} \da^{1-2n} + \vea \da^{1-n} + \ma^2 \ba \da^{-n} \right),  
\een
that
\ben \label{J4}
J_4 \le C \da,
\een
and that
\ben \label{J2}
J_2 \le C \ma^2 \ta(\za)^{1-n},
\een
for some $C > 0$ independent of $\alpha$. Using \eqref{estlwa1} and \eqref{estVa} it is easily seen that

\ben \label{J1}
\bal
J_1 \le 
 o \left( \vea \ta(\za)^{1-n}\right) + C \ma^2 \ta(\za)^{1-n} + C \ma^{n-1} \da^{1-2n},
\eal
\een
where we have used that $\da^2 = O(\ma)$ for any $n \ge 3$ due to \eqref{contda3}. Finally there easily holds:
\ben \label{J3}
J_3 \le C \left( \vea + \ma \ba \right) \ma^\pui \Vert \Ra \Vert_\infty  \ta(\za)^{1-n}.
\een
Gathering \eqref{JB}, \eqref{J4}, \eqref{J2}, \eqref{J1} and \eqref{J3} in \eqref{claimcont21} we therefore obtain \eqref{estordre2a}. Clearly there holds $\ma^\pui \Vert \Ra \Vert_\infty \le 1$, but this control is not precise enough for our purposes. We now use the improved estimate \eqref{estordre2a} to obtain a better control on $\Ra$.

\medskip
\noindent \textbf{Step $2$:} \emph{ there holds
\ben \label{Rainfty}
\Vert \Ra \Vert_\infty \le C M_\alpha
\een
for some positive $C$, where
\ben \label{Ma}
M_\alpha = \ma^\pui \da^{2-n}  + \ma^{2 - \frac{n}{2}} + \ma^{5 - \frac{3n}{2}}\da^{n+2} + \vea^2 \ma^{1 - \frac{3n}{2}} \da^{n+2} + \ma^{5 - \frac{3n}{2}} \da^n \ba^2
\een
for all $\alpha$.}

\medskip
\noindent By definition of $\Ra$ as in \eqref{defRa} and by \eqref{sysconf}, it satisfies :
\ben \label{equaRa}
\bal
 \triangle_g \Ra = \left( \tilde{f}_\alpha - \tilde{f}_\alpha(0) \right) \va^{2^*-1} +  \tilde{f}_\alpha(0) \left( \va^{2^*-1} - \Ba^{2^*-1} \right) - \tilde{h}_\alpha \va \\
+ \left( \tilde{b}_\alpha + \left| \tilde{U}_\alpha + \vpa^{2^*} \Lx \Za \right|_\xi^2 \right) \va^{-2^*-1} .\\
\eal
\een
Let $(\ya)_\alpha$ be a sequence of points of $B_0(2 \da)$. First, if $\da \le |\ya| \le 2 \da$,  there holds by \eqref{contbulles}:
\ben \label{ya2da}
 \left| \Ra(\ya) \right| = O \left(\ma^\pui \da^{2-n} \right) = O (M_\alpha),
\een
where $M_\alpha$ is as in \eqref{Ma}. Assume now that $\ya \in B_0(\da)$. A Green formula for $\triangle_\xi$ in $B_0(2 \da)$ at $\ya$ gives, using \eqref{contda3}:
\ben \label{estRa}
\left| \Ra(\ya) \right|  \le C \left( \mathcal{I}_1 + \mathcal{I}_2 + \mathcal{I}_3 + \mathcal{I}_4 + \mathcal{I}_5 \right)
\een
for some positive $C$, where we have let:
\ben \label{estRaint}
\bal
&\mathcal{I}_1 = \int_{B_0(2 \da)} |\ya-y|^{2-n} \left| \tilde{f}_\alpha(y) - \tilde{f}_\alpha(0)\right|\va^{2^*-1}(y) dy, \\
& \mathcal{I}_2 =  \int_{B_0(2 \da)} |\ya-y|^{2-n}  \va^{2^*-2}(y) \left|\Ra(y) \right| dy, \\
& \mathcal{I}_3 =  \int_{B_0(2 \da)} |\ya-y|^{2-n} \va(y) dy, \\
& \mathcal{I}_4 = \int_{B_0(2 \da)} |\ya-y|^{2-n} \left( 1 + \left| \Lx \Za \right|_\xi^2 (y) \right) dy, \\
& \mathcal{I}_5 =  \int_{\partial B_0(2 \da)} |\ya-y|^{2-n} \left| \Ra (y) \right| d\sigma (y) .\\
\eal
\een
Using Claim \ref{claimcontbulle} and \eqref{contda3} there holds that:
\ben \label{RaI5}
\mathcal{I}_5 \le C \ma^\pui \da^{2-n},
\een
that
\ben \label{RaI1}
\mathcal{I}_1 \le C \ma^{\frac{n}{2}} \ta(\ya)^{2-n},
\een
that
\ben \label{RaI3}
\mathcal{I}_3 \le C \ma^\pui \ta(\ya)^{4-n},
\een
and that
\ben \label{RaI2}
\mathcal{I}_2 \le C \left( \frac{\ma}{\ta(\ya)} \right)^2 \NRa.
\een
Using now \eqref{estordre2a}, \eqref{estVa} and \eqref{estPak} there holds:
\ben \label{RaI41}
\bal
\mathcal{I}_4   \le & C \Bigg(\ma^{\pui} \da^{2-n} + \ma^{5 - \frac{3n}{2}} \da^{n+2} \\
 & + \vea^2 \ma^{1 - \frac{3n}{2}} \da^{n+2} + \ma^{5 - \frac{3n}{2}} \da^n \ba^2 + \ma^{1 - \frac{n}{2}} \NRa^2 \da^{n+2} \ba^2 \Bigg).
\eal
\een
Note that, by definition of $\Ra$ as in \eqref{defRa} and by Claims \ref{claimlocal} and \ref{claimcontbulle} there holds:
\ben \label{NRaoma}
\NRa = o \left( \ma^{1 - \frac{n}{2}} \right).
\een
In particular, combining \eqref{NRaoma} with \eqref{contda3} and \eqref{defeaba} we obtain that for any $n \ge 6$:
\[  \ma^{1 - \frac{n}{2}} \NRa^2 \da^{n+2} \ba^2 = o \left( \NRa \right).  \]
This gives, in \eqref{RaI41}:
\ben \label{RaI42}
\bal
\mathcal{I}_4   & \le  C \Bigg(\ma^{\pui} \da^{2-n} + \ma^{5 - \frac{3n}{2}} \da^{n+2} + \vea^2 \ma^{1 - \frac{3n}{2}} \da^{n+2} + \ma^{5 - \frac{3n}{2}} \da^n \ba^2 \Bigg) \\
 &  \qquad+ o \left( \NRa \right),
 \eal
\een
so that gathering \eqref{RaI5}, \eqref{RaI1}, \eqref{RaI3}, \eqref{RaI2} and \eqref{RaI42} in \eqref{estRa} we obtain:
\ben \label{estRaopt}
\left| \Ra(\ya) \right| \le C \left( \frac{\ma}{\ta(\ya)}\right)^2 \NRa  + C M_\alpha + o \left( \NRa \right),
\een
where $M_\alpha $ is as in \eqref{Ma}. Note that since $\ta(\ya) \ge \ma$ there holds:
\ben \label{inegp}
 \left( \frac{\ma}{\ta(\ya)}\right)^2 \le C \left( \frac{\ma}{\ta(\ya)}\right)^p \textrm{ for any } 1 \le p \le 2. 
 \een
In particular the following estimate holds for any $1 \le p \le 2$:
\ben \label{estRap}
\left| \Ra(\ya) \right| \le C \left( \frac{\ma}{\ta(\ya)}\right)^p \NRa  + C M_\alpha + o \left( \NRa \right).
\een
Plugging \eqref{estRap} into \eqref{estRaint} we can refine the estimate on $\mathcal{I}_2$. We obtain, as long as $p < n-4$, that:
\ben \label{estRapimp}
\bal
\mathcal{I}_2 & \le C \left( \frac{\ma}{\ta(\ya)}\right)^2 \bigg( M_\alpha + o (\NRa) \bigg) + \left( \frac{\ma}{\ta(\ya)}\right)^{p+2} \NRa \\
& \le C M_\alpha + o ( \NRa) + \left( \frac{\ma}{\ta(\ya)}\right)^{p+2} \NRa.
\eal
\een
Using \eqref{estRapimp} in \eqref{estRa} it is easily seen that the estimate on $\Ra$ improves accordingly. Combining \eqref{estRapimp} and \eqref{estRap} we therefore obtain by induction that there exists some $p_0 > n-4$ such that 
\be
\left| \Ra(\ya) \right| \le C \left( \frac{\ma}{\ta(\ya)}\right)^{p_0} \NRa  + C M_\alpha + o \left( \NRa \right).
\ee
Plugging once again this estimate into \eqref{estRaint} and using this time that $p_0 > n-4$ we obtain the following estimate on $\mathcal{I}_2$:
\be
\mathcal{I}_2 \le C \left( \frac{\ma}{\ta(\ya)}\right)^{n-2} + C M_\alpha + o (\NRa),
\ee
which, combined with \eqref{RaI5}, \eqref{RaI1}, \eqref{RaI3} and \eqref{RaI42} in \eqref{estRa} and with \eqref{ya2da} gives that:
\ben \label{estRaopt2}
\left| \Ra(\ya) \right| \le C \left( \frac{\ma}{\ta(\ya)}\right)^{n-2} \NRa  + C M_\alpha + o \left( \NRa \right)
\een
for any sequence of points $\ya \in B_0(2 \da)$. To prove \eqref{Rainfty} we now proceed by contradiction. We assume that $\Ra$ is not identically zero and further assume that:
\ben \label{contraRa}
\underset{\alpha \to + \infty}{\lim} \frac{\NRa}{M_\alpha} \to + \infty. 
\een
We now let $\ya$ be such that
\ben \label{defya} 
| \Ra(\ya)| = \NRa.
\een
Then \eqref{estRaopt2} shows that 
\ben \label{maxdistfinie}
|\ya| = O(\ma),
\een
so that in particular $\frac{1}{C} \ma \le \ta(\ya) \le C \ma$ for some positive $C$. We introduce the functions 
\ben \label{tildeRa}
\tRa(x) = \ma^\pui \Ra(\ma x)  \quad \textrm{ and } \quad  \tva(x) = \ma^\pui \va(\ma x),
\een
which are well-defined in $B_0(2 \frac{\rra}{\ma})$. By \eqref{equaRa} $\NtRa ^{-1} \tRa$ satisfies:
\ben \label{equatRa}
\bal
& \triangle_g \left( \NtRa^{-1} \tRa \right) = \left( \tilde{f}_\alpha(\ma \cdot) - \tilde{f}_\alpha(0) \right) \NtRa^{-1} \tva^{2^*-1} \\
& \quad  +  \tilde{f}_\alpha(0) \NtRa^{-1} \left( \tva^{2^*-1} - \tBa^{2^*-1} \right) 
 - \ma^2 \tilde{h}_\alpha(\ma \cdot) \tva \NtRa^{-1} \\
& \quad + \left( \ma^{2n} \tilde{b}_\alpha(\ma \cdot) + \left| \ma^n \tilde{U}_\alpha +\ma^n  \vpa^{2^*} \Lx \Za \right|_\xi^2 \right)(\ma \cdot) \NtRa^{-1} \tva^{-2^*-1} .\\
\eal
\een
With \eqref{tildeRa} assumption \eqref{contraRa} becomes:
\ben \label{contratRa}
\NtRa > > \left( \frac{\ma}{\da}\right)^{n-2} + \ma + \ma^{4-n} \da^{n+2} + \vea^2 \ma^{-n} \da^{n+2} + \ma^{4-n} \da^n \ba^2,
\een
so that there holds :
\ben \label{convresc1}
\bal
 \left| \tilde{f}_\alpha(\ma y) - \tilde{f}_\alpha(0) \right|& \NtRa^{-1} 
 + \ma^2 \tilde{h}_\alpha(\ma y) \NtRa^{-1} \\   
 & + \ma^{2n} \left( \tilde{b}_\alpha(\ma y) +  \left| \ma^n \tilde{U}_\alpha(\ma y ) \right|_\xi \right)^2 \NtRa^{-1}  \to 0 \\ 
\eal
\een
in $C^0_{loc}(\RR^n)$.
Independently, by Claim \ref{claimlocal}, there holds for some positive $C$: 
\ben \label{convresc2}
\left| \NtRa^{-1} \left( \tva^{2^*-1}(y) - \tBa^{2^*-1}(y) \right) \right| \le C,
\een
and using \eqref{estordre2a} we obtain:
\ben \label{convresc3}
\bal
\ma^{2n} \Big| \Lx \Za \Big|^2_\xi(\ma y) \NtRa^{-1} & \le \NtRa^{-1} \left( \vea^2 \ma^2 + \ma^4 \ba^2 + \left( \frac{\ma}{\da}\right)^{4n-2} + \ma^6 \right) \\
& \qquad + \ba^2 \ma^{3 + \frac{n}{2}} \NRa.
 \eal
\een
Since $\ma^{3 + \frac{n}{2}} \NRa = O( \ma^4)$, \eqref{convresc3} becomes, using \eqref{contratRa} and \eqref{contda3}:
\ben \label{convresc4}
\ma^{2n} \Big| \Lx \Za \Big|_\xi(\ma y) \to 0
\een
in $C^0_{loc}(\RR^n)$ as $\alpha \to + \infty$. Gathering \eqref{convresc1}, \eqref{convresc2} and \eqref{convresc4} in \eqref{equatRa} and by standard elliptic theory we obtain that $\NtRa^{-1} \tRa \to \hat{R}_0$ in $C^1_{loc}(\RR^n)$, where $\hat{R}_0$ is a solution of
\ben \label{eqlinearise}
\triangle_\xi \hat{R}_0 = \left(2^*-1 \right) f_0(x_0) U^{2^*-2} \hat{R}_0,
\een
where $x_0 = \underset{\alpha \to + \infty}{\lim} \xa$ and $U$ is as in Claim \ref{claimlocal}. In addition, estimate \eqref{estRaopt2} gives, with assumption \eqref{contraRa} and letting $\alpha \to + \infty$:
\[ \left| \hat{R}_0 (x) \right| \le C \left( 1 + |x|\right)^{2-n} \textrm{ for any } x \in \RR^n.\]
In particular, $U^{2^*-2} \hat{R}_0^2 \in L^1(\RR^n)$. Since $\hat{R}_0$ satisfies \eqref{eqlinearise}, $\hat{R}_0$ is given by the Bianchi-Egnell classification result \cite{BianchiEgnell}. 
Since $\xa$ is a critical point of $\va$ and by the definition of $\ma$ and $\Ba$ as in \eqref{defma} and \eqref{bullea} we clearly have that $\hat{R}_0(0) = 0$ and $\nabla \hat{R}_0(0) = 0$, and this implies then that $\hat{R}_0 \equiv 0$, see \cite{BianchiEgnell}. However, if we let $\hat{y}_\alpha = \frac{\ya}{\ma}$, where $\ya$ is as in \eqref{defya}, there holds by \eqref{maxdistfinie} that $\hat{y}_\alpha = O(1)$, so that $\hat{y}_\alpha \to \hat{y}_0 \in \RR^n$ with $\hat{R}_0 (\hat{y}_0) = 1$. This contradicts \eqref{contraRa} and concludes the proof of \eqref{Rainfty}.

\medskip
\noindent \textbf{Step $3$:} \emph{Conclusion.} 
We now plug \eqref{Rainfty} into estimate \eqref{estordre2a} and get:
\ben \label{estordre2b}
\bal
& \left| \Lx \Big(\Za - \Va - \sum_{k=1}^n \Pak \Big) \right|_\xi   (\za)  \le C \Big( \ma^{n-1} \da^{1-2n} + \vea \da^{1-n} + \ma^2 \ba \da^{-n} \Big) \\
& \qquad \qquad + o \Big( \vea \ta(\za)^{1-n}\Big) + o \Big( \ma^2 \ba \ta(\za)^{-n} \Big)  + C \ma^2 \ta(\za)^{1-n} .\\
\eal
\een
To obtain \eqref{estordre2b} we crucially used estimate \eqref{estea1} of Claim \ref{claimcont1} along with \eqref{contda3}.  
To estimate $\vea$ and $\ba$ we proceed as before. Let $0<\gamma <1$. We write with \eqref{claimint1plus} and \eqref{contda3} that:
\ben \label{inegtri}
\begin{aligned}
& \int_{B_0(2 \gamma \da) \backslash B_0(\gamma \da) }  \left| \Lx \Va + \sum_{k=1}^n \Lx \Pak  \right|_\xi^2 (y) dy \le  C \ma^{2n-2} \da^{2-3n} \gamma^{2-3n} \\
& \qquad  + C \int_{B_0(2 \gamma \da) \backslash B_0(\gamma \da) }  \left| \Lx \left( \Za - \Va - \sum_{k=1}^n \Pak \right) \right|_\xi^2 (y) .\\
\end{aligned}
\een 
On one side, using \eqref{asymptova} and \eqref{asymptopak}, there holds, by \eqref{contda3}:
\[  \int_{B_0(2 \gamma \da) \backslash B_0(\gamma \da) }  \left| \Lx \Va + \sum_{k=1}^n \Lx \Pak  \right|_\xi^2 (y) dy \ge \frac{1}{C} \Big( \gamma^{2-n} \vea^2 \da^{2-n} + \gamma^{-n} \ba^2 \ma^4 \da^{-n}\Big)\]
for some positive constant $C$ that does not depend on $\gamma$. On the other side there holds, using \eqref{estordre2b}:
\[ 
\bal
  \int_{B_0(2 \gamma \da) \backslash B_0(\gamma \da) }  & \left| \Lx \left( \Za - \Va - \sum_{k=1}^n \Pak \right) \right|_\xi^2 (y)  \\
 & \le C \Big( \gamma^n \ma^{2n-2} \da^{2 - 3n} + \gamma^n \vea^2 \da^{2-n}  + \gamma^n \ba^2 \ma^4 \da^{-n} + \gamma^{2-n} \ma^4 \da^{2-n}   \Big).
\eal
\]
Gathering the latter two results in \eqref{inegtri} yields \eqref{estea2}, up to choosing $\gamma$ small enough, since there always holds:
\[ \ma^2 = O \left( \ma^{n-1} \da^{-n} \right)\]
under assumption \eqref{contda3}. Finally, \eqref{estlwa2} clearly follows from \eqref{estVa}, \eqref{estPak}, \eqref{estea2}, \eqref{estordre2b} and \eqref{contda3} and this concludes the proof of the Claim.
\end{proof}

\noindent The next step of the proof consists in showing that estimate \eqref{estlwa2} actually holds up to the radius $\rra$. By Claim \ref{claimcont2} above it is therefore enough to show that $\rra$ satisfies \eqref{contda3}.

\begin{claim} \label{controlera}
Assume that \eqref{convdonnees} holds. Let $(\ua, \Wa)_\alpha$ be a sequence of solutions of \eqref{systype} satisfying \eqref{contradiction} and $(\xa)_\alpha$, $(\rho_\alpha)_\alpha$ be two sequences satisfying \eqref{pc1} and \eqref{p2}. We assume that the assumptions of Theorem \ref{Th1} hold. Then we have:
\begin{itemize}
\item $\rra = O \left( \ma^{\frac{1}{2}}\right)$ if  $3 \le n \le 5$ and $b_0 \not \equiv 0$, 
\item $ \rra = O \left( \ma^{\frac{n-1}{n}}\right)$ if $ n \ge 3$ and $X_0(x_0) \not = 0$,  
\item $\rra = O \left( \ma^{\frac{n-2}{n-1}} \right)$ if $n \ge 6$ and $\nabla f_0(x_0) \not = 0 $,
\item $\rra = O \left( \ma^{\frac{n-4}{n-2}} \right) $ if  $n \ge 6$ and 
\[ \frac{n-2}{4(n-1)} R(g)(x_0) -h_0(x_0) - C(n) \frac{\triangle_g f_0(x_0)}{f_0(x_0)} > 0 ,\]
\end{itemize}
where $x_0 = \underset{\alpha \to + \infty}{\lim \xa}$ and $C(n)$ is some positive constant depending only on $n$ given by \eqref{defC(n)} below. As a consequence, Claims \ref{claimcont1} and \ref{claimcont2} hold for $\da = \rra$ and there holds:
\ben \label{estoptlwa}
\left| \Lx \Za \right|_\xi(\za) \le C \left( \frac{\ma}{\rra} \right)^{n-1} \ta(\za)^{-n}
\een
for all $\alpha$ and all $\za \in B_0(2 \rra)$.
\end{claim}

\begin{proof}
We first assume that $3 \le n \le 5$ and $b_0 \not \equiv 0$. Let $G_\alpha$ be the Green function of the operator $\triangle_g + h_\alpha$ in $M$. Since $\triangle_g + h_0$ is coercive there holds, for any $x \in M$ and for some positive constant $C$:
\[ \ua(x) \ge \frac{1}{C} \int_M \left( f_\alpha \ua^{2^*-1} + \frac{b_\alpha}{\ua^{2^*+1}}\right)dv_g.\]
We have
\[  f_\alpha \ua^{2^*-1} + \frac{b_\alpha}{\ua^{2^*+1}} \ge \frac{2 \cdot 2^*}{2^*-1} b_\alpha \left( \frac{(2^*+1)b_\alpha}{(2^*-1) f_\alpha}\right)^{- \frac{2^*+1}{2 \cdot 2^*}}, \]
and since $b_0 \not \equiv 0$ there thus exists some positive $\ve_0$ independent of $\alpha$ such that:
\ben \label{minorua}
\ua \ge \ve_0.
\een
By the definition of $\rra$ as in \eqref{defra} we therefore obtain
\ben \label{contra345}
 \rra = O \left( \ma^{\frac{1}{2}} \right) .
 \een
If $3 \le n \le 4$, \eqref{estoptlwa} is a consequence of Claim \ref{claimcont1} with \eqref{contra345} since in these dimensions \eqref{contra345} implies $\ma = O \left( \ma^{n-1} \rra^{-n} \right)$. If $n = 5$, \eqref{estoptlwa} is a consequence of Claim \ref{claimcont2} and of \eqref{contra345}.

\medskip
\noindent Assume next that $n \ge 3$ and $X_0(x_0) \not = 0$. This means that $\vea \not \to 0$ as $\alpha \to \infty$. We proceed by contradiction and assume that $\rra >> \ma^{\frac{n-1}{n}}$. We then let 
\[ \da = \min \left( \rra, \ma^{\frac{1}{2}} \right).\]
There holds $\da >> \ma$. Estimate \eqref{estea1} of Claim \ref{claimcont1} therefore applies and shows, since $\vea \not \to 0$, that:
\[ \vea = O \left( \ma^{n-1} \da^{-n}\right) .\]
However, by definition of $\da$ it is easily seen that $\da >> \ma^{\frac{n-1}{n}}$, which yields a contradiction with $\vea \not \to 0$. Hence, $\rra = O \left( \ma^{\frac{n-1}{n}}\right)$ and \eqref{estoptlwa} follows from Claim \ref{claimcont2}.

\medskip
\noindent Assume now that $n \ge 6$ and that $X_0(x_0) = 0$. We let $(\da)_\alpha$ be some arbitrary sequence of positive numbers satisfying 
\ben \label{propda}
\frac{\da}{\ma} \to + \infty, \quad \da \le \rra \textrm{ and } \da = O \left( \ma^{\frac{n-4}{n-2}} \right).
\een
In view of the assumptions of Theorem \ref{Th1}, we only have two cases left to consider. First, assume that $\nabla f_0(x_0) \not = 0$. To extract informations on $\rra$ as we previously did for the $X_0(x_0) \not = 0$ case we start obtaining an estimate on $|\nabla \tilde{f}_\alpha(0)|_\xi$. To do this, we let $Y \in \RR^n$ be some fixed vector of norm $1$ and apply a Pohozaev identity to $\va$ on $B_0(\da)$, where $\da$ satisfies \eqref{propda}. It writes as:
\ben \label{gradf0}
\bal
 \int_{\partial B_0(\da)} &\left( \frac{1}{2} Y^k \nu_k \left| \nabla \va \right|_\xi^2 - Y^k \partial_k \va \partial_\nu \va  \right) d\sigma = \\
& - \int_{B_0(\da)} Y^k \partial_k \va \tilde{h}_\alpha \va dy  + \int_{B_0(\da)} Y^k \partial_k \va \tilde{f}_\alpha \va^{2^*-1} dy \\
&  + \int_{B_0(\da)} Y^k \partial_k \va  \left( \tilde{b}_\alpha + \left|\tilde{U}_\alpha + \vpa^{2^*} \Lx \Za \right|_\xi^2 \right) \va^{-2^*-1} dy. \\
\eal
\een
Using the definition of $\rra$ as in \eqref{defra} and \eqref{propda} we have that
\ben \label{gradf1}
  \int_{\partial B_0(\da)} \left( \frac{1}{2} Y^k \nu_k \left| \nabla \va \right|_\xi^2 - Y^k \partial_k \va \partial_\nu \va  \right) d\sigma = O \left( \ma^{n-2} \da^{1-n} \right), 
 \een
that
\ben \label{gradf2}
  \int_{B_0(\da)} Y^k \partial_k \va \tilde{h}_\alpha \va dy = o(\ma), 
  \een
and that
\ben \label{gradf3}
\bal
  \int_{B_0(\da)} Y^k \partial_k \va & \tilde{f}_\alpha \va^{2^*-1} dy = O \left( \ma^n \da^{-1-n} \right) + O (\ma) \\
  &  - \left( \frac{n-2}{2n} f_0(x_0)^{- \frac{n}{2}} K_n^{-n} + o(1) \right) Y^k \partial_k \tilde{f}_\alpha(0).
  \eal
  \een
By \eqref{propda} we can use Claim \ref{claimcont2} to estimate the last integral appearing in \eqref{gradf0}, thus obtaining:
\ben \label{gradf4}
\bal
 \int_{B_0(\da)} Y^k \partial_k \va  \left( \tilde{b}_\alpha + \left|\tilde{U}_\alpha + \vpa^{2^*} \Lx \Za \right|_\xi^2 \right) \va^{-2^*-1} dy  \\
 = O \left( \ma^{n-2} \da^{1-n}\right) + o (\ma).
 \eal
\een
Gathering \eqref{gradf1}, \eqref{gradf2}, \eqref{gradf3} and \eqref{gradf4} in \eqref{gradf0} we obtain that for any $Y \in \RR^n, |Y|_\xi = 1$:
\[ Y^k \partial_k \tilde{f}_\alpha(0) = O(\ma) + \left( \ma^{n-2} \da^{1-n} \right),\]
which gives in the end:
\ben \label{estgradf}
\left| \nabla \tilde{f}_\alpha \right|_\xi(0) = O(\ma) + O \left( \ma^{n-2}  \da^{1-n} \right).
\een
There holds $ \left| \nabla \tilde{f}_\alpha \right|_\xi(0) \not \to 0$ as $\alpha \to + \infty$ since we assumed $\nabla f_0(x_0) \not = 0$ . In particular, we obtain with \eqref{estgradf} that there holds:
\ben \label{estdadfnon0}
\da = O \left( \ma^{\frac{n-2}{n-1}} \right)
\een  
for any sequence $(\da)_\alpha$ satisfying \eqref{propda}. Assume now by contradiction that there holds $ \rra >> \ma^{\frac{n-4}{n-2}} $ as $\alpha \to + \infty$. The sequence $\ma^{\frac{n-4}{n-2}}$ therefore satisfies \eqref{propda} and hence with \eqref{estdadfnon0} there holds:
\[ \ma^{\frac{n-4}{n-2}} = 0 \left( \ma^{\frac{n-2}{n-1}}\right),\]
which is impossible since there holds $\frac{n-2}{n-1} > \frac{n-4}{n-2}$ for any $n \ge 6$. Hence, $\rra$ satisfies $\rra = O \left( \ma^{\frac{n-4}{n-2}} \right)$ and hence $\rra$ satisfies \eqref{propda}, which gives in the end, with \eqref{estdadfnon0}:
\[ \rra = O \left( \ma^{\frac{n-2}{n-1}} \right) . \]
Estimate \eqref{estoptlwa} now follows form Claim \ref{claimcont2}.

\medskip

\noindent Finally, if there holds $X_0(x_0) = 0$ and $\nabla f_0(x_0) = 0$, we assume that  
\ben \label{condcourbure}
 \frac{n-2}{4(n-1)} R(g)(x_0) -h_0(x_0) - C(n) \frac{\triangle_g f_0(x_0)}{f_0(x_0)} > 0,
\een
where
\be 
C(n) = \pui K_n^{-n} \left( \int_{\RR^n} \left( 1 + \frac{|x|^2}{n(n-2)}\right)^{2-n} dx\right)^{-1}
\ee
and
\ben \label{defKn-n}
 K_n^{-n} = 2^{-n} \left( n(n-2)\right)^{\frac{n}{2}} \omega_n 
 \een
is the energy of the standard bubble. Easy computations give that 
\ben \label{defC(n)}
C(n) = \frac{(n-2)(n-4)}{8(n-1)}. 
\een

We obtain the control on $\rra$ as a direct consequence of the geometric condition \eqref{condcourbure}. We write a Pohozaev identity on $B_0(\rra)$: it writes as
\ben \label{pohora}
\bal
\int_{B_0(\rra)}&  \left( x^k \partial_k \va + \pui \va \right) \triangle_\xi \va dx \\
& = \int_{\partial B_0(\rra)} \left( \frac{1}{2} \rra |\nabla \va |_\xi^2 - \pui \va \partial_\nu \va - \rra \left( \partial_\nu \va \right)^2 \right) d\sigma.
\eal
\een
On one hand, using the definition of $\rra$, there holds:
\ben \label{bord}
 \int_{\partial B_0(\rra)} \left( \frac{1}{2} \rra |\nabla \va |_\xi^2 - \pui \va \partial_\nu \va - \rra \left( \partial_\nu \va \right)^2 \right) d\sigma = O \left( \left( \frac{\ma}{\rra} \right)\right)^{n-2}.
\een
On the other hand we can write, using \eqref{sysconf}, that
\ben \label{plein}
\bal
\int_{B_0(\rra)}&  \left( x^k \partial_k \va + \pui \va \right) \triangle_\xi \va dx = K_1 + K_2 + K_3,\\
\eal
\een
where
\ben
\bal
K_1 &= - \int_{B_0(\rra)} \left( x^k \partial_k \va + \pui \va \right) \tilde{h}_\alpha(x) \va(x) dx, \\
K_2 &=   \int_{B_0(\rra)}  \left( x^k \partial_k \va + \pui \va \right) \tilde{f}_\alpha (x) \va^{2^*-1}(x) dx,  \\
K_3 & = \int_{B_0(\rra)}  \left( x^k \partial_k \va + \pui \va \right) \left( \tilde{b}_\alpha + \left|\tilde{U}_\alpha + \vpa^{2^*} \Lx \Za \right|_\xi^2 \right) \va^{-2^*-1} dx .\\
 \eal
\een
Since $n \ge 6$ and by \eqref{rsurmu} we have, by easy computations and using Lebesgue's dominated convergence theorem, that:
\ben \label{K1}
\bal
K_1 = \frac{4(n-1)}{n-4} K_n^{-n} f_0(x_0)^{- \frac{n}{2}} & \left(h_0(x_0) - \frac{n-2}{4(n-1)} R(g)(x_0) \right)\ma^2 + o(\ma^2) \\
\eal
\een
where $K_n^{-n}$ is as in \eqref{defKn-n}. Using the definition of $\rra$ and since $\tilde{b}_\alpha \ge 0$ we can write that:
\[ K_3 \le \int_{B_0(2 R_\alpha \ma)}  \left( x^k \partial_k \va + \pui \va \right) \left( \tilde{b}_\alpha + \left|\tilde{U}_\alpha + \vpa^{2^*} \Lx \Za \right|_\xi^2 \right) \va^{-2^*-1}, \]
where $R_\alpha$ is as in \eqref{Ralpha}. We once again consider a sequence of radii $(\da)_\alpha$ satisfying \eqref{propda}. Using Claim \ref{claimcont2} on $B_0(2 \da)$ the latter becomes:
\ben \label{K3}
K_3 \le o \left( \left( \frac{\ma}{\da} \right) \right)^{n-2} + o(\ma^2). 
\een
Finally, since $f_\alpha \to f_0$ in $C^2(M)$ as $\alpha \to + \infty$ we can write with \eqref{rsurmu} and the dominated convergence theorem that:
\ben \label{K2a}
\bal
K_2 = \pui  f_0(x_0)^{ - \frac{n}{2}} \frac{\triangle_g f_0(x_0)}{f_0(x_0)} K_n^{-n} \ma^2 + o (\ma^2) \\
+ o \left( \left( \frac{\ma}{\rra} \right) \right)^{n-2} + o \left( \ma |\nabla \tilde{f}_\alpha |_\xi(0) \right).
\eal
\een
Using estimate \eqref{estgradf} relative to the sequence $(\da)_\alpha$ to estimate $\left| \nabla \tilde{f}_\alpha \right|_\xi(0)$, combining it with  \eqref{K2a}, \eqref{bord}, \eqref{plein}, \eqref{K1} and \eqref{K3} and plugging everything into \eqref{pohora} we obtain:
\ben \label{pohora2}
\bal
\Bigg[ \frac{n-2}{4(n-1)} R(g)(x_0) - h_0(x_0) & - C(n) \frac{\triangle_g f_0(x_0)}{f_0(x_0)}\Bigg] \ma^2  \\
& \le   o(\ma^2) + O \left( \frac{\ma}{\rra}\right)^{n-2} + o \left( \left( \frac{\ma}{\da} \right) \right)^{n-2}.
\eal
\een
Assume by contradiction that $\rra >> \ma^{\frac{n-4}{n-2}}$. Then choosing $\da = \ma^{\frac{n-4}{n-2}}$ yields a contradiction in \eqref{pohora2} because of \eqref{condcourbure}. Hence $\rra = O \left( \ma^{\frac{n-4}{n-2}} \right)$ and \eqref{estoptlwa} follows from Claim \ref{claimcont2}. 
\end{proof}

\subsection{Conclusion of the proof of Proposition \ref{propoblow}}

Thanks to the asymptotic description of the defects of compactness of the sequences $\va$ and $\Lx \Za$ in the ball $B_0(\rra)$ obtained in the previous subsection we now conclude the proof of Proposition \ref{propoblow}. In the following Claim we show that $r_\alpha = \ra$, hence that the sequence $\va$ equals the standard bubble at first order up to the radius $\ra$. 

\begin{claim} \label{raegalerhoa}
There holds that, up to a subsequence:
\ben \label{regalerho}
\rra = \ra.
\een
In particular, $\ra \to 0$ as $\alpha \to + \infty$ and Claims \ref{claimcont1} and \ref{claimcont2} apply with $\da = \ra$ by Claim \ref{controlera}.
\end{claim}

\begin{proof}
For any $x \in B_0(2)$, we let:
\ben \label{blowdistra}
 \hva(x) =  \ma^{1 - \frac{n}{2}} \rra^{n-2} \va( \rra x).  
\een
Using \eqref{sysconf} it is easily seen that $\hva$ satisfies:
\ben \label{eqhva}
 \triangle_\xi \hva + \rra^2 \hat{h}_\alpha \hva = \hat{f}_\alpha \hva^{2^*-1} + \frac{\hat{a}_\alpha}{ \hva^{2^*+1}},
\een
where we have let:
\ben \label{coeffhat}
\begin{aligned}
& \hat{a}_\alpha(x) = \frac{\rra^{4n-2}}{\ma^{2n-2}} \left( \hat{b}_\alpha + \left| \hat{U}_\alpha + \hat{\varphi}_\alpha^{2^*} \Lx \Za(\rra \cdot) \right|_\xi^2 \right)(x), \\
& \hat{h}_\alpha (x) = \tilde{h}_\alpha(\rra x), \\
& \hat{f}_\alpha (x) = \tilde{f}_\alpha (\rra x), \\
& \hat{b}_\alpha(x) = \tilde{b}_\alpha(\rra x), \\
& \hat{U}_\alpha(x) = \tilde{U}_\alpha(\rra x) ,\\
& \hat{\varphi}_\alpha(x) = \vpa(\rra x). \\
\end{aligned}
\een
By definition of $\rra$, by Claim \ref{claimcontbulle} and by \eqref{blowdistra} there holds, for some positive $C$ that does not depend on $\alpha$:
\ben \label{conthva}
\bal
\frac{1}{C} & \left( \left( \frac{\ma}{\rra}\right)^2 + \frac{f_\alpha(\xa)}{n(n-2)} |x|^2 \right)^{1- \frac{n}{2}} \le \hva \le C  \left( \frac{n(n-2)}{f_\alpha(\xa)} \right)^{\pui} |x|^{2-n} .
\eal
\een
Using estimate \eqref{estoptlwa} and \eqref{conthva} one obtains that  for any $x \in B_0(2)$,
\ben \label{contasurv}
  \frac{\hat{a}_\alpha}{ \hva^{2^*+1}}(x) \le C  \left( \left( \frac{\ma}{\rra}\right)^2 + \frac{f_\alpha(\xa)}{n(n-2)} |x|^2 \right)^{\frac{n}{2}} \in  L^\infty(B_0(2)) .
  \een
By \eqref{eqhva}, \eqref{conthva}, \eqref{contasurv} and standard elliptic theory we get that 
\ben \label{convhva}
\hva \to \hat{v} \textrm{  in } C^1_{loc} (B_0(2) \backslash \{ 0\})
\een
as $\alpha \to + \infty$, and we have that, for $x \not = 0$:
\ben \label{prophv}
\hat{v}(x) = \frac{\lambda_0}{|x|^{n-2}} + H(x),
\een
where $\lambda_0 =  \left( \frac{n(n-2)}{f_0(x_0)}\right)^{\pui}$ and $H$ is a superharmonic function in $B_0(2)$. By Claim \ref{claimcontbulle} there also holds $H \ge 0$ in $B_0(2)$.

\medskip

\noindent We claim that $H(0) > 0$ if $\rra < \ra$. Indeed, by the definition of $\rra$ as in \eqref{defra}, if we assume $\rra < \ra$ there exists a sequence $\ya \in B_0(\rra)$ such that there holds:

\begin{itemize}
\item either $\va(\ya) = (1+\ve) \Ba(\ya)$,
\item etiher $| \nabla \va (\ya)|_\xi = (1+\ve) |\nabla \Ba (\ya)|_\xi$,
\item or $\left( x^k \partial_k \va + \pui \va \right)(\ya) = 0$.
\end{itemize}

\noindent Letting $\hat{y}_\alpha = \frac{\ya}{\rra}$, it is easily seen that each of the above three cases implies that either $H(\hat{y}_\alpha)$ or $\nabla H (\hat{y}_\alpha)$ are nonzero which implies, since $H$ is superharmonic and nonnegative, that $H(0) > 0$.

\medskip

We therefore show that $\rra = \ra$ by showing that $H(0) \le 0$. Note that since $H$ is nonnegative and superharmonic this will actually show that $H(0) = 0$, and hence that $H$ is everywhere zero.
To do this, we let $0 < \delta < 1$ and write a Pohozaev identity for $\va$ on $B_0(\delta \rra)$. We have:
\ben \label{pohoraconclu}
\bal
\int_{B_0(\delta \rra)}&  \left( x^k \partial_k \va + \pui \va \right) \triangle_\xi \va dx \\
& = \int_{\partial B_0(\delta \rra)} \left( \frac{1}{2} \delta \rra |\nabla \va |_\xi^2 - \pui \va \partial_\nu \va - \delta  \rra \left( \partial_\nu \va \right)^2 \right) d\sigma.
\eal
\een
By \eqref{convhva} and \eqref{prophv} there holds:
\ben \label{bordpohoconclu}
\bal
& \int_{\partial B_0(\rra)} \left( \frac{1}{2} \delta \rra |\nabla \va |_\xi^2 - \pui \va \partial_\nu \va - \delta \rra \left( \partial_\nu \va \right)^2 \right) d\sigma \\
 & \quad  = \left( \frac{1}{2} (n-2)^2 \lambda_0 \omega_{n-1} H(0) + \ve(\delta) + o(1) \right) \left( \frac{\ma}{\rra}\right)^{n-2}, 
\eal
\een
where until the end of this subsection $\ve(\delta)$ will denote some bounded quantity such that
\ben \label{epsdelta}
 \lim_{\delta \to 0} \ve(\delta) = 0 .
\een
Independently, there holds
\[ \int_{B_0(\delta \rra)}  \left( x^k \partial_k \va + \pui \va \right) \triangle_\xi \va dx = L_1 + L_2 + L_3,\]
where
\[
\bal
L_1 &= - \int_{B_0(\rra)} \left( x^k \partial_k \va + \pui \va \right) \tilde{h}_\alpha(x) \va(x) dx, \\
L_2 &=   \int_{B_0(\rra)}  \left( x^k \partial_k \va + \pui \va \right) \tilde{f}_\alpha (x) \va(x)^{2^*-1} dx,  \\
L_3 & = \int_{B_0(\rra)}  \left( x^k \partial_k \va + \pui \va \right) \left( \tilde{b}_\alpha + \left|\tilde{U}_\alpha + \vpa^{2^*} \Lx \Za \right|_\xi^2 \right)(x) \va(x)^{-2^*-1} dx .\\
\eal
\]
Straightforward computations yield: 
\ben \label{L1}
L_1 = \left \{
\bal
& O\left(\ma \rra \right)  & \textrm{ if } n =3, \\
& O \left( \ma^2 \ln \left( \frac{\rra}{\ma} \right) \right) & \textrm{ if } n =4, \\
& \frac{4(n-1)}{n-4} K_n^{-n} f_0(x_0)^{- \frac{n}{2}}  \\
& \times \left(h_0(x_0) - \frac{n-2}{4(n-1)} R(g)(x_0) \right)\ma^2 + o(\ma^2) &  \textrm{ if } n \ge 5, \\
\eal 
\right.
\een
where $K_n^{-n}$ is as in \eqref{defKn-n}. Because of Claim \ref{controlera}, the conclusion of \eqref{estgradf} is still valid when taking $\da = \rra$. Hence, mimicking the computations that led to \eqref{K2a} we therefore obtain:
\ben \label{L2}
L_2 = \left \{ 
\bal
& o \left( \frac{\ma}{\rra} \right) & \textrm{ if } n = 3, \\
& \pui  f_0(x_0)^{ - \frac{n}{2}} \frac{\triangle_g f_0(x_0)}{f_0(x_0)} K_n^{-n} \ma^2 + o (\ma^2)  \qquad  &\\
& \qquad \qquad \qquad \qquad + o \left( \left( \frac{\ma}{\rra} \right)^{n-2} \right)  & \textrm{ if } n \ge 4 .\\
\eal \right.
\een
Using the definition of $\rra$ as in \eqref{defra} and the optimal estimate \eqref{estoptlwa}, we obtain that:
\ben \label{L3}
L_3 \le o(\ma^2) + o \left( \left( \frac{\ma}{\rra}\right)^{n-2} \right).
\een
We are now able to conclude the proof of Claim \ref{raegalerhoa}. We first treat the case where $3 \le n \le 5$. Then \eqref{bordpohoconclu}, \eqref{L1}, \eqref{L2} and \eqref{L3} in \eqref{pohoraconclu}, along with Claim \ref{controlera} give:
\[  \left( \frac{1}{2} (n-2)^2 \lambda_0 \omega_{n-1} H(0) + \ve(\delta) + o(1) \right) \left( \frac{\ma}{\rra}\right)^{n-2} \le  o \left( \left( \frac{\ma}{\rra} \right)^{n-2} \right), \]
so that letting $\alpha \to + \infty$ and then $\delta \to 0$ yields:
\[ H(0) \le 0.\]
Assume now that $n \ge 6$. Then  \eqref{bordpohoconclu}, \eqref{L1}, \eqref{L2} and \eqref{L3} in \eqref{pohoraconclu} give:
\ben \label{pohoran6}
\bal
& \frac{n-4}{4(n-1)} K_n^n f_0(x_0)^{\frac{n}{2}} \left( \frac{1}{2} (n-2)^2 \lambda_0 \omega_{n-1} H(0) + \ve(\delta) + o(1) \right) \left( \frac{\ma}{\rra}\right)^{n-2} \\
& \quad \quad  \le  \Bigg [ h_0(x_0) - \frac{n-2}{4(n-1)} R(g)(x_0) + C(n) \frac{\triangle_g f_0(x_0)}{f_0(x_0)}\Bigg] \ma^2 + o(\ma^2) ,
 \eal
\een
where $K_n^{-n}$ is as in \eqref{defKn-n}
and $C(n)$ is as in \eqref{defC(n)}. We separate the proof for $n \ge 6$ between three cases. Assume first that $X_0(x_0) \not = 0$. Then by Claim \ref{controlera} there holds $\rra = O \left(\ma^{\frac{n-1}{n}} \right)$ and then
 \[ \ma^2 = o \left( \left( \frac{\ma}{\rra}\right)^{n-2} \right),\]
 so that \eqref{pohoran6} gives, taking the limit $\alpha \to \infty$ and then $\delta \to 0$:
 \[ H(0) \le 0.\]
Assume then that $\nabla f_0(x_0) \not = 0$. By Claim \ref{controlera} there holds $\rra = O \left( \ma^{\frac{n-2}{n-1}}\right)$ and then
\[ \ma^2 = o \left( \left( \frac{\ma}{\rra} \right)^{n-2} \right). \]
Once again, \eqref{pohoran6} gives $H(0) \le 0$ taking the limit $\alpha \to \infty$ and $\delta \to 0$. Assume finally that 
\[ h_0(x_0) - \frac{n-2}{4(n-1)} R(g)(x_0) + C(n) \frac{\triangle_g f_0(x_0)}{f_0(x_0)} < 0.\]
In this case there holds
\[
\bal
& \frac{1}{2} (n-2)^2 \lambda_0 \omega_{n-1} \frac{n-4}{4(n-1)} K_n^{n} f_0(x_0)^{ \frac{n}{2}}  H(0) \\
& \le \lim_{\alpha \to + \infty} \rra^{n-2} \ma^{4-n}   \Bigg [ h_0(x_0) - \frac{n-2}{4(n-1)} R(g)(x_0) + C(n) \frac{\triangle_g f_0(x_0)}{f_0(x_0)}\Bigg] 
 \eal
 \]
 and hence, once again, $H(0) \le 0$. This therefore concludes the proof of Claim \ref{raegalerhoa} and shows that $\rra = \ra$.
\end{proof}

\noindent Remember that the definition of $\rra$ in \eqref{defra} depended on some fixed parameter $\ve > 0$. Claim \ref{raegalerhoa} shows that $\rra$ actually does not depend on $\ve$ since $\rra = \ra$. This yields then:
\be
 \sup_{B_0(\ra)} \left| \frac{\va}{\Ba} - 1 \right|  = o(1) 
 \ee
as $\alpha \to + \infty$. Since by Claim \ref{raegalerhoa} there holds $\ra \to 0$ as $\alpha \to \infty$, this concludes the proof of Proposition \ref{propoblow}.

\section{Instability results} \label{instabilite}

In this section we prove Theorem \ref{thinstabEL}. We show that the assumptions of Theorem \ref{Th1} are sharp by constructing blowing-up sequences of solutions of system \eqref{systype}. In dimensions $3$ to $5$ we construct such examples on the standard sphere. In dimensions $n \ge 6$ we construct them on closed manifolds of positive scalar curvature admitting a locally conformally flat pole and with no conformal Killing $1$-forms. We adapt, in dimensions greater than $6$, the constructions of Druet-Hebey \cite{DruHeb} and distinguish between dimension $6$ and dimensions $n \ge 7$. 

\medskip
A manifold $(M,g)$ is said to have a conformally flat pole at $x_0 \in M$ if $g$ is conformally flat in a neighborhood of $x_0$; locally conformally flat manifolds are such manifolds. Recall that a manifold $(M,g)$ is said to have no conformal Killing $1$-forms (or equivalently, no conformal Killing vector-fields) if any $1$-form satisfying $\Lg X = 0$ in $M$ is zero. Nontrivial conformal Killing $1$-forms may be found on specific manifolds, but as shown in Beig-Chru\'{s}ciel-Schoen \cite{BeChSc} they generically do not exist. Examples of manifolds of positive scalar curvature with a conformally flat pole and having no conformal Killing $1$-forms are obtained by considering quotients of $\mathbb{S}^n$ by isometry groups acting freely and properly. For instance the projective space $P_n (\mathbb{R}) = \mathbb{S}^n / \{ \pm 1 \}$ is an example in any dimension $n \ge 3$. 

\subsection{Instability in dimension $3 \le n \le 5$}

We state our instability result in dimension $3$, for the sake of clarity. However, such an easy construction can be carried out in dimensions $4$ and $5$ without additional difficulties.

\begin{prop} \label{instabn3}
Let $(\mathbb{S}^3,h)$ be the standard sphere. There exists a sequence $(U_\alpha, Y_\alpha)_\alpha$ respectively of smooth $(2,0)$-tensor fields and smooth $1$-forms such that
\[ U_\alpha \longrightarrow U \textrm{ in } C^0(\mathbb{S}^3), \]
and
\[ Y_\alpha \longrightarrow Y \textrm{ in } C^0(\mathbb{S}^3) \]
as $\alpha \to + \infty$, where $U \not \equiv 0$, and there exists a sequence $(\ua,\Wa)_\alpha$, with $\ua > 0$, satisfying:
\ben \label{sysinstab3}
\left \{ 
\bal
& \triangle_{h} \ua + \frac{3}{4} \ua = \frac{3}{4} \ua^{5} + \frac{\left|U_\alpha +  \Lh \Wa \right|_h^2 }{\ua^{7}}, \\
& \Dh \Wa = \ua^{6} X + Y_\alpha,
\eal
 \right. 
\een
where $X \not \equiv 0$ is some fixed $1$-form in $\mathbb{S}^3$ and $\max_M \ua \to + \infty$ as $\alpha \to + \infty$.
\end{prop}
\noindent Note that in view of Theorem \ref{Th1}, the coefficient $X$ in \eqref{sysinstab3} vanishes somewhere.
\begin{proof}
We consider $x_0 \in \mathbb{S}^3$ and consider spherical coordinates centered at $x_0$, which we will denote by $(r, \theta, \phi)$. It is well known that in these coordinates the metric $h$ takes the following form:
\[ h (x) = 
\begin{pmatrix}
1 & 0 & 0 \\
0 & \sin^2 r & 0 \\
0 & 0 & \sin^2 r \sin^2 \theta \\
\end{pmatrix},
\]
where $r = d_h(x,x_0)$. Let $(\la)_\alpha$, $\la > 1$, be a sequence of positive numbers converging to $1$ as $\alpha \to + \infty$. Let, for any $x \in \mathbb{S}^3$ :
\ben \label{defvpa3}
\vpa(x) =  \left( \la^2 - 1\right)^{\frac{1}{4}} \left( \la - r \right)^{- \frac{1}{2}},
\een
where $r = d_h(x_0,x)$. The functions $\vpa$ satisfy:
\ben \label{equavpa3}
\triangle_h \vpa + \frac{3}{4} \vpa = \frac{3}{4} \vpa^{5}.
\een
Let $0 < \delta < \frac{\pi}{2}$. For any $\alpha$, let $Z_\alpha: (0,\pi) \to \mathbb{R}$ be the maximal solution of the following ODE:
\ben \label{defZa3}
\Za'' +2 \cot r \Za'  + \left( 1 - 2 \cot^2 r \right)\Za = - \frac{3}{4} \vpa^6(r),
\een
satisfying $\Za(\pi/2) = 1$ and $\Za'(\pi/2) = 0$, where $\vpa$ is as in \eqref{defvpa3}. By \eqref{defvpa3} and standard ODE theory it is easily seen that, for any $\ve > 0$, $\Za$ is uniformly bounded in $\alpha$ in $C^2([\ve,\pi-\ve])$. Let now $\eta \in C^\infty([0,\pi])$ be such that $\eta \not \equiv 0$ in $[0,\pi]$ and $\eta \equiv 0$ in $[0,\delta]$ and in $[\pi-\delta, \pi]$. We let, for any $x \in \mathbb{S}^3$:
\ben \label{defWa3}
\Wa(x) = \eta(r) \Za(r) \frac{\partial}{\partial r},
\een 
where $r = d_h(x_0,x)$. By definition of $\eta$ and $\Za$, $\Wa$ is smooth in $\mathbb{S}^3$, it is zero in $B_{x_0}(\delta)$ and in $B_{-x_0}(\delta)$ and it is uniformly bounded in $\alpha$ in $C^2(\mathbb{S}^3)$. Straightforward computations using \eqref{defZa3} and the fact that $\Wa$ is radial and only depends on $r$ yield that $\Wa$ satisfies:
\ben \label{eqWa3}
\Dh \Wa = \vpa^{6} X_0 + Y_\alpha,
\een
where we have let
\[ X_0(x) = \eta(r) \frac{\partial}{\partial r}\]
and
\be
\bal
Y_\alpha & = - \frac{4}{3} \Big( 2 \eta'(r) \Za'(r) + \eta''(r) \Za(r) + 2 \cot r \eta'(r) \Za(r) \Big) \frac{\partial}{\partial r}. \\
\eal
\ee
It is then easily seen that $Y_\alpha$ converges in $C^0(\mathbb{S}^3)$ to some smooth $1$-form $Y_0$. It remains to define:
\ben \label{defUa3}
U_\alpha = - \Lh \Wa,
\een
where $\Wa$ is as in \eqref{defWa3}. Here again, since $\Wa$ is uniformly bounded in $C^2 (\mathbb{S}^3)$, $U_\alpha$ converges in $C^0(\mathbb{S}^3)$ to some symmetric traceless $(2,0)$-tensor $U_0$. Since we assumed that $\Za(\pi/2) = 1$ and $\Za'(\pi/2)=0$ there holds:
\[  {\Lg \Wa}_{rr} \left( \frac{\pi}{2}, \theta, \phi \right) = \frac{4}{3} \eta' \left( \frac{\pi}{2} \right), \]
so that it is always possible to choose $\eta$ so as to have a nonzero $U_0$. 
Using \eqref{equavpa3}, \eqref{eqWa3} and \eqref{defUa3} we obtain in the end that $(\vpa, \Wa)$ satisfy:
\be
\left \{ 
\bal
& \triangle_{h} \vpa + \frac{3}{4} \vpa = \frac{3}{4} \vpa^{5} + \frac{\left|U_\alpha +  \Lh \Wa \right|_h^2 }{\vpa^{7}}, \\
& \Dh \Wa = \vpa^{6} X + Y_\alpha,
\eal
 \right. 
\ee
which concludes the proof of Proposition \ref{instabn3}.
\end{proof}

\subsection{Instability in dimensions $n \ge 7$}

\begin{prop} \label{instabn7t}
Let $(M,g)$ be a closed manifold of dimension $n \ge 7$. We assume that $(M,g)$ has a locally conformally flat pole, that $(M,g)$ has positive scalar curvature and that $(M,g)$ has no nontrivial conformal Killing $1$-forms. There exist examples of smooth functions $\tau$ with $\nabla \tau \not \equiv 0$ 
and of traceless divergence-free tensor fields $U \not \equiv 0$ such that there exists sequences $(h_\alpha,\ua, \Wa)_\alpha$ of smooth functions and smooth $1$-forms in $M$ satisfying:
\ben \label{limhat}
 h_\alpha \underset{\alpha \to \infty}{\longrightarrow} \frac{n-2}{4(n-1)} \left( R(g) - |\nabla \tau|^2_g \right) \textrm{ in } C^0(M), 
 \een
$\ua >0$, $\sup_M \ua \to \infty$ and 
\ben \label{instabn7uaWat}
\left \{ 
\bal
& \triangle_{g} \ua + h_\alpha \ua = \frac{n(n-2)}{4} \ua^{2^*-1} + \frac{\left|U +  \Lg \Wa \right|_g^2 }{\ua^{2^*+1}}, \\
& \Dg \Wa = - \frac{n-1}{n}\ua^{2^*} \nabla \tau .
\eal
 \right. 
\een
\end{prop}

\begin{proof}
Let $x_0 \in M$ be such that $(M,g)$ has a locally conformally flat pole at $x_0$. Let $\vp \in C^\infty(M)$, $\vp> 0$ be such that $\tilde g : = \vp g$ is the round sphere metric in $\tilde{B}_{x_0}(\delta)$, where $\delta > 0$ and $ \tilde{B}_{x_0}(\delta)$ is the ball of center $x_0$ and radius $\delta$ measured with respect to $\tilde g$. Let $\eta \in C^\infty(M)$ be a nonnegative function satisfying $\eta \equiv 1$ in $ \tilde{B}_{x_0}(\frac{\delta}{2})$ and $\eta \equiv 0$ outside of $ \tilde{B}_{x_0}(\delta)$. Let $\tau \in C^\infty(M)$ be a smooth function satisfying $\nabla \tau \equiv 0 $ in $\tilde{B}_{x_0}(\delta)$ and $U$ be a nonzero traceless divergence-free $(2,0)$-tensor field. We let $(u_0, W_0)$, with $u_0 > 0$, be a smooth solution of the following system:
 \ben \label{sysu0n7}
\left \{ \bal
& \triangle_{\tilde g} u_0 + \frac{n-2}{4(n-1)} \left( S_{\tilde g}- \vp^{2-2^*} |\nabla \tau|_g^2 \right) u_0 = \frac{n(n-2)}{4} \ua^{2^*-1} + \frac{1}{\vp^{2 \cdot 2^*}} \frac{\left|U +  \Lg W_0 \right|_g^2 }{\ua^{2^*+1}}, \\
& \Dg W_0 = - \frac{n-1}{n}\vp^{2^*} u_0^{2^*} \nabla \tau.
\eal
\right.
 \een
 By the conformal covariance property of the conformal laplacian finding such a $(u_0,W_0)$ amounts to solve the following system:
  \ben
\left \{
 \bal
& \triangle_{ g} (\vp u_0) + \frac{n-2}{4(n-1)} \left( S_{g} - |\nabla \tau|_g^2 \right) (\vp u_0) = \frac{n(n-2)}{4} (\vp u_0)^{2^*-1} + \frac{\left|U +  \Lg W_0 \right|_g^2 }{(\vp u_0)^{2^*+1}} , \\
& \Dg W_0 = - \frac{n-1}{n}(\vp u_0)^{2^*} \nabla \tau,
\eal
\right.
 \een
 which is always possible as long as $U$ is nonzero and $\Vert \nabla \tau \Vert_{\infty} + \Vert U \Vert_\infty \le C$, where $C$ is a positive constant that depends only on $n$ and $g$. We refer for this to the existence result in Premoselli \cite{Premoselli1}.
Let $(\la)_\alpha$, $\la > 1$, be a sequence of numbers such that $\la \to 1$ as $\alpha \to \infty$. We let in the following:
\ben \label{defvpainstabt}
\vpa(x) = \left( \la^2 - 1\right)^{\frac{n-2}{4}} \left( \la - r \right)^{1 - \frac{n}{2}},
\een
where $r = \cos d_{\tilde g} \left(x_0, x \right)$. Since $\tilde g$ is the round metric in $ \tilde{B}_{x_0}(\delta)$ and by definition of $\eta$ there holds:
\ben \label{equavpainstabt}
\left(  \triangle_{\tilde g} + \frac{n-2}{4(n-1)} S_{\tilde g} \right) (\eta \vpa) = \frac{n(n-2)}{4} \eta \vpa^{2^*-1} + 2 \langle \tilde{\nabla} \eta, \tilde{\nabla} \vpa \rangle_{\tilde g} + \vpa \triangle_{\tilde g} \eta,
 \een
 where $\tilde{\nabla}$ stands for the gradient operator for the metric $\tilde g$.
For the sake of simplicity we let in the following: 
\ben \label{Ent}
\bal
E_n(u_0) & = \frac{n-2}{4(n-1)} \vp^{2-2^*} |\nabla \tau|_g^2 u_0.
\eal
\een
Note that there holds:
\ben \label{u0t}
E_n(u_0)(x_0) = 0.
\een
Since $M$ has no conformal Killing $1$-forms, we can let $Z_\alpha$ be the unique $1$-form satisfying in $M$:
\ben \label{defZainstabt}
\Dg \Za = - \frac{n-1}{n}\left( u_0 + \eta \vpa \right)^{2^*} \vp^{2^*} \nabla \tau.
\een
We define:
\ben \label{defAainstabt}
\bal
A_\alpha =&  \frac{1}{\vp^{2 \cdot 2^*}}\left( \frac{\left| U + \Lg \Za \right|_g^2}{\left( u_0 + \eta \vpa \right)^{2^*+1}} -  \frac{\left| U + \Lg W_0 \right|_g^2}{u_0^{2^*+1}} \right) \\
& +  \left( \eta^{2^*-1} - \eta \right)\vpa  - 2 \langle \tilde{\nabla} \eta, \tilde{\nabla} \vpa \rangle_{\tilde g} - \vpa \triangle_{\tilde g} \eta \\
\eal
\een
and
\ben \label{defFainstabt}
\bal
& F_\alpha  =  \frac{n(n-2)}{4}\left[ \left( u_0+\eta \vpa \right)^{2^*-1} - {u_0}^{2^*-1} - (\eta \vpa)^{2^*-1} \right] \\
\eal
\een
We also let $\psa$ be the unique solution in $M$ of:
\ben \label{defpsainstabt}
\triangle_{\tilde g} \psa + \frac{n-2}{4(n-1)} S_{\tilde g} \psa = \left( F_\alpha + A_\alpha \right).
\een
Finally, we let
\ben \label{defuainstabt}
\tua = u_0 + \vpa + \psa,
\een
where $\vpa$ is as in \eqref{defvpainstabt} and define $\Wa$ as the unique $1$-form in $M$ satisfying:
\ben \label{defWainstabt}
\Dg \Wa = - \frac{n-1}{n}\vp^{2^*} \tua^{2^*}  \nabla \tau.
\een
It is easily seen that $(\tua,\Wa)$ satisfies in $M$
\ben \label{equainstabn7t}
\left \{ 
\bal
& \triangle_{\tilde g} \tua + \tilde{h}_\alpha \tua = \frac{n(n-2)}{4} \tua^{2^*-1} + \frac{1}{\vp^{2 \cdot 2^*}} \frac{\left| U + \Lg \Wa \right|_g^2 }{\tua^{2^*+1}}, \\
& \Dg \Wa = - \frac{n-1}{n}(\vp \tua)^{2^*} \nabla \tau,
\eal
 \right. 
\een
where we have let
\ben \label{defhainstabt}
\bal
\tilde{h}_\alpha \tua & = \frac{n-2}{4(n-1)} S_{\tilde g} \tua + \frac{n(n-2)}{4} \left[ \tua^{2^*-1} - \left( u_0 + \eta \vpa \right)^{2^*-1}\right] - E_n(u_0) \\
& \qquad \qquad + \frac{1}{\vp^{2 \cdot 2^*}} \left( \frac{\left| U + \Lg \Wa \right|_g^2}{\tua^{2^*+1}}  - \frac{\left| U + \Lg \Za \right|_g^2}{\left(u_0 + \eta \vpa \right)^{2^*+1}} \right),
\eal
\een
where $Z_\alpha$ and $\Wa$ are as in \eqref{defZainstabt} and \eqref{defWainstabt}. In what follows we investigate the convergence of $\tilde{h}_\alpha$. First, we always have
\[ \left| \left(u_0 + \eta \vpa \right)^{2^*} - {u_0}^{2^*} \right| \le C \ma^\pui  \textrm{ on } M \backslash  \tilde{B}_{x_0}(\delta) \]
for some positive constant $C$, where we have let
\ben \label{defmainstabt}
\ma^2 = \la  - 1, 
\een
so that with \eqref{sysu0n7} and \eqref{defZainstabt} there holds by standard elliptic theory, see Section \ref{stddelltheory} :
\ben \label{Zainfinstabt}
\left \Vert \Lg \left( \Za - W_0 \right) \right \Vert_{L^\infty(M)} = O \left( \ma^\pui \right).
\een
In particular this yields
\ben \label{Aa1instabt}
\frac{1}{u_0^{2^*+1}} \left| \left| U + \Lg \Za \right|^2_g - \left| U + \Lg W_0 \right|^2_g \right|(x) \le C  \ma^\pui 
\een
for any $x \in M$, where $C$ does not depend on $\alpha$ or on $x$. With \eqref{Aa1instabt} and \eqref{defvpainstabt} we can write that: 
\[  \left| U + \Lg \Za \right|^2_g \left( \left(u_0 + \eta \vpa \right)^{-2^*-1} - {u_0}^{-2^*-1} \right) \le C \min \left( 1, \vpa \right)  \]
in $M$. This then gives, with \eqref{Aa1instabt} and by the definition of $A_\alpha$ in \eqref{defAainstabt} that:
\ben \label{Aat}
\left| A_\alpha \right| \le O \left( \min (1, \vpa) \right) + O \left(\ma^\pui \right),
\een
where $\ma$ is defined in \eqref{defmainstabt}. Independently, there holds with \eqref{defFainstabt} that
\ben \label{Fat}
\left| F_\alpha \right| \le O \left( \min (\vpa, \vpa^{2^*-2}) \right).
\een
Let $\xa$ be a sequence of points in $M$. By definition of $\vpa$ as in \eqref{defvpainstabt} and by \eqref{defmainstabt} there holds:
\ben \label{vpacontbullest}
\frac{1}{C} \left( \frac{\ma}{ \ma^2 + d_g(x_0, \xa)^2 } \right)^\pui \le \vpa(\xa) \le C \left( \frac{\ma}{ \ma^2 + d_g(x_0, \xa)^2  } \right)^\pui
\een 
for some $C > 0$ independent of $\alpha$. A Green formula using \eqref{defpsainstabt} gives:
\ben \label{greenpsainstabt}
\bal
 \left| \psa(\xa) \right| & \le C \int_{M} d_g(\xa,y)^{2-n} \left| F_\alpha \right|(y)  dv_h(y) \\
 & + C \int_{M} d_g(\xa,y)^{2-n} \left| A_\alpha \right|(y)  dv_h(y). 
 \eal
 \een 
Using \eqref{Aat} there holds:
\ben \label{estpsainstab1t}
\bal
\int_{M} d_g(\xa,y)^{2-n}  & \left| A_\alpha \right|(y)  dv_h(y)   \\
& =  O \left( \int_{d_g(x_0,y) \le \sqrt{\ma}} d_g(\xa,y)^{2-n} dv_h(y) \right) \\
& + O \left( \int_{d_g(x_0,y) \ge \sqrt{\ma}} d_g(\xa,y)^{2-n} \vpa(y) dv_h(y) \right) \\
& + O \left( \ma^\pui \right) . \\
\eal
\een
With \eqref{vpacontbullest} we obtain:
\ben \label{intinstab1t}
 \int_{d_g(x_0,y) \ge \sqrt{\ma}} d_g(\xa,y)^{2-n} \vpa(y) dv_h(y) = O(\ma) 
 \een
so that there holds:
\ben   \label{estpsa1instabt}
\int_{M} d_g(\xa,y)^{2-n}  \left| A_\alpha \right|(y)  dv_h(y)  = O(\ma).
\een
Using \eqref{Fat} there holds: 
\[ \bal
\int_{M} d_g(\xa,y)^{2-n} & \left| F_\alpha \right|(y)  dv_h(y)  \\
& =  O \left( \int_{d_g(x_0,y) \le \sqrt{\ma}} d_g(\xa,y)^{2-n} \vpa^{2^*-2} dv_h(y) \right) \\
& + O \left( \int_{d_g(x_0,y) \ge \sqrt{\ma}} d_g(\xa,y)^{2-n} \vpa(y) dv_h(y) \right) \\
\eal
\]
and using \eqref{intinstab1t} and \eqref{vpacontbullest} yields in the end: 
\ben \label{estpsa2instabt}
\int_{M} d_g(\xa,y)^{2-n} \left| F_\alpha \right|(y)  dv_h(y) = O \left( \frac{\ma}{\ta(\xa)}\right)^2 + O(\ma),
\een
where $\ta(\xa) = \left( \ma^2 + d_g(\xa,x_0)^2 \right)^{\frac{1}{2}}$ and $\ma$ is as in \eqref{defmainstabt}. In the end, \eqref{estpsa1instabt} and \eqref{estpsa2instabt} in \eqref{greenpsainstabt} give:
\ben \label{estpsainstabt}
\left| \psa(\xa) \right| = O \left( \frac{\ma}{\ta(\xa)}\right)^2 + O(\ma).
\een
Note that so far all the computations in the proof of Proposition \ref{instabn7t} actually hold for any $n \ge 6$. We now use the assumption $n \ge 7$ to write that
\[ \ma^2 \ta(\xa)^{-2} = 
\left \{ 
\bal
& o \left( \ma^\pui \ta(\xa)^{4-n} \right)  &\textrm{ if } \ta(\xa) << \sqrt{\ma} \\
& O(\ma) &\textrm{ if } \ta(\xa) \ge \frac{1}{C} \sqrt{\ma} \\
\eal
\right.\]
so that \eqref{estpsainstabt} becomes, for $n \ge 7$:
\ben \label{estpsainstabn7t}
\left| \psa(\xa) \right| = o \left( \ma^\pui \ta(\xa)^{4-n} \right) + O(\ma).
\een
We can now investigate the convergence of $\tilde{h}_\alpha$ defined as in \eqref{defhainstabt}. First, note that \eqref{estpsainstabn7t} implies that 
\ben \label{psapetituat}
\frac{\psa}{\tua} \to 0 \textrm{ and } \tua^{2^*-3} \psa \to 0 \textrm{ in } C^0(M),
\een
where $\tua$ is as in \eqref{defuainstabt}. In particular \eqref{psapetituat} along with the definition of $\tua$ show that $\sup_M \tua \to + \infty$. There also holds by \eqref{psapetituat}:
\ben \label{convha1t}
 \frac{1}{\tua} \left[ \tua^{2^*-1} - \left( u_0 + \eta \vpa \right)^{2^*-1}\right] = o(1) \textrm{ in } C^0(M) .
 \een
 Then, since $u_0$ is assumed to satisfy $E_n(u_0)(x_0) = 0$, where $E_n$ is defined in \eqref{Ent}, we have as $\alpha \to + \infty$
 \ben \label{convha2t}
  \frac{ E_n(u_0) }{\tua} \to \frac{E_n(u_0)}{u_0} = \frac{n-2}{4(n-1)}\vp^{2-2^*} |\nabla \tau|_g^2  \textrm{ in } C^0(M). 
 \een
 Finally, we write that
 \[
 \bal
  \frac{\left| U + \Lg \Wa \right|_g^2}{\tua^{2^*+1}}  - \frac{\left| U + \Lg \Za \right|_g^2}{\left(u_0 + \eta \vpa \right)^{2^*+1}} & = \left( \left| U + \Lg \Wa \right|_g^2 - \left| U + \Lg \Za \right|_g^2\right) \tua^{-2^*-1} \\
 & + \left| U + \Lg \Za \right|_g^2 \left( \tua^{-2^*-1} - \left(u_0 + \eta \vpa \right)^{-2^*-1} \right),
 \eal
 \]
 where $\Wa$ and $\Za$ are as in \eqref{defWainstabt} and \eqref{defZainstabt}. With \eqref{defuainstabt}, \eqref{Aa1instabt} and \eqref{psapetituat} there holds:
 \ben \label{convha3t}
\frac{1}{\tua}  \left| U + \Lg \Za \right|_g^2 \left( \tua^{-2^*-1} - \left(u_0 + \eta \vpa \right)^{-2^*-1} \right) = o(1) \textrm{ in } C^0(M).
 \een
 We have, independently, by \eqref{defZainstabt}, \eqref{defWainstabt} and \eqref{defuainstabt}:
 \[ \Dg \left( \Wa - \Za \right) = - \frac{n-1}{n} \vp^{2^*} \left( \tua^{2^*} - \left( \tua - \psa \right)^{2^*} \right) \nabla \tau .\] 
By standard elliptic theory (see Section \ref{stddelltheory}) and \eqref{psapetituat} there holds,  
since $X \equiv 0$ in $B_{x_0}(\ve)$:
\ben \label{convha4t}
 \left \Vert \Lg \left( \Wa - \Za \right) \right \Vert_{L^\infty(M)}  = o(1).
 \een
 Gathering \eqref{convha1t}, \eqref{convha2t}, \eqref{convha3t} and \eqref{convha4t} in \eqref{defhainstabt} we therefore obtain:
 \ben \label{convthat}
  \tilde{h}_\alpha \to \frac{n-2}{4(n-1)} \left( S_{\tilde g}  - \vp^{2-2^*} |\nabla \tau|_g^2 \right) \textrm{ in } C^0(M).
  \een
We define, in the end: 
 \ben \label{defvraiuan7}
\ua = \vp \tua.   
  \een
The conformal covariance property of the conformal laplacian gives that
\be
 \left(  \triangle_g + \frac{n-2}{4(n-1)} R(g) \right) \ua = \vp^{2^*-1} \left(  \triangle_{\tilde g} + \frac{n-2}{4(n-1)} S_{\tilde g} \right) \tua,
 \ee
and in the end we obtain that $(\ua, \Wa)$ satisfies:
\be
\left \{ 
\bal
& \triangle_{g} \ua + h_\alpha \ua = \frac{n(n-2)}{4} \ua^{2^*-1} + \frac{\left| U + \Lg \Wa \right|_g^2 }{\ua^{2^*+1}}, \\
& \Dg \Wa =  - \frac{n-1}{n}\ua^{2^*} \nabla \tau
\eal
 \right. 
\ee
in $M$, where we have let:
\[ h_\alpha \ua= \frac{n-2}{4(n-1)} R(g) + \vp^{2^*-2} \left( \tilde{h}_\alpha - \frac{n-2}{4(n-1)} S_{\tilde g} \right) \ua  \]
and where $\ua$ and $\Wa$ are as in \eqref{defvraiuan7} and \eqref{defWainstabt}. The convergence in \eqref{convthat} gives in the end:
\be
 h_\alpha \to \frac{n-2}{4(n-1)} \left( R(g) - |\nabla \tau|_g^2 \right), 
 \ee
 which concludes the proof of Proposition \ref{instabn7t}.
\end{proof}
\noindent Notice that the same argument that led to the existence of $(u_0,W_0)$ satisfying \eqref{sysu0n7} shows that the limiting system \eqref{instabn7uaWat} when $h_\alpha$ is replaced by its limit in \eqref{limhat} possesses nontrivial solutions.

\subsection{Instability in dimension $6$}

\begin{prop} \label{instab6}
Let $(M,g)$ be a closed manifold of dimension $6$. We assume that $(M,g)$ has a locally conformally flat pole, that $(M,g)$ has positive scalar curvature and that $(M,g)$ has no nontrivial conformal Killing $1$-forms. There exist examples of smooth functions $\tau,h$ with $\nabla \tau \not \equiv 0$, $h > 6$
and of traceless and divergence-free tensor fields $U \not \equiv 0$ in $M$ such that there exists sequences $(h_\alpha,\ua, \Wa)_\alpha$ of smooth functions and smooth $1$-forms in $M$ satisfying:
\be
 h_\alpha \underset{\alpha \to \infty}{\longrightarrow} h  \textrm{ in } C^0(M), 
 \ee
$\ua >0$, $\sup_M \ua \to \infty$ and 
\be
\left \{ 
\bal
& \triangle_{g} \ua + h_\alpha \ua = 6 \ua^{2} + \frac{\left|U +  \Lg \Wa \right|_g^2 }{\ua^{4}}, \\
& \Dg \Wa = - \frac{5}{6}\ua^{3} \nabla \tau .
\eal
 \right. 
\ee
\end{prop}

\begin{proof}
As before, we let $x_0 \in M$ be such that $(M,g)$ has a locally conformally flat pole at $x_0$, $\vp \in C^\infty(M)$, $\vp> 0$ be such that $\tilde g : = \vp g$ is the round sphere metric in $\tilde{B}_{x_0}(\delta)$ (the ball being taken with respect to the metric $\tilde g$) and we let $\eta \in C^\infty(M)$ be a nonnegative function satisfying $\eta \equiv 1$ in $ \tilde{B}_{x_0}(\frac{\delta}{2})$ and $\eta \equiv 0$ outside of $ \tilde{B}_{x_0}(\delta)$. Let $(\la)_\alpha$, $\la > 1$, be a sequence of numbers such that $\la \to 1$ as $\alpha \to \infty$. We let $\tau$ be a smooth function in $M$ which is constant in $\tilde{B}_{x_0}(\delta)$ and $U$ a nonzero traceless divergence free tensor field in $M$. Since $(M,g)$ has no nontrivial conformal Killing $1$-forms we can let $(u_0,W_0)$, with $u_0 > 0$, be a solution of
 \ben \label{sysu06}
\left \{ \bal
& \triangle_{\tilde g} u_0 + \frac{1}{5} S_{\tilde g} u_0 = - 6 \ua^{2} + \frac{1}{\vp^{6}} \frac{\left|U +  \Lg W_0 \right|_g^2 }{\ua^{4}}, \\
& \Dg W_0 = - \frac{5}{6}\vp^{3} u_0^{3} \nabla \tau.
\eal
\right.
 \een
 Here again, the conformal covariance property of the conformal laplacian reduces the resolution of \eqref{sysu06} to the resolution of the following system:
  \ben
\left \{
 \bal
& \triangle_{ g} (\vp u_0) + \frac{1}{5} S_{g} (\vp u_0) = - 6 (\vp u_0)^{2} + \frac{\left|U +  \Lg W_0 \right|_g^2 }{(\vp u_0)^{4}} , \\
& \Dg W_0 = - \frac{5}{6}(\vp u_0)^{3} \nabla \tau.
\eal
\right.
 \een
The arguments developed in Dahl-Humbert-Gicquaud \cite{DaGiHu}, Proposition $2.1$, yield the existence of such an $(u_0,W_0)$, at least when $\Vert \nabla \tau \Vert_\infty$ is small enough. 
Let $(\la)_\alpha$, $\la > 1$, be a sequence of numbers such that $\la \to 1$ as $\alpha \to \infty$. We let in the following:
\ben \label{defvpainstab6}
\vpa(x) = \left( \la^2 - 1\right) \left( \la - r \right)^{-2},
\een
where $r = \cos d_{\tilde g} \left(x_0, x \right)$. Since $\tilde g$ is the round metric in $ \tilde{B}_{x_0}(\delta)$ and by definition of $\eta$ there holds:
\ben \label{equavpainstab6}
\left(  \triangle_{\tilde g} + \frac{1}{5} S_{\tilde g} \right) (\eta \vpa) =6 \eta \vpa^{2} + 2 \langle \tilde{\nabla} \eta, \tilde{\nabla} \vpa \rangle_{\tilde g} + \vpa \triangle_{\tilde g} \eta,
 \een
 where $\tilde{\nabla}$ stands for the gradient operator for the metric $\tilde g$. We let $Z_\alpha$ be the unique $1$-form satisfying in $M$:
\ben \label{defZainstab6}
\Dg \Za = - \frac{5}{6}\left( u_0 + \eta \vpa \right)^{3} \vp^{3} \nabla \tau.
\een
We define:
\ben \label{defAainstab6}
\bal
A_\alpha & = \frac{1}{\vp^{6}} \left( \frac{\left| U + \Lg \Za \right|_g^2}{\left( u_0 + \eta \vpa \right)^{4}} -  \frac{\left|U +  \Lg W_0 \right|_g^2}{u_0^{4}} \right) \\
& +  \left( \eta^{2} - \eta \right)\vpa  - 2 \langle \tilde{\nabla} \eta, \tilde{\nabla} \vpa \rangle_{\tilde g} - \vpa \triangle_{\tilde g} \eta .\\
\eal
\een
We also let $\psa$ be the unique solution in $M$ of:
\ben \label{defpsainstab6}
\triangle_{\tilde g} \psa + \frac{1}{5} S_{\tilde g} \psa = \vp^{2} A_\alpha.
\een
Finally, we let
\ben \label{defuainstab6}
\tua = u_0 + \vpa + \psa,
\een
where $\vpa$ is as in \eqref{defvpainstab6} and define $\Wa$ as the unique $1$-form in $M$ satisfying:
\ben \label{defWainstab6}
\Dg \Wa = - \frac{5}{6}\vp^{3} \tua^{3}  \nabla \tau .
\een
It is easily seen that $(\tua,\Wa)$ satisfies in $M$
\ben \label{equainstab6}
\left \{ 
\bal
& \triangle_{\tilde g} \tua + \tilde{h}_\alpha \tua = 6 \tua^{2} + \frac{1}{\vp^{6}} \frac{\left|U +  \Lg \Wa \right|_g^2 }{\tua^{4}}, \\
& \Dg \Wa = - \frac{5}{6}(\vp \tua)^{3} \nabla \tau,
\eal
 \right. 
\een
where we have let
\ben \label{defhainstab6}
\bal
\tilde{h}_\alpha \tua & = \frac{1}{5} S_{\tilde g} \tua + 6 \left( \tua^{2} - \vpa^{2}\right) \\
& \qquad \qquad + \frac{1}{\vp^{6}} \left( \frac{\left| U + \Lg \Wa \right|_g^2}{\tua^{4}}  - \frac{\left| U + \Lg \Za \right|_g^2}{\left(u_0 + \eta \vpa \right)^{4}} \right),
\eal
\een
where $Z_\alpha$ and $\Wa$ are as in \eqref{defZainstab6} and \eqref{defWainstab6}. In what follows we investigate the convergence of $\tilde{h}_\alpha$. First, we always have
\[ \left| \left(u_0 + \eta \vpa \right)^{3} - {u_0}^{3} \right| \le C \ma^2  \textrm{ on } M \backslash  \tilde{B}_{x_0}(\delta) \]
for some positive constant $C$, where we have let
\ben \label{defmainstab6}
\ma^2 = \la  - 1, 
\een
so that with \eqref{sysu06} and \eqref{defZainstab6} there holds by standard elliptic theory, see Section \ref{stddelltheory} :
\ben \label{Zainfinstab6}
\left \Vert \Lg \left( \Za - W_0 \right) \right \Vert_{L^\infty(M)} = O \left( \ma^2 \right).
\een
In particular this yields
\ben \label{Aa1instab6}
\frac{1}{u_0^{4}} \left| \left| U + \Lg \Za \right|^2_g - \left | U + \Lg W_0 \right|^2_g \right|(x) \le C  \ma^2
\een
for any $x \in M$, where $C$ does not depend on $\alpha$ or on $x$. With \eqref{Aa1instab6} and \eqref{defvpainstab6} we can write that: 
\[  \left| U +  \Lg \Za \right|^2_g \left( \left(u_0 + \eta \vpa \right)^{-4} - {u_0}^{-4} \right) \le C \min \left( 1, \vpa \right)  \]
in $M$. This then gives, with \eqref{Aa1instab6} and by the definition of $A_\alpha$ in \eqref{defAainstab6} that:
\ben \label{Aa6}
\left| A_\alpha \right| \le O \left( \min (1, \vpa) \right) + O \left(\ma^2 \right),
\een
where $\ma$ is defined in \eqref{defmainstab6}. 
Let $\xa$ be a sequence of points in $M$. By definition of $\vpa$ as in \eqref{defvpainstab6} and by \eqref{defmainstab6} there holds:
\ben \label{vpacontbulles6}
\frac{1}{C} \left( \frac{\ma}{ \ma^2 + d_g(x_0, \xa)^2 } \right)^2 \le \vpa(\xa) \le C \left( \frac{\ma}{ \ma^2 + d_g(x_0, \xa)^2  } \right)^2
\een 
for some $C > 0$ independent of $\alpha$. A Green formula using \eqref{defpsainstab6} gives:
\ben \label{greenpsainstab6}
\bal
 \left| \psa(\xa) \right| &  \le 
  C \int_{M} d_g(\xa,y)^{-4} \left| A_\alpha \right|(y)  dv_h(y). 
 \eal
 \een 
The computations that led to \eqref{estpsa1instabt} apply here and yield:
\ben   \label{estpsa1instab6}
\int_{M} d_g(\xa,y)^{-4}  \left| A_\alpha \right|(y)  dv_h(y)  = O(\ma).
\een
In the end, \eqref{estpsa1instab6} and \eqref{greenpsainstab6} give:
\ben \label{estpsainstab6}
\left| \psa(\xa) \right| = 
O(\ma).
\een
In particular \eqref{estpsainstab6} along with the definition of $\tua$ in \eqref{defuainstab6} show that $\sup_M \tua \to + \infty$. 
We now write that
 \[
 \bal
  \frac{\left| U + \Lg \Wa \right|_g^2}{\tua^{4}}  - \frac{\left| U + \Lg \Za \right|_g^2}{\left(u_0 + \eta \vpa \right)^{4}} & = \left( \left| U + \Lg \Wa \right|_g^2 - \left| U + \Lg \Za \right|_g^2\right) \tua^{-4} \\
 & + \left| U + \Lg \Za \right|_g^2 \left( \tua^{-4} - \left(u_0 + \eta \vpa \right)^{-4} \right),
 \eal
 \]
 where $\Wa$ and $\Za$ are as in \eqref{defWainstab6} and \eqref{defZainstab6}. Mimicking the computations that led to \eqref{convha4t} and using \eqref{estpsainstab6} we obtain:
 \ben \label{convha36}
\frac{1}{\tua}  \left| \Lg \Za \right|_g^2 \left( \tua^{-4} - \left(u_0 + \eta \vpa \right)^{-4} \right) = o(1) \textrm{ in } C^0(M),
 \een
and
\ben \label{convha46}
 \left \Vert \Lg \left( \Wa - \Za \right) \right \Vert_{L^\infty(M)}  = o(1).
 \een
Finally, using \eqref{estpsainstab6} we have that 
\ben \label{convha63}
\ua^2 - \vpa^2 =  2 u_0 \ua - u_0^2 + o(1) + o(\ua). 
\een
Gathering \eqref{convha36} and  \eqref{convha46} in \eqref{defhainstab6} we obtain:
\ben \label{convtha6}
 \tilde{h}_\alpha \to \frac{1}{5} S_{\tilde g} + 12u_0  \textrm{ in } C^0(\mathbb{S}^6).
 \een
We finally define: 
 \ben \label{defvraiuan6}
\ua = \vp \tua.   
  \een
The conformal covariance property of the conformal laplacian gives that
\be
 \left(  \triangle_g + \frac{1}{5} R(g) \right) \ua = \vp^{2} \left(  \triangle_{\tilde g} + \frac{1}{5} S_{\tilde g} \right) \tua,
 \ee
and we obtain that $(\ua, \Wa)$ satisfies:
\be
\left \{ 
\bal
& \triangle_{g} \ua + h_\alpha \ua = 6 \ua^{2} + \frac{\left| U +  \Lg \Wa \right|_g^2 }{\ua^{4}}, \\
& \Dg \Wa = -\frac{5}{6}\ua^{3} \nabla \tau
\eal
 \right. 
\ee
in $M$, where we have let:
\[ h_\alpha \ua= \frac{1}{5} R(g) + \vp^{2} \left( \tilde{h}_\alpha - \frac{1}{5} S_{\tilde g} \right) \tua  \]
and where $\ua$ and $\Wa$ are as in \eqref{defvraiuan6} and \eqref{defWainstab6}. The convergence in \eqref{convtha6} gives in the end:
\be
 h_\alpha \to \frac{1}{5} R(g) + 12 \vp^2 u_0, 
 \ee
 which concludes the proof of Proposition \ref{instab6}.
\end{proof}
\noindent Here again,  the limiting system obtained by replacing $h_\alpha$ by $h$ in \eqref{instabn7uaWat} possesses nontrivial solutions provided $\Vert \nabla \tau \Vert_\infty + \Vert U \Vert_\infty$ is small enough. It is a straightforward consequence of the arguments developed in Premoselli \cite{Premoselli1}.

\section{A Harnack inequality}\label{HI}

We prove in this section a Harnack inequality for solutions of the Einstein-Lichnerowicz equation which was used throughout the paper~:

\begin{prop}\label{prop-HI}
Let $a,f,h$ be smooth functions on $\overline{B_0(2)}\subset {\mathbb R}^n$ and let $u\in C^2\left(B_0(2)\right)$ be a positive solution in $B_0(2)$ of 
$$\Delta_\xi u + h u = f u^{2^*-1} + \frac{a}{u^{2^*+1}}\hskip.1cm.$$
We assume that $f\ge 0$ and $a\ge 0$, $a\not\equiv 0$ and that 
\be 
\Vert u \Vert_{L^\infty\left(B_0(2)\right)} \le M
\ee
for some positive $M$. Then there exists $C>0$ depending only on $\left\Vert h\right\Vert_{L^\infty\left(B_0(2)\right)}$, $\left\Vert a\right\Vert_{L^\infty\left(B_0(2)\right)}$, $\left\Vert f\right\Vert_{L^\infty\left(B_0(2)\right)}$  and $M$
such that
$$\sup_{B_0(1)} u\le C \inf_{B_0(1)} u\hskip.1cm.$$
\end{prop}
The proof follows the lines of the standard Nash-Moser iterative scheme. However, the negative power nonlinearity of $u$ makes the proof more involved. In particular, unlike the standard Harnack inequality, Proposition \ref{prop-HI} is no longer an \emph{a priori} estimate and does not induce a control on the $L^\infty$-norm of $\nabla u$ in $B_0(1)$.

\begin{proof}
First of all, since $u$, $a$ and $f$ are nonnegative there holds
\[  \triangle_\xi u + hu \ge 0 \]
in $B_0(2)$. Hence Theorem $8.18$ in Gilbarg-Trudinger \cite{GilTru} applies and shows that for any $1 \le p < \frac{n}{n-2}$, there exists $C_1(h,p)$ depending only on $\Vert h\Vert_{L^\infty(B_0(2))}$ and $p$ such that
\begin{equation} \label{Harnackminor}
\inf_{B_0(1)} u \ge C_1(h,p) \Vert u \Vert_{L^p(B_0(2))}.
\end{equation}

\noindent We now aim at proving that for any $p \ge 1$ there exist $C = C(a,h,f,M,p)$ such that
\[ \sup_{B_0(1)} u \le C \Vert u \Vert_{L^p(B_0(2))}, \]
an estimate which together with \eqref{Harnackminor} concludes the proof of the proposition. We adapt the steps of the proof of Theorem $4.1$ in Han-Lin \cite{HanLin}. Let $k \ge 2^*+2$ be given and $\eta \in C^\infty_c(B_0(2))$ be a smooth positive function with compact support in $B_0(2)$. Multiplying the equation satisfied by $u$ by $\eta^2 u^k$ and integrating yields 
\be
\bal
 \frac{4(k-1)}{(k+1)^2}  \int_{B_0(2)} \left | \nabla u^{\frac{k+1}{2}}\right |^2 \eta^2 dx &\le \int_{B_0(2)} \left(  \eta^2 f u^{k+2^*-1}  + a \eta^2 u^{k-2^*-1} - \eta^2 h u^{k+1}\right)\, dx\\
&+ \int_{B_0(2)} |\nabla \eta|^2 u^{k+1} \,dx 
\eal
\ee
Let $0 < r < R \le 2$. Assume that $\eta$ is compactly supported in $B_0(R)$, that it equals $1$ in $B_0(r)$ and that it satisfies $| \eta | \le 1$ and $| \nabla \eta | \le \frac{2}{R-r}$ in $B_0(2)$. It is then easily seen that there exists $C_2>0$, depending only on $\Vert h \Vert_{L^\infty(B_0(2))}$, $\Vert a \Vert_{L^\infty(B_0(2))}$, $\Vert f \Vert_{L^\infty(B_0(2))}$ and $M$
such that
$$
 \int_{B_0(2)} \left|  \eta^2 f u^{k+2^*-1}  + a \eta^2 u^{k-2^*-1} - \eta^2 h u^{k+1} \right| dx \le C_2 \left( \int_{B_0(R)}  u^{k+1} dx \right)^{\frac{k-2^*-1}{k+1}} 
$$
and
$$
\int_{B_0(2)}  |\nabla \eta|^2 u^{k+1} dx \le C_2 (R-r)^{-2} \left( \int_{B_0(R)} u^{k+1} dx \right)^{\frac{k-2^*-1}{k+1}}.
$$
Independently, Sobolev's inequality shows that 
\[ \int_{B_0(2)} (\eta u^{\frac{k+1}{2}})^{2^*} \le K \left( \int_{B_0(R)} \left | \nabla( \eta u^{\frac{k+1}{2}} )\right|^2 dx \right)^{\frac{2^*}{2}}\]
for some $K>0$ which leads to 
$$
\left( \int_{B_0(2)} (\eta^\frac{2}{k+1} u )^{\frac{2^*}{2}(k+1)} \right)^{\frac{2}{2^*}} \le C_3  \left( \int_{B_0(R)} |\nabla \eta |^2 u^{k+1}dx + \int_{B_0(2)} \eta^2 | \nabla u^{\frac{k+1}{2}}|^2 dx \right)
$$
for some positive $C_3$. We now let $\gamma = k +1 \ge 2^*+3$ and $\chi = \frac{2^*}{2}$. Combining the above estimates, we obtain that 
\begin{equation} \label{improvedHolder}
\Vert u \Vert_{L^{\chi \gamma}(B_0(r))} \le C_4^{\frac{1}{\gamma}} \left( \frac{\gamma}{(R-r)^2}\right)^{\frac{1}{\gamma}} \Vert u \Vert^{\frac{\gamma-2^*-2}{\gamma}}_{L^\gamma(B_0(R))}
\end{equation}
for some positive constant $C_4$ depending only on $M$,  $\Vert h \Vert_{L^\infty(B_0(2))}$, $\Vert a \Vert_{L^\infty(B_0(2))}$ and $\Vert f \Vert_{L^\infty(B_0(2))}$.
We now pick some $0 < r < 2$ and define two sequences $\gamma_i$ and $r_i$ by $\gamma_i = \chi^i \gamma$ and $r_0 = 2, r_{i+1} = r_i - (2-r)2^{-i-1}$. Inequality \eqref{improvedHolder} then gives that, for any $i \ge 0$:
\[ \Vert u \Vert_{L^{\gamma_{i+1}}(B_0(r_{i+1}))} \le C_4^{\frac{1}{\chi^i \gamma}} \left(  \frac{2 \cdot {2 \chi}^i \gamma}{(2-r)^2}\right)^{\frac{1}{\chi^i \gamma}} \Vert u \Vert_{L^{\gamma_i}(B_0(r_i))} ^{1 - \frac{2^*+2}{\chi^i \gamma}} \]
and we thus obtain that there exists 
some constant $C_6>0$ depending on $\Vert h \Vert_{L^\infty(B_0(2))}$, $\Vert a \Vert_{L^\infty(B_0(2))}$, $\Vert f \Vert_{L^\infty(B_0(2))}$ and $M$ 
but which does not depend on $r$ nor on $\gamma$ such that
\[\Vert u \Vert_{L^{\gamma_i}(B_0(r_i))} \le \frac{C_6}{(2-r)^{2 \sum_{k=0}^i \frac{1}{\gamma} \chi^{-k} }} \Vert u \Vert_{L^{\gamma}(B_0(2))}^{\alpha_i} \]
for all $i \ge 0$, where $\alpha_i = \alpha_i(\gamma) = \prod_{k=0}^i  \left( 1 - \frac{2^*+2}{\chi^k \gamma} \right)$. Passing to the limit as $i \to \infty$ we thus obtain:
\begin{equation} \label{presqueHarnack}
\Vert u \Vert_{L^{\infty}(B_0(r))} \le \frac{C_6}{(2-r)^{\frac{n}{\gamma}}} \Vert u \Vert_{L^{\gamma}(B_0(2))}^{\alpha} .
\end{equation}
 
\noindent To conclude the proof of Proposition \ref{prop-HI}, we need to improve estimate \eqref{presqueHarnack}. Let $1 \le p < \gamma$. There holds:
\[ (2-r)^{-\frac{n}{\gamma}} \Vert u \Vert_{L^{\gamma}(B_0(2))}^{\alpha} \le (2-r)^{-\frac{n}{\gamma}} \Vert u \Vert_{L^\infty(B_0(2))}^{\frac{\alpha}{\gamma}(\gamma - p)} \Vert u \Vert_{L^p(B_0(2))}^{\frac{\alpha}{\gamma}}, \]
so that a Young inequality 
of exponents $\frac{\gamma}{\gamma - p \alpha}$ and $\frac{\gamma}{p \alpha}$ combined with  \eqref{presqueHarnack} yields, for any $\ve > 0$:
\begin{equation} \label{Harnacketapefin}
\Vert u \Vert_{L^{\infty}(B_0(r))} \le C_6 \ve^{\frac{\gamma}{\gamma - p \alpha}}  \Vert u \Vert_{L^\infty(B_0(2))}^{\frac{(\gamma-p)\alpha}{\gamma - p \alpha}} + \frac{C_6}{(2-r)^{\frac{n }{p \alpha}}} \frac{p}{\gamma} \ve^{- \frac{\gamma}{p \alpha}} \Vert u \Vert_{L^p(B_0(2))}. 
\end{equation}
It is easily seen that $\alpha(\gamma) \to 1$ as $\gamma \to \infty$. Hence, choosing $\ve = (2 C_6)^{-1}$ one can then pick $\gamma$ large enough (depending on $a,h,f$ and $M$) so as to have, in \eqref{Harnacketapefin}, for any $p > 1$:
\[ \Vert u \Vert_{L^{\infty}(B_0(r))}  \le \frac{2}{3} \Vert u \Vert_{L^\infty(B_0(2))} + C_7 (2-r)^{-\beta} \Vert u \Vert_{L^p(B_0(2))}\hskip.1cm, \]
where $C_7$ and $\beta > 0$ only depend on $\Vert h \Vert_{L^\infty(B_0(2))}$, $\Vert a \Vert_{L^\infty(B_0(2))}$, $\Vert f \Vert_{L^\infty(B_0(2))}$, $p$ and $M$. 
The conclusion follows using Lemma $4.3$ in Han-Lin \cite{HanLin}. 
\end{proof}

\section{Standard elliptic theory for the Lam\'e operator in $M$ }\label{stddelltheory}

We deal in this subsection with several properties of the Lam\'e operator $\Dg$ on $(M, g)$. It is a differential operator between sections of the cotangent bundle $T^*M $. If $X$ is a $1$-form in $M $, $\Dg$ writes in coordinates as:
\[ \overrightarrow{\triangle_g} X_i = \nabla^j \nabla_j X_i + \nabla^j \nabla_i X^j - \frac{2}{n} \nabla_i ( \mathrm{div}_g X) .\]
If we write formally $\overrightarrow{\triangle_g} X(x) = L(x, \nabla) X$ then the principal symbol of the operator $\overrightarrow{\triangle_g} $ at some point $x \in M$ and for some $\xi \in T_x M $ is given by the determinant of the map $L(x,\xi)$ seen as a linear endomorphism of $T_x^* M$. Thus there holds
\begin{equation} \label{unifelliptic}
 |L(x,\xi) | = \left( 2 - \frac{2}{n}\right) | \xi |_g^{2n} 
 \end{equation} 
which shows that $\overrightarrow{\triangle_g} $ is uniformly elliptic in $M$. It also satisfies the so-called strong ellipticity condition (also called Legendre-Hadamard condition) since for any $x \in M$ and any $\eta \in T_x^*M$:
\begin{equation} \label{strongelliptic}
(L(x,\xi) \eta)_i \eta^i = |\xi|_g^2 |\eta|_g^2 + \left( 1 - \frac{2}{n} \right) |\langle \xi, \eta \rangle|_g^2 \ge |\xi|_g^2 |\eta|_g^2\hskip.1cm.
\end{equation}
Since $M$ is closed, integrating by parts, one gets that, for any $1$-forms $X$ and $Y$,
\begin{equation} \label{ippDeltah}
\int_{M}  \langle \Dg X, Y \rangle_g dv_g = \frac{1}{2} \int_{M} \langle \mathcal{L}_g X, \mathcal{L}_g Y \rangle_g dv_g \hskip.1cm.
 \end{equation}
 In particular, \eqref{ippDeltah} shows that $\Dg$ is self-adjoint in $H^1(M)$ (we still denote the Sobolev space of $1$-forms by $H^1(M)$ since no ambiguity will occur) and that there holds in $M$
 \begin{equation} \label{Killingh}
 \Dg X = 0 \Longleftrightarrow \mathcal{L}_g X = 0
 \end{equation}
 for any $1$-form $X$. Fields of $1$-forms in $M$ satisfying $\mathcal{L}_g X = 0$ are called conformal Killing $1$-forms and by \eqref{ippDeltah} and standard Fredholm theory the set of those $1$-forms is finite dimensional. With \eqref{strongelliptic}, \eqref{ippDeltah} and \eqref{Killingh} standard results of elliptic theory for elliptic operators acting on vector bundles on closed manifolds apply, see for instance Theorem $27$, Appendix H in Besse \cite{Besse}, or Theorem $5.20$ in Giaquinta-Martinazzi \cite{GiaquintaMartinazzi}. In particular, for $1$-forms which are $L^2$-orthogonal to the subspace of conformal Killing $1$-forms, we have the following estimates~: 

 \begin{prop} \label{stdellipticclosed}
 For any $p > 1$, there exist constants $C_1 = C_1(g,p)$ and $C_2 = C_2(g,p)$ depending only on $g$ and $p$ such that for any $1$-form $X$ in $M$:
 \begin{equation} \label{regellipstd}
  \Vert X \Vert_{W^{2,p}(M)} \le C_1 \Vert \Dg X\Vert_{L^p(M)} + C_2 \Vert X \Vert_{L^1(M)}. 
  \end{equation}
  If, in addition, $X$ satisfies
 \begin{equation} \label{Killingortho}
  \int_{M} \langle X, K \rangle_g dv_g  = 0
  \end{equation}
 for all conformal Killing $1$-form $K$, then, for any $p > 1$, we can choose $C_2 = 0$ in \eqref{regellipstd}. 
 \end{prop}

\section{Green functions for Lam\'e-type systems} \label{Greenfunc}

We define, for $1 \le i \le n$, a $1$-form $H_i$ in $\RR^n \backslash \{ 0 \}$ by:
\ben \label{solfondDx}
\bal
 \mathcal{G}_i(y)_j =&  - \frac{1}{4(n-1) \omega_{n-1}} |y|^{2-n} \left((3n-2) \delta_{ij} + (n-2) \frac{y_i y_j}{|y|^2} \right)
 \eal
\een
for any $y \neq 0$. Note that the matrices $(\mathcal{G}_{i}(y)_j)_{ij} $ thus defined are symmetric: for any $y \neq 0$,
\begin{equation} \label{solfondsym}
\mathcal{G}_{i}(y)_j = \mathcal{G}_{j}(y)_i.
\end{equation}
Let $X$ be a field of $1$-forms in $\RR^n$. Integrating by parts and using Stoke's formula it is easily seen that for any $R > 0$ and for any $x \in B_0(R)$ there holds:
\begin{equation} \label{ippsolfond}
\begin{aligned}
X_i(x) = \int_{B_0(R)} \mathcal{G}_{i}(x-y)_j \Dx X(y)^j dx + \int_{\partial B_0(R)} \mathcal{L}_\xi X(y)^{kl} \nu_k(y) \mathcal{G}_{i}(x-y)_l d\sigma \\
- \int_{\partial B_0(R)} \mathcal{L}_{\xi} \big( \mathcal{G}_i(x - \cdot) \big)_{kl}(y) \nu(y)^k X(y)^l d \sigma.
\end{aligned}
\end{equation}
This means in a distributional sense that 
\[ \Dx \big( \mathcal{G}_i (x - \cdot) \big) = \delta_x e_i\hskip.1cm, \]
where $e_i$ is the $i$-th vector of the canonical basis and there holds, for any $1$-form $Y$: $\langle \delta_x e_i, Y \rangle = Y_i(x)$. Equivalently, if we write $\mathcal{G}(x,y) = (\mathcal{G}_{i}(x-y)_j)_{1 \le i,j \le n}$, we get that
\begin{equation} \label{solfondvect}
 \Dx \mathcal{G} (x, \cdot) = \delta_x \mathrm{Id}\hskip.1cm,
 \end{equation}
where $\Dx$ is now seen as a matrix of differential operators acting on a distribution-valued matrix. Note that  the standard results of distribution theory easily extend to distribution-valued matrices, see for instance Schwartz \cite{Schwartz}. 

\noindent If now $X$ is some smooth field of $1$-forms in $L^1(\RR^n)$,
the $1$-form defined in $\RR^n$ by 
\[ W_i(x) = \int_{\RR^n} \mathcal{G}_{i}(x-y)_j Y^j (y) dy = (\mathcal{G} \star Y)_i(x) \]
satisfies in a weak sense, because of \eqref{ippsolfond}:
\begin{equation} \label{convolformes}
\Dx W_i (x) = Y_i(x) \hskip.1cm.
\end{equation}

\medskip

The system \eqref{systype} we are interested in in this article is invariant up to adding to $W_\alpha$ some conformal Killing $1$-form in $M$. We exploit this invariance all along the article by noting that the only relevant quantity to investigate is $\mathcal{L}_g W_\alpha$ and not the $1$-form $W_\alpha$ in itself. In particular we use several times a Green identity for $\Dx$ with Neumann boundary conditions that is proven in what follows. We let 
\begin{equation} \label{Killingspace}
K_R = \{ X \in H^1(B_0(R)), \mathcal{L}_\xi X = 0\}
 \end{equation}
be the Kernel subspace of $1$-forms associated to the Neumann problem for $\Dx$ in $B_0(R)$. The orthogonal subspace of $K_R$ in $B_0(R)$ is the set of $1$-forms $Y \in H^1(B_0(R))$ such that for any $K \in K_R:$
\[ \int_{B_0(R)} \langle Y, K \rangle_\xi dx = 0. \]
Elements of $K_R$ are infinitesimal generators of conformal transformations of $B_0(R)$ and are classified, see Schottenloher \cite{Schottenloher}. In particular $K_R$ is finite dimensional, $\dim K_R = \frac{1}{2}(n+1)(n+2)$, and it is spanned by smooth $1$-forms. Let $m = \frac{1}{2}(n+1)(n+2)$ and  $(K_j)_{j=1 \dots m}$ be an orthonormal basis of $K_0(R)$ for the $L^2$-scalar product, that is
\[ \int_{B_0(R)} \langle K_l, K_p \rangle_\xi dx =  \delta_{lp} .\]
Given a $1$-form $X \in H^1(B_0(R))$ we shall denote by $\pi_R(X)$ its orthogonal projection on $K_R$ given by:
\begin{equation} \label{projkil}
 \pi_R(X) = \sum_{j=1}^{m} \left(  \int_{B_0(R)} \langle K_j, X\rangle dx \right) K_j.
 \end{equation}
The following proposition states the existence of Green $1$-forms satisfying Neumann boundary conditions:

\begin{prop} \label{vecteurGreen}
For any $1 \le i \le n$ and any $R > 0$ there exists a unique $\mathcal{G}_{i,R}$ defined in $B_0(R)\times B_0(R) \backslash D$, where $D = \{ (x,x), x \in B_0(R) \}$, such that $\mathcal{G}_{i,R}(x, \cdot)$ is orthogonal to $K_R$ for any $x \in B_0(R)$ and such that for any smooth $1$-form $X$ in $\overline{B_0(R)}$ there holds:
\begin{equation} \label{ippNeumann}
\begin{array}{l}
{\ds \big( X - \pi_R(X) \big)_i(x)  = \int_{B_0(R)} \mathcal{G}_{i,R}(x,y)_j \Dx X(y)^j dx} \\
{\ds \qquad\qquad\qquad\qquad\quad+ \int_{\partial B_0(R)} \mathcal{L}_\xi X(y)^{kl} \nu_k(y) \mathcal{G}_{i,R}(x,y)_l d\sigma \hskip.1cm,}
\end{array}\end{equation}
where $\pi_R(X)$ is as in \eqref{projkil}. Moreover $\mathcal{G}_{i,R}$ is continuous and continuously differentiable in each variable in $B_0(R) \times B_0(R) \backslash D$. Furthermore, if $K$ denotes any compact set in $B_0(R)$ and if we let 
\begin{equation} \label{rapportdomaine}
 \delta = \frac{1}{R} d(K, \partial B_0(R)) > 0
 \end{equation}
there holds:
\begin{equation} \label{dersolNeumann}
 |x-y| |\nabla \mathcal{G}_{i,R}(x,y)| + |\mathcal{G}_{i,R}(x,y)| \le C(\delta) |x-y|^{2-n} 
 \end{equation}
for any $ x\in B_0(K)$ and any $y \in B_0(R)$, whether the derivative in \eqref{dersolNeumann} is taken with respect to $x$ or $y$, and where $C(\delta)$ is a positive constant that only depends on $\delta$ as in \eqref{rapportdomaine} (in particular it does not depend on $x$).
\end{prop}

\begin{proof}

The proof of this proposition goes through a sequences of claims. The techniques used are strongly inspired from Robert \cite{RobDirichlet}.

\begin{claim} \label{claimexis}
Let $F$ and $G$ be smooth $1$-forms, in $B_0(R)$ and in $\partial B_0(R)$ respectively, satisfying:
\begin{equation} \label{compatibilite}
\int_{B_0(R)}  F_l K^l  d\xi + \int_{\partial B_0(R)} G_l K^l d\sigma= 0
\end{equation} 
for any $K \in K_R$, where $K_R$ is as in \eqref{Killingspace}.
Then there exists a unique smooth $1$-form $Z$ orthogonal to $K_R$ such that
\begin{equation} \label{eqmodele} \left \{ 
\begin{aligned}
& \Dx Z = F & \quad B_0(R) \\
& \nu^k \mathcal{L}_{\xi} Z_{kl} = G_l & \quad \partial B_0(R).\\
\end{aligned} \right. \end{equation}
\end{claim}

\begin{proof}
The existence and uniqueness of $Z$ is ensured by the Lax-Milgram theorem applied on the orthogonal complement of $K_R$ to the symmetric bilinear form
\[ B(X,Y) = \frac{1}{2} \int_{B_0(R)} \langle \mathcal{L}_\xi X, \mathcal{L}_\xi Y \rangle dx\]  
and to the linear  form:
\[ L(X) = \int_{B_0(R)} F_l X^l d\xi - \int_{\partial B_0(R)}  G_l X^l d\sigma.  \]
The coercivity of $B(X,X)$ on the orthogonal complement of $K_R$ follows from the definition of $K_R$ and is obtained via the direct method. We claim now that $Z$ is smooth in $B_0(R)$. This is a consequence of general elliptic regularity results up to the boundary for elliptic systems satisfying complementing boundary conditions, as stated in Agmon-Douglis-Nirenberg \cite{ADN2}. Due to \eqref{unifelliptic} and \eqref{strongelliptic} the problem \eqref{eqmodele} is complemented so that Theorem 10.5 in Agmon-Douglis-Nirenberg \cite{ADN2} applies and shows that $Z$ is smooth. 
\end{proof}
For any $1 \le i \le n$ and any $x \in B_0(R)$ we let $U^R_{i,x}$ be the unique $1$-form in $H^1(\overline{B_0(R)})$, orthogonal to $K_R$, satisfying:
\begin{equation} \label{Uix}
\left \{ 
\begin{aligned}
& \Dx U^R_{i,x} = - \sum_{j=1}^{m} (K_j)_i(x) K_j & \quad B_0(R) \\
& \nu^k \mathcal{L}_{\xi} \big( U^R_{i,x}\big)_{kl} = - \nu^k \mathcal{L}_{\xi} \big( H_i(x- \cdot) \big)_{kl}  & \quad \partial B_0(R).\\
\end{aligned} \right.
\end{equation}
The existence and smoothness of $U^R_{i,x}$ is ensured by Claim \ref{claimexis}. Indeed, the compatibility condition \eqref{compatibilite} is satisfied by applying \eqref{ippsolfond} to any $K \in K_R$.

\medskip

\noindent We now let, for $x \neq y$:
\begin{equation} \label{solfondNeumann}
\mathcal{G}_{i,R}(x,y) = \mathcal{G}_i (x - y) + U^R_{i,x}(y) - \sum_{j=1}^{m} \left(  \int_{B_0(R)} \langle K_j, \mathcal{G}_i(x- \cdot) \rangle dy \right) K_j(y),
\end{equation}
where $\mathcal{G}_i$ is as in \eqref{solfondDx}. By construction, $\mathcal{G}_{i,R}(x,\cdot)$ is a $1$-form defined in $B_0(R) \backslash \{ x\}$. It clearly belongs to $L^2(B_0(R))$, is orthogonal to $K_R$ and continuously differentiable in $B_0(R) \backslash \{ x\}$. Combining \eqref{ippsolfond} and \eqref{Uix} it is easily seen that \eqref{ippNeumann} holds. The next three claims aim at finishing the proof of the proposition. The first one is a uniqueness result. 

\begin{claim} \label{claimunicite}
Assume that for some $1 \le i \le \frac{n}{n-2}$ and for some $x \in B_0(R)$ there exists a $1$-form $M_i$ in $L^1(B_0(R))$ such that for any $X \in C^2(\overline{B_0(R)})$ satisfying that $\mathcal{L}_\xi X_{kl} \nu^k = 0$ on $\partial B_0(R)$ there holds:
\begin{equation} \label{condunicite}
 \int_{B_0(R)} \langle M_i, \Dx X \rangle_\xi d\xi = \left( X - \pi(X)\right)_i(x) .
 \end{equation}
Then $\mathcal{G}_{i,R}(x, \cdot) - M_i \in K_R$, where $K_R$ is as in \eqref{Killingspace}.

\begin{proof}
Let $F_i$ = $\mathcal{G}_{i,R}(x, \cdot) - M_i$. Let $Y$ be a smooth $1$-form with compact support in $B_0(R)$. By Claim \ref{claimexis} there exists a smooth $1$-form $X$ in $B_0(R)$ such that $\Dx X = Y - \pi_R(Y)$ in $B_0(R)$ and $\mathcal{L}_\xi X_{kl} \nu^k = 0$ in $\partial B_0(R)$, where $\pi_R$ is as in \eqref{projkil}. Using \eqref{ippNeumann} and \eqref{condunicite} there holds:
\[ 0 =  \int_{B_0(R)} \langle F_i, \Dx X \rangle_\xi d\xi =   \int_{B_0(R)} \langle F_i, Y - \pi_R(Y) \rangle_\xi d\xi =  \int_{B_0(R)} \langle F_i - \pi_R(F_i), Y \rangle_\xi d\xi \]
by definition of $\pi_R$. Assume for a while that $F_i$ belongs to $L^p(B_0(R))$ for some $p > 1$. A density argument then shows that $F_i = \pi_R(F_i)$.

\noindent It thus remains to prove that $F_i \in L^p(B_0(R))$ for some $p > 1$. We only need to prove this for $M_i$. Let $p \in (1, \frac{n}{n-2})$ and define $q = \frac{p}{p-1}$. Let $Y$ be a smooth $1$-form compactly supported in $B_0(R)$. By Claim \ref{claimexis} there exists a smooth $1$-form $X$ in $B_0(R)$, orthogonal to $K_R$, such that $\Dx X = Y - \pi_R(Y)$ in $B_0(R)$ and $\mathcal{L}_\xi X_{kl} \nu^k = 0$ in $\partial B_0(R)$. Then the definition of $\pi_R$ and \eqref{condunicite} yield:
\begincal
  \int_{B_0(R)} \langle M_i - \pi_R(M_i), Y \rangle_\xi d\xi &=& \int_{B_0(R)} \langle M_i , Y - \pi_R(Y) \rangle_\xi d\xi \\
&=&  \int_{B_0(R)} \langle M_i ,  \Dx X \rangle_\xi d\xi \\
&=& X_i(x).
\fincal
Elliptic regularity results for complemented elliptic systems, as those stated in the proof of Claim \ref{claimexis}, show that there exists a constant $C$ only depending on $q$ such that $\Vert X \Vert_{W^{2,q}} \le C \Vert Y - \pi_R(Y) \Vert_{L^q}$ where we omit to say that these norms are taken on $B_0(R)$ for the sake of clarity. Since $q > \frac{n}{2}$ we thus obtain, using the Sobolev inequality for the embedding of $W^{2,q}$ in $C^0(\overline{B_0(R)})$:
\[   \left | \int_{B_0(R)} \langle M_i - \pi_R(M_i), Y \rangle_\xi d\xi \right | \le C \Vert Y - \pi_R(Y) \Vert_{L^q} \le C \Vert Y \Vert_{L^q}.  \]
A density argument then shows that $M_i - \pi_R(M_i)$ and thus $M_i$ belongs to $L^p(B_0(R))$ for all $p \in (1, \frac{n}{n-2})$.
\end{proof}
\end{claim}

\noindent We now state some rescaling-invariance property of the Green $1$-forms $\mathcal{G}_{i,R}$:
\begin{claim} \label{claimrescaling}
For any $R > 0$ there holds:
\begin{equation} \label{Greenrescaling}
\mathcal{G}_{i,R}(x,y) = \frac{1}{R} \mathcal{G}_{i,1} \left( \frac{x}{R}, \frac{y}{R} \right)
\end{equation} 
for any $x, y \in B_0(R)$.
\end{claim}

\begin{proof}
Let $Y$ be a smooth, compactly supported $1$-form in $B_0(R)$. We define, in $B_0(1)$, $Y_R = Y(R \cdot)$, which is then compactly supported in $B_0(1)$. Let $x \in B_0(R)$ and $1 \le i \le n$. Equation \eqref{ippNeumann} shows that
\begin{equation} \label{rescaleB1}
\left( Y_R - \pi_1(Y_R)\right)_i \left( \frac{x}{R} \right) = \int_{B_0(1)} \left \langle \mathcal{G}_{i,1} \left( \frac{x}{R}, y \right), \Dx Y_R(y) \right \rangle_\xi d\xi.
\end{equation}
Since for any $y \in B_0(1)$ there holds $\Dx Y_R(y) = R^2 \Dx Y (Ry)$ we easily obtain: 
\[  \int_{B_0(1)} \left \langle \mathcal{G}_{i,1} \left( \frac{x}{R}, y \right), \Dx Y_R(y) \right \rangle_\xi dy = \frac{1}{R} \int_{B_0(R)} \left \langle \mathcal{G}_{i,1} \left( \frac{x}{R}, \frac{y}{R} \right) , \Dx Y (y) \right \rangle_\xi dy. \]
Let now $(L_j)_{1 \le j \le m}$ be an orthonormal basis for $K_1$, where $K_1$ is as in \eqref{Killingspace}. Let $Z_j =  R^{- \frac{n}{2}} L_j \left( {\frac{x}{R}} \right)$ for any $1 \le j \le m$ and any $x \in B_0(R)$. Then $(Z_j)_{1 \le j \le m}$ is an orthonormal basis for $K_R$ since there holds
\[ \int_{B_0(R)} \langle Z_k, Z_l \rangle_\xi d\xi = \int_{B_0(R)} R^{-n} \langle L_k \left( \frac{y}{R}\right), L_l \left( \frac{y}{R}\right) \rangle_\xi dy = \int_{B_0(1)} \langle L_k, L_l \rangle_\xi d\xi = \delta_{kl}. \]
Hence one has, by definition of $\pi_R$:
\[ \begin{aligned}
\pi_R(Y) (x) & = \sum_{j = 1}^{m}   \left(  \int_{B_0(R)} \langle R^{- \frac{n}{2}} L_j \left( \frac{y}{R} \right), Y(y) \rangle dy \right) R^{- \frac{n}{2}} L_j \left(\frac{x}{R} \right)  \\
 & = \sum_{j = 1}^{m}   \left(  \int_{B_0(1)} \langle L_j (y), Y_R(y) \rangle dy \right) L_j \left( \frac{x}{R} \right)  = \pi_1(Y_R) \left( \frac{x}{R} \right). \\
 \end{aligned} \]
 In the end \eqref{rescaleB1} becomes:
 \[ \left( Y - \pi_R(Y) \right)_i(x) = \int_{B_0(R)} \left \langle \frac{1}{R}  \mathcal{G}_{i,1} \left( \frac{x}{R}, \frac{y}{R} \right) , \Dx Y (y) \right \rangle_\xi dy.
 \]
Finally, $\frac{1}{R}  \mathcal{G}_{i,1} \left( \frac{x}{R}, \frac{\cdot}{R} \right)$ is orthogonal to $K_R$: indeed for $1 \le j \le m$ there holds:
 \begincal
\int_{B_0(R)} \left \langle \frac{1}{R}  \mathcal{G}_{i,1} \left( \frac{x}{R}, \frac{y}{R}\right), R^{- \frac{n}{2}} L_j \left( \frac{y}{R} \right)  \right \rangle_\xi dy &=& R^{\frac{1}{2}} \int_{B_0(1)} \langle \mathcal{G}_{i,1}( \frac{x}{R}, y), L_j(y) \rangle_\xi dy \\
&=& 0\hskip.1cm,
\fincal
where the last equality is true since $\mathcal{G}_{i,1}( \frac{x}{R}, \cdot)$ is orthogonal to $K_1$ by definition. Using Claim \ref{claimunicite} we then obtain \eqref{Greenrescaling}. 
\end{proof}
\noindent The last ingredient of the proof is a symmetry property of $\mathcal{G}_{i,R}$~:
\begin{claim} \label{symetrie}
For any $x,y \in B_0(R)$ there holds:
\begin{equation} \label{symGreen}
\mathcal{G}_{i,R}(x,y)_j  =  \mathcal{G}_{j,R}(y,x)_i - \pi_R \left( {}^t \mathcal{G}_i(\cdot, x) \right)_j
 \end{equation}
where we have set:
\begin{equation} \label{greentranspo}
 {}^t \mathcal{G}_i(\cdot, x)_j(y) = \mathcal{G}_{j,R}(y,x)_i .
 \end{equation}
\end{claim}

\begin{proof}
Let $\Psi \in C^0(\overline{B_0(R)})$ be a $1$-form orthogonal to $K_R$. We define a $1$-form in $B_0(R)$ by
\begin{equation} \label{Hiform}
H_i(x) = \int_{B_0(R)} \mathcal{G}_{j,R}(y,x)_i \Psi(y) ^j dy \hskip.1cm.
\end{equation}
By the explicit construction of $\mathcal{G}_{j,R}$ in \eqref{solfondNeumann} it is easily seen that $H$ is continuous in $\overline{B_0(R)}$. Also, $H$ is orthogonal to the conformal Killing $1$-forms since by Fubini's theorem, for any $K \in K_R$, 
\begin{equation} \label{orthokil}
 \int_{B_0(R)} H_i(y) K^i (y) dy = \int_{B_0(R)} \Psi^j(z) \int_{B_0(R)} \mathcal{G}_{j,R}(z,y)_i K^i(y) dy dz = 0 
 \end{equation}
since by construction $\mathcal{G}_{j,R}(z, \cdot)$ is orthogonal to $K_R$ for any $z \in B_0(R)$. By Claim \ref{claimexis} and since $\Psi$ is orthogonal to $K_R$ we can let $F$ be the unique $C^1$ $1$-form in $B_0(R)$ orthogonal to $K_R$  satisfying $\Dx F = \Psi$ in $B_0(R)$ and $\mathcal{L}_\xi F_{kl} \nu^k = 0$ in $\partial B_0(R)$. Let also $\Phi$ be a smooth $1$-form such that $\mathcal{L}_\xi \Phi_{kl} \nu^k = 0$ on $\partial B_0(R)$. With Fubini's theorem, equation \eqref{ippNeumann} and using the properties of $\Psi$ and $\Phi$ there holds:
\[ \begin{aligned}
\int_{B_0(R)} H_i(y) \Dx \Phi^i(y) dy & = \int_{B_0(R)} \Psi^j (z) \int_{B_0(R)}  \mathcal{G}_{j,R}(z,y)_i \Dx \Phi^i (y)  dy dz \\
& = \int_{B_0(R)} \Psi^j(z) \left( \Phi - \pi_R(\Phi)\right)_j (z) dz \\
& = \int_{B_0(R)}  \Psi^j(z) \Phi_j(z) dz \\
& = \int_{B_0(R)} \Phi_j(z) \Dx F^j (z) dz \\
& = \int_{B_0(R)} F_j(z) \Dx \Phi^j(z) dz \hskip.1cm,\\
\end{aligned} \]
where we integrated by parts to obtain the last inequality since the boundary terms vanish. In particular:
\[ \int_{B_0(R)} \left( H - F \right)_j(y) \Dx \Phi^j(y) dy = 0 \]
for any smooth $\Phi$ with $\mathcal{L}_\xi \Phi_{kl} \nu^k = 0$ on $\partial B_0(R)$. Note that by a density argument the above inequality remains true for $\Phi \in W^{2,p}(B_0(R))$ for any $p > 1$ and orthogonal to $K_R$. By construction $F$ is orthogonal to $K_R$ and thanks to \eqref{orthokil} so is $H$. By Claim \ref{claimexis} we can thus choose $\Phi$ to be the unique $1$-form orthogonal to $K_R$ satisfying $\Dx \Phi = F - H$ in $B_0(R)$ and $\mathcal{L}_\xi \Phi_{kl}\nu^k = 0$ in $\partial B_0(R)$ to obtain, with the above inequality, that $F = H$. Independently, using \eqref{ippNeumann} gives:
\[ F_i(x) = \int_{B_0(R)} \mathcal{G}_{i,R}(x,y)_j \Psi^j(y) dy\]
so that
\begin{equation} \label{presquesym}
\int_{B_0(R)} \left( \mathcal{G}_{j,R}(y,x)_i - \mathcal{G}_{i,R}(x,y)_j \right) \Psi^j(y)dy = 0
\end{equation}
for any continuous Killing-free $\Psi$. Let now $X$ be any smooth $1$-form in $B_0(R)$. Choose $\Psi = X - \pi_R(X)$. There holds
\begincal
&&\int_{B_0(R)} \mathcal{G}_{j,R}(y,x)_i \pi_R(X)^j(y) dy \\ 
&&\quad = \sum_{p=1}^{m} \int_{B_0(R)} \mathcal{G}_{j,R}(y,x)_i \left(  \int_{B_0(R)} \langle K_p(z), X(z)\rangle_\xi dz \right) K_p(y)^j dy   \\
&&\quad = \sum_{p=1}^{m} \int_{B_0(R)} X_l(z) \left(  \int_{B_0(R)} \mathcal{G}_{j,R}(y,x)_i  K_p(y)^j dy \right)  K_p^l(z) dz \\
&&\quad = \sum_{p=1}^{m}  \int_{B_0(R)} X_l(z) \pi_R \left( {}^t \mathcal{G}_i(\cdot, x) \right)^l(z) dz\hskip.1cm,
\fincal
where ${}^t \mathcal{G}_i(\cdot, x)$ is as in \eqref{greentranspo}. Since $\mathcal{G}_{i,R}(x, \cdot)$ has no conformal Killing part, equation \eqref{presquesym} becomes:
\[ \int_{B_0(R)} \left( \mathcal{G}_{j,R}(y,x)_i - \mathcal{G}_{i,R}(x,y)_j -  \pi_R \left( {}^t \mathcal{G}_i(\cdot, x) \right)_j(y) \right) X^j(y) dy = 0\]
for any smooth $1$-form $X$ and this concludes the proof of the claim.
\end{proof}

We are now able to end the proof of Proposition \ref{vecteurGreen}. Let $x \in B_0(1)$ and consider $U_{i,x}^1$ as defined in \eqref{Uix}. Since $\Dx$ is coercive on the orthogonal of $K_1$ we can use elliptic regularity results -- as those stated in the proof of Claim \ref{claimexis} -- to get that there exist positive constants $C_1, C_2$ that do not depend on $x$ such that
\begin{equation} \label{bornederUix}
 \Vert U_{i,x}^1 \Vert_{C^1(\overline{B_0(R)})} \le C_1 + C_2 \Vert \mathcal{L}_\xi \mathcal{G}_i(x - \cdot), \nu \otimes \cdot \Vert_{C^1(\partial B_0(R))}\hskip.1cm,
 \end{equation}
where we have let $ ( \mathcal{L}_\xi \mathcal{G}_i(x - \cdot), \nu \otimes \cdot)_l = \nu^k \mathcal{L}_{\xi} \big( \mathcal{G}_i(x - \cdot) \big)_{kl}$. Let $K$ be some compact set in $B_0(1)$ and assume $x \in K$. It is easily seen by the definition of $\mathcal{G}_i$ as in \eqref{solfondDx} that
\[ \Vert \mathcal{L}_\xi \mathcal{G}_i(x - \cdot), \nu \otimes \cdot \Vert_{C^1(\partial B_0(R))} \le C_3 d \left(K, \partial B_0(1) \right)^{1-n}\]
for some positive constant $C_3$ independent of $x$. By the definition of $\mathcal{G}_{i,1}(x, \cdot)$ in \eqref{solfondNeumann} one therefore easily obtains that for $x \in K$ and $y \in B_0(1)$:
\begin{equation} \label{presquedersolNeumann}
  |x-y| |\nabla_y \mathcal{G}_{i,1}(x,y)| + |\mathcal{G}_{i,1}(x,y)| \le C(\delta) |x-y|^{2-n}  \hskip.1cm,
  \end{equation}
where $\delta$ is as in \eqref{rapportdomaine}. This gives \eqref{dersolNeumann} when the derivative is taken with respect to $y$. The same estimate when the derivative is taken with respect to $x$ is obtained differentiating \eqref{symGreen} and combining with \eqref{presquedersolNeumann}. Finally, \eqref{dersolNeumann} for any positive $R$ is obtained combining \eqref{presquedersolNeumann} with Claim \ref{claimrescaling}. 
\end{proof}

\bibliographystyle{amsplain}
\bibliography{biblio}

\end{document}